\documentclass{amsart}
\usepackage[alphabetic,initials]{amsrefs}
\usepackage{amssymb}
\usepackage[pdftex]{graphicx}
\usepackage[all]{xy}
\usepackage{color}
\usepackage{amsfonts}
\usepackage{sidecap}
\usepackage{enumerate}
\usepackage{eucal}

\newcommand{\R}{\mathbb{R}}
\newcommand{\cR}{\mathcal{R}}
\newcommand{\N}{\mathbb{N}}
\newcommand{\Z}{\mathbb{Z}}
\newcommand{\C}{\mathcal{C}}
\newcommand{\cN}{\mathcal{N}}

\newcommand{\cP}{\mathcal{P}}

\newcommand{\U}{\mathcal{U}}
\newcommand{\loc}{\mathrm{loc}}
\newcommand{\case}[1]{\subsubsection*{#1}} 
\newcommand{\step}[1]{\par\medskip\par\noindent\textit{#1}} 

\newcommand{\supp}{\operatorname{supp}}
\newcommand{\sgn}{\operatorname{sgn}}
\renewcommand{\Re}{\operatorname{Re}}

\def\Xint#1{\mathchoice
{\XXint\displaystyle\textstyle{#1}}
{\XXint\textstyle\scriptstyle{#1}}
{\XXint\scriptstyle\scriptscriptstyle{#1}}
{\XXint\scriptscriptstyle\scriptscriptstyle{#1}}
\!\int}
\def\XXint#1#2#3{{\setbox0=\hbox{$#1{#2#3}{\int}$}
\vcenter{\hbox{$#2#3$}}\kern-0.5\wd0}}

\def\dashint{\Xint-}

\renewcommand{\phi}{\varphi}
\renewcommand{\epsilon}{\varepsilon}

\theoremstyle{plain}
\newtheorem{theorem}{Theorem}[section]
\newtheorem{lemma}[theorem]{Lemma}
\newtheorem{proposition}[theorem]{Proposition}
\newtheorem{corollary}[theorem]{Corollary}
\theoremstyle{definition}

\theoremstyle{remark}
\newtheorem{remark}[theorem]{Remark}
\numberwithin{equation}{section}

\title[Free boundary in the Signorini problem]{Higher regularity of the free boundary in the elliptic Signorini problem}

\author{Herbert Koch}
\address{Mathematisches Institut, Universit\"{a}t Bonn, Endenicher Allee 64,  53115 Bonn,
Germany}
\email{koch@math.uni-bonn.de}
\author{Arshak Petrosyan}
\address{Department of Mathematics, Purdue University, West Lafayette,
  IN 47907}
\email{arshak@math.purdue.edu}
\author{Wenhui Shi}
\address{Mathematisches Institut, Universit\"{a}t Bonn, Endenicher Allee 62,  53115 Bonn, Germany}
\email{wenhui.shi@hcm.uni-bonn.de}
\thanks{A.P. is partially supported by NSF grant
  DMS-1101139}
\thanks{W.S. is partially supported by NSF grant
  DMS-1101139 and the Hausdorff Center for Mathematics} 
\begin{document}

\begin{abstract}
In this paper we study the higher regularity of the free boundary for
the elliptic Signorini problem. By using a partial hodograph-Legendre
transformation we show that the regular part of the
free boundary is real analytic. The first complication in the study is
the invertibility of the hodograph transform (which is only
$C^{0,1/2}$) which can be overcome by studying the precise asymptotic
behavior of
the solutions near regular free boundary points. The second and main
complication in the study is
that the equation satisfied by the Legendre transform is
degenerate. However, the equation has a subelliptic structure and can
be viewed as a perturbation of the Baouendi-Grushin operator. By
using the $L^p$ theory available for that operator, we can bootstrap
the regularity of the Legendre transform up to real analyticity, which
implies the real analyticity of the free boundary.

\end{abstract}

\maketitle

\section{Introduction}
\label{sec:introduction}
\subsection{Problem set-up and known results}
Let $B_R$ be the Euclidean ball in $\R^n$ ($n\geq 3$) centered at the
origin with radius $R>0$. Let $B_R^+=B_R\cap \{x_n>0\}$,
$B'_R=B_R\cap\{x_n=0\}$ and $(\partial B_R)^+=\partial
B_R\cap\{x_n>0\}$. Consider local minimizers of the Dirichlet functional
$$ J(v)=\int_{B_R^+}|\nabla v|^2dx$$
over the closed convex set
$$\mathcal{K}=\{v\in W^{1,2}(B_R^+):  v\geq 0 \text{ on } B'_R \},$$
i.e.\ functions $u\in \mathcal{K}$ which satisfy
$$J(u)\leq J(v), \text{ for any } v\in \mathcal{K} \text{ with } v-u=0\text{ on } (\partial B_R)^+.$$
This problem is known as the
\emph{(boundary) thin obstacle problem} or the \emph{(elliptic) Signorini
  problem}. 
It was shown in \cite{AC} that the local minimizers $u$ are of class
$C^{1,1/2}_{\text{loc}}(B_R^+\cup B'_R)$. Besides, $u$ will satisfy 
\begin{align}
\Delta u =0 &\quad \text{ in } B_R^+,\label{eq:u-harm}\\
u\geq 0,\quad -\partial _{x_n}u\geq 0,\quad u\cdot \partial _{x_n}u=0&\quad \text{ on } B'_R.\label{eq:compl-cond}
\end{align}
The boundary condition \eqref{eq:compl-cond} is known as the
\emph{complementarity \emph{or} Signorini boundary condition}. One of the main
features of the problem is that the following sets are apriori unknown:
\begin{alignat*}{2}
\Lambda_u &=\{x\in B'_R:u(x)=0\}&\quad&\emph{coincidence set}\\
\Omega_u& =\{x\in B'_R: u(x)>0\}&\quad&\emph{positivity set}\\
\Gamma_u &=\partial _{B'_R} \Omega_u,&\quad&\emph{free boundary}
\end{alignat*}
where by $\partial _{B'_R}$ we understand the boundary in the relative topology of
$B'_R$. 
The free boundary $\Gamma_u$ sometimes is said to be \emph{thin}, to
indicate that it is (expected to be) of codimension two. One of the
most interesting questions in this problem is the study of the structure and the regularity of
the free boundary $\Gamma_u$. 

To put our results in a proper perspective, below we give a brief
overview of some of the known results in the literature. The proofs can be found in
\cites{ACS,CSS,GP} and in Chapter~9 of \cite{AHN}.

We start by noting that we can extend solutions $u$ of the Signorini problem
\eqref{eq:u-harm}--\eqref{eq:compl-cond} to the entire ball $B_R$ in
two different ways: either by even symmetry in $x_n$ variable or by
odd symmetry. The even extension will be harmonic in $B_R\setminus
\Lambda_u$, while the odd extension will be so in $B_R\setminus
\overline{\Omega_u}$. In a sense, those two extensions can be viewed as two
different branches of a two-valued harmonic function. This gives a
heuristic explanation for the monotonicity of Almgren's \emph{frequency function}
$$
N_{x_0}(u,r)=\frac{r\int_{B_r^+(x_0)}|\nabla u|^2}{\int_{\partial
    B_r(x_0)^+} u^2},
$$
which goes back to Almgren's study of multi-valued harmonic functions
\cite{Alm}. In particular, the limiting value 
$$
\kappa(x_0)=N_{x_0}(0_+,u)
$$
for $x_0\in\Gamma_u$ turns out to be a very effective tool in
classifying free boundary points.
By using the monotonicity of the frequency $N_{x_0}$, it can be shown that the
rescalings 
\begin{equation}\label{eq:blowups}
\tilde u_{r,x_0}(x)=\dfrac{u(x_0+rx)}{\left(\dashint_{\partial
      B_r^+(x_0)}u^2\right)^{1/2}}. 
\end{equation}
converge, over subsequences $r=r_j\to 0_+$, to solutions $\tilde u_0$ of the
Signorini problem \eqref{eq:u-harm}--\eqref{eq:compl-cond} in
$\R^n_+$. Such limits are known as \emph{blowups} of $u$ at
$x_0$. Moreover, it can be shown that such blowups will be homogeneous
of degree $\kappa(x_0)$, regardless of the sequence $r_j\to 0_+$.
It is readily seen from the the definition that the mapping
$x_0\mapsto \kappa(x_0)$ is upper semicontinuous on
$\Gamma_u$. Furthermore, it can be shown that $\kappa(x_0)\geq 3/2$ for every
$x_0\in\Gamma_u$ and, more precisely, that the following alternative
holds:
$$
\kappa(x_0)=3/2\quad\text{or}\quad
\kappa(x_0)\geq 2.
$$ 
This brings us to the notion of a regular point.
A point $x_0\in \Gamma_u$ is called \emph{regular} if $\kappa(x_0)=3/2$. By
classifying all possible homogeneous solutions of homogeneity $3/2$, the
above definition is equivalent to saying that the blowups of $u$ at
$x_0$ have the form
\begin{equation*}
\tilde u_0(x)=c_n \Re(x_{n-1}+ix_n)^{3/2},
\end{equation*}
after a possible rotation of coordinate axes in $\R^{n-1}$. 

In what follows, we will
denote by $\cR_u$ the set of regular free boundary points, and call it the
\emph{regular set} of $u$:
$$
\cR_u=\{x_0\in\Gamma_u:\kappa(x_0)=3/2\}.
$$
The upper semicontinuity of $\kappa$, and the gap of values between
$3/2$ and $2$ implies that $\cR_u$ is a relatively open subset of
$\Gamma_u$. Besides, it is known that $\cR_u$ is 
locally a $C^{1,\alpha}$ regular $(n-2)$-dimensional
surface. 
In this paper, we are interested in the higher regularity of $\cR_u$. Since the codimension of the free boundary $\Gamma_u$ is
two, this question is meaningful only when $n\geq 3$. In fact, in dimension
$n=2$ the complete characterization of the coincidence set and the
free boundary was already found by Lewy \cite{Le}: $\Lambda_u$ is a
locally finite union of closed intervals.

\subsection{Notations} We will use fairly standard
notations in this paper.  By $\R^n$ we denote the $n$-dimensional
Euclidean space of points $x=(x_1,\ldots,x_n)$, $x_i\in\R$,
$i=1,\ldots, n$. For any $x\in\R^n$ we denote
$x'=(x_1,\ldots,x_{n-1})$ and $x''=(x_1,\ldots,x_{n-2})$. We also
identify $x'$ with $(x',0)$, thereby effectively embedding $\R^{n-1}$ into
$\R^n$. Similarly, we identify $x''$ with $(x'',0)$ and $(x'',0,0)$.

For $r>0$,
\begin{align*}
B_r(x_0)&=\{x\in \R^n: |x-x_0|<r\},\\
B_r^+(x_0)&=B_r(x_0)\cap \{x_n>0\},\\
B'_r(x_0)& =B_r(x_0)\cap \{x_n=0\},\\
B''_r(x_0)& =B'_r(x_0)\cap \{x_{n-1}=0\}.
\end{align*}
If $x_0$ is the origin, we will simply write $B_r$, $B_r^+$, $B'_r$ and $B''_r$.
Let $d(E,F):=\inf_{x\in E, y\in F}|x-y|$ be the Euclidean distance between two sets $E, F\subset \R^n$.

\subsection{Assumptions}
\label{sec:assumptions}
 In this paper we are interested in local properties of the solutions
and their free boundaries only near regular points and therefore,
without loss of generality, we make the following assumptions.

We will assume that $u$ solves the Signorini
problem \eqref{eq:u-harm}--\eqref{eq:compl-cond} in $B_2^+$ and that all free boundary points in $B_2'$ are regular, i.e.
\begin{equation}
\Gamma_u=\cR_u.
\end{equation}
Furthermore, we will assume that there exists $f\in
C^{1,\alpha}(\overline {B_2''})$ 
with
\begin{equation}\label{eq:f}
f(0)=|\nabla _{x''}f(0)|=0,
\end{equation}
such that
\begin{align}
\Gamma _u& =\{(x'',x_{n-1})\in B'_2: x_{n-1}=f(x'')\},\\
\Lambda _u&=\{(x'',x_{n-1})\in B'_2: x_{n-1}\leq f(x'')\}.
\end{align}
Next we assume $u\in C^{1,1/2}(B_2^+\cup B'_2)$ and that
\begin{align}
\partial_{x_{n-1}}u&>0\quad \text{ in } B_2^+\cup \Omega _u;\label{eq:non22}\\
\partial_{x_n}u&<0\quad \text{ in } B_2^+\cup (\Lambda_u\setminus\Gamma_u) ,
\quad \partial_{x_n} u=0\quad \text{on } \Omega_u\cup\Gamma_u.\label{eq:non33}
\end{align}
Moreover, we will also assume the following nondegeneracy property for
directional derivatives in a cone of tangential directions: for any $\eta>0$, there exist $r_{\eta}>0$ and $c_{\eta}>0$ such that
\begin{equation}\label{eq:non1}
\partial_\tau u(x)\geq c_{\eta}d(x,\Lambda_u)\quad\text{for $x\in
  \overline{B_{r_{\eta}}^+(x_0)}$ and
  $\tau\in C'_{\eta}(\nu'_{x_0})$, $|\tau|=1$,}
\end{equation}
for any $x_0\in \Gamma_u\cap B_1'$, where $\nu'_{x_0}$ is the unit normal in $\R^{n-1}$ to $\Gamma_u$ at $x_0$
outward to $\Lambda_u$ and
$$
C'_{\eta}(e_0)=\{x'\in \R^{n-1}: x'\cdot e_0\geq \eta |x'|\},
$$
for a unit vector $e_0\in\R^{n-1}$. 

We explicitly remark that if $u$ is a solution to the Signorini
problem, then the assumptions \eqref{eq:f}-\eqref{eq:non1} hold at any
regular free boundary point after a possible translation, rotation and
rescaling of $u$ (see e.g. \cite{CSS}, \cite{AHN}).

\subsection{Main results} Following the approach of Kinderlehrer and Nirenberg
\cite{KN} in the classical obstacle problem, we will use the partial
hodograph-Legendre transformation method to improve on the known
regularity of the free boundary.
The idea is to straighten the free boundary and then apply the
boundary regularity of the solution to the transformed elliptic
PDE\@. This works relatively simply for the classical obstacle problem,
and allows to prove $C^\infty$ regularity and even the real analyticity of
the free boundary.

In the Signorini problem, the free boundary $\Gamma_u$ is of
codimension two, and in order to straighten both $\Gamma_u$ and $\Lambda_u$
we need to make a partial hodograph transform in two variables. Namely, for $u$ satisfying the assumptions in
Section~\ref{sec:assumptions}, consider the transformation
\begin{equation}\label{eq:part-hod}
T:x\mapsto y=(x'', \partial_{x_{n-1}}u, \partial_{x_n}u),\quad x\in B_1^+\cup B_1'.
\end{equation}
Consider also the associated partial Legendre transform of $u$ given
by
\begin{equation}\label{eq:part-Leg}
v(y)=u(x)-x_{n-1}y_{n-1}-x_ny_n.
\end{equation}
Formally,
the inverse of $T$ is given by
$$
T^{-1}: y\mapsto (y'',-\partial_{y_{n-1}}v, -\partial_{y_n}v)
$$
and we can recover the free boundary in the following way:
$$
f(y'')=-\partial_{y_{n-1}}v(y'',0,0).
$$
However, we note that the mapping $T$ is only $C^{0,1/2}$ regular and even the
local invertibility of such mapping is rather nontrivial. Besides,
even if one has a local invertibility of $T$, the function $v$ will
satisfy a degenerate elliptic equation, and apriori it is not clear
if the equation will have enough structure to be useful.

Concerning the first complication, we make a careful asymptotic
analysis based the precise knowledge of the blowups, and this does
allow to establish the
local invertibility of $T$.

\begin{theorem}\label{thm:T-invert} Let $u$ be a solution of the Signorini problem
  \eqref{eq:u-harm}--\eqref{eq:compl-cond} in $B_2^+$, satisfying the
  assumptions in Section~\ref{sec:assumptions}. Then, there exists a
  small $\rho=\rho_u>0$  such that the partial hodograph
  transformation $T$ in \eqref{eq:part-hod} is injective in $B_\rho^+$. 
\end{theorem}

Then via an asymptotic analysis  of $v$ at the straightened free
boundary points, we observe that the fully nonlinear degenerate
elliptic equation for $v$ has a subelliptic structure, which can be
viewed as a perturbation of the Baouendi-Grushin operator. Then using
the $L^p$ theory for the Baouendi-Grushin operator and a bootstrapping
argument, we obtain the smoothness and even the real analyticity of $v$.

\begin{theorem}\label{thm:fb-regul} Let $u$ be as in
  Theorem~\ref{thm:T-invert} and $v$ be given by
  \eqref{eq:part-Leg}. Then exists
  $\delta=\delta_u>0$ such that the mapping $y''\mapsto
  \partial_{y_{n-2}}v(y'',0,0)$ is real analytic on $B''_{\delta}$. In
  particular, $\Gamma_u=\cR_u$ is locally an analytic surface.
\end{theorem}

\subsection{Related problems} 
The Signorini problem is just one example of a problem with
thin free boundaries. Many problems with thin free boundaries arise
when studying problems for the fractional Laplacian and using the
Caffarelli-Silvestre extension \cite{CS} to localize the
problem at the expense of adding an extra dimension (which makes the
free boundary ``thin''). Thus, the Signorini problem can be viewed as an obstacle
problem for half-Laplacian, see \cite{CSS}. We hope that the methods in
this paper can be used to study the higher regularity of the free
boundary in many such problems. 

A different approach to the study of the higher regularity of thin
free boundaries is being developed by De Silva and Savin
\cites{DS1,DS2,DS3}. In particular, in \cite{DS3} they
prove the $C^\infty$ regularity of $C^{2,\alpha}$ free boundaries in
the thin analogue of Alt-Caffarelli minimization problem \cite{CRS}. Their method is based on Schauder-type
estimates rather than hodograph-Legendre transform used in this
paper. 

At about the same time as we completed this work, De Silva and Savin \cite{DS5}
extended their approach to include the Signorini problem, as well as
lowered the initial regularity assumption on the free boundary to
$C^{1,\alpha}$. The main ingredient in their proof is an interesting
new higher order boundary Harnack principle, applicable to regular, as
well as slit domains (see also \cite{DS4}).

\subsection{Outline of the paper} The paper is organized as follows:

- In Section~\ref{sec:nond-prop} we study the so-called
$3/2$-homogeneous blowups of the solutions near regular points. This
is achieved by a combination a boundary Hopf-type principle in
domains with $C^{1,\alpha}$ slits, as well as Weiss- and Monneau-type
monotonicity formulas.

- In Section~\ref{sec:hodograph} we introduce the partial hodograph
transformation and show it is a homeomorphism in a neighborhood of
the regular free boundary points. This is achieved by using the
precise behavior of the solutions near regular free boundary points
established in Section~\ref{sec:nond-prop}.

- In Section~\ref{sec:legendre-funct-nonl} we consider the
corresponding Legendre transform $v$ and show some basic regularity of
$v$ inherited from $u$. 

- In Section~\ref{sec:smoothness} we
study the fully nonlinear PDE satisfied by $v$, which is the transformed PDE of
$u$. We show the linearization of the PDE is a perturbation of the
Baouendi-Grushin operator. Using the $L^p$ estimates available for
this operator, and a bootstrapping argument, we obtain the smoothness
of $v$, which in turn implies the smoothness of the free boundary.

- In Section~\ref{sec:analyticity}, we give more careful estimates on
the derivatives of $v$ which imply that $v$, and consequently
$\Gamma_u$, is real analytic.

\section{$3/2$-homogeneous blowups of solutions}
\label{sec:nond-prop}

\subsection{A stronger nondegeneracy property}
We start our study by establishing a stronger nondegeneracy property
for the tangential derivatives $\partial_\tau u$ than the one given in
\eqref{eq:non1}. Namely, we want to improve the lower bound in
\eqref{eq:non1} to a multiple of $d(x,\Lambda_u)^{1/2}$. To achieve
this, we construct a barrier function by using the
$C^{1,\alpha}$-regularity of $\Gamma_u$, to obtain a result that can
be viewed as a version of the boundary Hopf lemma for the domains of
the type $B_1\setminus \Lambda_u$.   

To proceed, we introduce some notations. For $\alpha\in (0,1)$, let
\begin{align*}
&\hat f_\alpha:[0,\infty)\rightarrow \R, \quad \hat f_\alpha(r)=r^{1+\alpha}\\
&\hat{\Lambda}_\alpha =B_1'\cap\{(x'', x_{n-1}):x_{n-1}\leq \hat
f_\alpha(|x''|)\},\quad  \hat D_\alpha = B_1\setminus \hat{\Lambda}_\alpha,
\end{align*}

\begin{lemma}\label{lem:hopf-barrier} There exists a
  continuous function $\hat h_\alpha$ on $B_1$ and a small $\rho=\rho(n,\alpha)>0$
  such that
\begin{alignat}{2}
 &\hat h_\alpha(0)=0,&\quad&\label{eq:hopf1}\\
 &\hat h_\alpha\leq  0&&\text{on } B'_\rho \cap \hat{\Lambda}_\alpha, \label{eq:hopf3}\\
 &\hat h_\alpha\leq 0&&\text{in }\cN_{\rho^4}(\hat \Lambda_\alpha)\cap\partial B_\rho, \quad \cN_{r}(E)=\{x:d(x,E)<r\}\label{eq:h3}\\
 &\Delta \hat h_\alpha\geq 0&&\text{in } \hat D_\alpha,\label{eq:h-subharm}
\end{alignat} 
and
\begin{equation}
\lim_{x_{n-1}\rightarrow
  0^+}\frac{\hat h_\alpha(0,x_{n-1},0)}{(x_{n-1})^{1/2}}=1.
\label{eq:hopf2}
\end{equation}
\end{lemma}
\begin{proof}
Inspired by the construction in \cite{LN}, we will show that the following
function satisfies the conditions of 
the lemma:
\begin{align*}
\hat h_\alpha (x)&=U(x)+\hat g_\alpha(U(x))-2\hat f_\alpha(|x|^{1/2}),
\intertext{where} 
U(x)&=\Re(x_{n-1}+ix_n)^{1/2}, \\
\hat g_\alpha(r)&=2\hat f_\alpha(r)+2(2n-3)r\int_0^r\frac{\hat f_\alpha(s)}{s^2}ds\\&=\left(2+\frac{2(2n-3)}{\alpha}\right)r^{1+\alpha}, \quad r\in [0,1).
\end{align*}
We next verify each of the properties \eqref{eq:hopf1}--\eqref{eq:hopf2}  in the lemma.

First of all, \eqref{eq:hopf1} is immediate.

Next, since
$U(x',0)=(x_{n-1}^+)^{1/2}$, we have 
\begin{equation}\label{eq:h2}
\hat h_\alpha(x',0)=(x_{n-1}^+)^{1/2}+\hat g_\alpha((x_{n-1}^+)^{1/2})-2\hat f_\alpha(|x'|^{1/2})\quad \text{on } B'_1.
\end{equation}
Therefore, on $B'_1\cap\hat \Lambda_\alpha$ one has
\begin{align*}
\hat h_\alpha(x',0)&\leq|x''|^{(\alpha+1)/2}+\left(2+\frac{2(2n-3)}{\alpha}\right)|x''|^{(\alpha +1)^2/2}-2|x''|^{(\alpha+1)/2}\\
& = -|x''|^{(\alpha+1)/2}\left(1-\left(2+\frac{2(2n-3)}{\alpha}\right)|x''|^{\alpha+1}\right).
\end{align*} 
Hence there exists $\rho=\rho(n,\alpha)>0$ such that $\hat h_\alpha\leq 0$ on
$B'_\rho\cap \hat\Lambda_\alpha$. This implies \eqref{eq:hopf3}. 

Further, to show \eqref{eq:h3}, notice that
$$
0\leq U\leq(\rho^4+\rho^{1+\alpha})^{1/2}\quad\text{in }\cN_{\rho^4}(\hat \Lambda_\alpha)\cap\partial B_\rho,
$$
which yields
$$
\hat h_\alpha\leq
(\rho^4+\rho^{1+\alpha})^{1/2}+\left(2+\frac{2(2n-3)}{\alpha}\right)(\rho^4+\rho^{1+\alpha})^{(1+\alpha)/2}-2\rho^{(1+\alpha)/2}\leq
0,
$$
for small $\rho$, as claimed.

Next, we show \eqref{eq:h-subharm}. It is easy to check that $\hat g_\alpha$ and $U$ satisfy
\begin{align}
\hat g''_\alpha(r)&=2\left(\hat f''_\alpha(r)+(2n-3)\frac{\hat f'_\alpha(r)}{r}\right); \label{eq:g1}\\
\Delta U&=0,\quad  \text{ in } B_1\setminus \{(x'',x_{n-1},0):x_{n-1}\leq 0\}\supseteq \hat D_\alpha;\label{eq:u1}\\
|\nabla U|^2&=\frac{1}{4}(x_{n-1}^2+x_n^2)^{-1/2}, \quad \text{ in } \hat D_\alpha\notag
\end{align}
Thus a direct computation shows that for $x\in \hat D_\alpha$
\begin{align}
\Delta \hat g_\alpha(U)&=\hat g''_\alpha(U)|\nabla U|^2+\hat g'_\alpha(U)\Delta U=\frac{1}{4}\hat g''_\alpha(U)(x_{n-1}^2+x_n^2)^{-1/2};\label{eq:g2}\\
\Delta \hat f_\alpha(|x|^{1/2})&=\frac{1}{4}\hat f''_\alpha(|x|^{1/2})|x|^{-1}+\frac{2n-3}{4}\hat f'_\alpha(|x|^{1/2})|x|^{-3/2},\notag\\
&=\frac{1}{4}\left(\hat f''_\alpha(|x|^{1/2})+(2n-3)\frac{\hat f'_\alpha(|x|^{1/2})}{|x|^{1/2}}\right)|x|^{-1},\notag\\
&\leq \frac{1}{4}\left(\hat f''_\alpha(|x|^{1/2})+(2n-3)\frac{\hat f'_\alpha(|x|^{1/2})}{|x|^{1/2}}\right)(x_{n-1}^2+x_n^2)^{-1/2}.\label{eq:f1}
\end{align}
Combining \eqref{eq:g1}--\eqref{eq:f1}, we obtain that for $x\in \hat D_\alpha$
\begin{multline}\label{eq:h1}
(x_{n-1}^2+x_n^2)^{1/2}\Delta \hat h_\alpha(x)= (x_{n-1}^2+x_n^2)^{1/2}[\Delta U + \Delta \hat g_\alpha(U)-2\Delta \hat f_\alpha(|x|^{1/2})]
\\\geq\frac{1}{2}\left(\hat f''_\alpha(U)+(2n-3)\frac{\hat f'_\alpha(U)}{U}\right)-\frac{1}{2}\left(\hat f''_\alpha(|x|^{1/2})+(2n-3)\frac{\hat f'_\alpha(|x|^{1/2})}{|x|^{1/2}}\right).
\end{multline}
Since $U(x)\leq |x|^{1/2}$  and
$\hat f''_\alpha(r)+(2n-3)\hat f'_\alpha(r)/r=(1+\alpha)(2n-3+\alpha)r^{\alpha-1}$ is
decreasing on $(0,\infty)$, then by \eqref{eq:h1} we have $\Delta
\hat h_\alpha\geq 0$ in $\hat D_\alpha$.  This shows \eqref{eq:h-subharm}.

Finally, by \eqref{eq:h2}, we have
$$
\lim_{x_{n-1}\rightarrow 0^+}
\frac{\hat h_\alpha(0,x_{n-1},0)}{(x_{n-1})^{1/2}}=1.
$$
This completes the proof of the lemma.
\end{proof}

Using the function constructed in Lemma~\ref{lem:hopf-barrier} as a
barrier, we have the improvement of the nondegeneracy for
nonnegative harmonic functions in $\hat D_\alpha$.

\begin{lemma}\label{cor:hopf1}
Let $w$ be a nonnegative superharmonic function in $\hat D_\alpha$, $w\in
C^{0,1/2}(B_1)$ and $w=0$ on $\hat \Lambda_\alpha$. Moreover, suppose that
$w$ satisfies 
\begin{equation}\label{eq:non2}
w(x)\geq c_0 d(x,\hat \Lambda_\alpha),\quad x\in B_1.
\end{equation}
Then there exists $\rho=\rho(n,\alpha,c_0)>0$ and $\epsilon_0=\epsilon_0(n,\alpha, c_0)>0$ such that
\begin{equation}\label{eq:non3}
w(0,x_{n-1},0)\geq \epsilon _0 (x_{n-1})^{1/2}, \quad 0<x_{n-1}<\rho.
\end{equation}
\end{lemma}
\begin{proof} 
Let $h$ and $\rho$ be as in Lemma~\ref{lem:hopf-barrier}. Then by
\eqref{eq:non2} and the continuity of $h$, there exists $\epsilon_0=\epsilon_0(\rho, c_0)>0$ such that 
\begin{equation*}
w(x)-\epsilon_0\hat h_\alpha(x)\geq 0, \quad x\in\partial B_\rho\setminus \cN_{\rho^4}(\hat \Lambda_\alpha),
\end{equation*}
which combined with \eqref{eq:hopf3}--\eqref{eq:h3} gives that 
$$
w-\epsilon_0\hat h_\alpha\geq 0\quad\text{on }\partial(\hat D_\alpha\cap B_\rho).
$$
Then, by the maximum principle
$$
w-\epsilon_0\hat h_\alpha\geq 0\quad\text{on }\hat D_\alpha\cap B_\rho.
$$
In particular,
\begin{align*}
w(0,x_{n-1},0)&\geq \epsilon_0 (x_{n-1})^{1/2}+\left(2+\frac{2(2n-3)}{\alpha}\right)(x_{n-1})^{(1+\alpha)/2}-2(x_{n-1})^{(1+\alpha)/2}\\
&\geq\epsilon_0 (x_{n-1})^{1/2},
\end{align*}
for $0<x_{n-1}<\rho$ with $\rho=\rho(n,\alpha)>0$ small.
\end{proof}

From Lemma~\ref{cor:hopf1}, we obtain the following nondegeneracy
property for the tangential derivatives of the solution to the
elliptic Signorini problem. 

\begin{proposition}\label{prop:hopf}
Let $u$ be a solution to the elliptic Signorini problem in $B_2^+$
satisfying the assumptions in Section~\ref{sec:assumptions}. Then for each $x_0\in \Gamma_u\cap B_1'$ and $\eta>0$, there exist
$\rho,\ \epsilon_0>0$ depending on $n, \alpha, c_\eta, \|f\|_{C^{1,\alpha}(B''_2)}$, such that 
\begin{equation}\label{eq:sighopf}
\partial _\tau u(x_0+t\nu'_{x_0})\geq \epsilon_0 t^{1/2}, \quad t\in (0,\rho),\quad \tau\in C'_{\eta}(\nu'_{x_0}).
\end{equation}
\end{proposition}
\begin{proof}
For each $x_0\in \Gamma_u\cap B_1'$, we can rotate a coordinate system
in $\R^{n-1}$ so that $\nu'_{x_0}=e_{n-1}$. Then we have
\begin{equation*}
B_1(x_0)\setminus \tilde\Lambda_{x_0,M}\subset B_1(x_0)\setminus \Lambda_u,
\end{equation*}
where 
\begin{align*}
\tilde\Lambda_{x_0,M}=\{(x'',x_{n-1})\in B'_1(x_0):
x_{n-1}-(x_0)_{n-1}\leq M |x''-(x_0)''|^{1+\alpha}\}&\\\quad\text{with }
M=\|f\|_{C^{1,\alpha}(B''_2)}.&
\end{align*}
Since $u$ is harmonic in $B_1(x_0)\setminus \tilde{\Lambda}_{x_0,M}$ and
$\partial _\tau u$ satisfies the nondegeneracy condition
\eqref{eq:non1}, then $\partial_\tau u$ satisfies the assumptions of
Lemma~\ref{cor:hopf1}, with a small difference that there is constant
$M$ in the definition of the set $\tilde \Lambda_{x_0,M}$ above. However, by a
simple scaling, we can make $M=1$. Thus, there exists $\rho=\rho(n,\alpha, c_\eta,
M)>0$ and $\epsilon_0=\epsilon_0(n,\alpha, c_\eta, M)>0$ such that
\eqref{eq:sighopf} holds.  
\end{proof}

\subsection{$3/2$-homogeneous blowups}
\label{sec:blowups} 
For our further study, we need to consider the following rescalings
\begin{equation}\label{eq:homgen-rescal}
u_{r,x_0}(x)=\frac{u(x_0+rx)}{r^{3/2}},\quad r>0
\end{equation}
with $x_0\in \Gamma_u$. Note that these are different form rescalings
\eqref{eq:blowups} in the sense that the $L^2$ norm of $u$ is not preserved under the rescaling, but it is better suited for the study of regular
points. First, from the growth estimate (see e.g. \cite{AC})
\begin{equation}\label{eq:growth-est}
\sup _{B_r^+(x_0)} |u| \leq C_ur^{3/2}
\end{equation}
for $x_0\in \Gamma_u\cap B_1'$, where $C_u=C(n,\|u\|_{L^2(B_2^+)})$ and $0<r<1$, we know that the family
$\{u_{r,x_0}\}_r$ is locally uniformly bounded. Moreover, by the interior $C^{1,1/2}$ estimate (see e.g. \cite{AC})
\begin{equation}\label{eq:r}
\|u\|_{C^{1,1/2}(\overline{B_r^+(x_0)})}\leq C_u, \quad 0<r<1
\end{equation}
we get that $\{u_{r,x_0}\}_r$ is uniformly bounded in $C^{1,1/2}_{\loc}(B_{1/r}^+\cup
B'_{1/r})$. Thus there
exists $u_{0}\in C^{1,1/2}(\R^n_+\cup\{x_n=0\}) $ such that
$u_{r,x_0}\rightarrow u_{0}$ in $C^{1,\alpha}_{\loc}$, for any $\alpha \in
(0,1/2)$ over a certain subsequence $r=r_j\rightarrow 0$. It is also
immediate to see that $u_{0}$ is a global solution of the
Signorini problem, i.e., a solution of \eqref{eq:u-harm}--\eqref{eq:compl-cond} in $\R^n_+$. Furthermore, it is important
to note that $u_{0}$ is nonzero, because of the nondegeneracy provided
by Proposition~\ref{prop:hopf}. Sometimes we will refer to the function
$u_0$ as the \emph{$3/2$-homogeneous blowup} to indicate the way it
was obtained.

The following Weiss-type monotonicity formula, whose proof can be found
in \cite{GP}, implies that $u_{0}$ is a homogeneous global solution
of the Signorini problem of degree $3/2$.  

\begin{lemma}[Weiss-type monotonicity formula]\label{lem:weiss}
\pushQED{\qed}
Let $u$ be a nonzero solution of the Signorini problem \eqref{eq:u-harm}--\eqref{eq:compl-cond} in $B^+_R$. For any $x_0\in\Gamma_u$ and $0<r<R-|x_0|$
define
$$
W_{x_0}(r,u)=\frac{1}{r^{n+1}}\int_{B_r^+(x_0)}|\nabla
u|^2-\frac{3/2}{r^{n+2}}\int_{(\partial B_r(x_0))^+} u^2.
$$
Then $r\mapsto W_{x_0}(r,u)$
is nondecreasing in $r\in (0,R-|x_0|)$. Moreover, for a.e.\
$r\in(0,R-|x_0|)$ we have
$$
\frac{d}{dr}W_{x_0}(r,u)=\frac{2}{r^{n+3}}\int_{(\partial B_r(x_0))^+}
((x-x_0)\cdot\nabla u-(3/2)u)^2.
$$
Furthermore, $W_{x_0}(r,u)\equiv const$  for $r_1<r<r_2$ if and only if
$u$ is homogeneous of degree $3/2$ with respect to $x_0$ in
$B_{r_2}(x_0)\setminus \overline {B_{r_1}(x_0)}$, i.e.
\[
(x-x_0)\cdot \nabla u-(3/2)u=0\quad \text{in } B_{r_2}(x_0)\setminus \overline {B_{r_1}(x_0)}.\qedhere
\]
\popQED
\end{lemma}
\begin{remark} \label{rem:weiss}
The Weiss-type monotonicity formula above is specifically
  adjusted to work with rescalings \eqref{eq:homgen-rescal}. Namely,
  by a simple change of variables, one can show that
$$
W_0(\rho,u_{r,x_0})=W_{x_0}(r\rho,u).
$$
Besides, by the definition of regular points, we also have that 
$$
W_{x_0}(0_+,u)=(N_{x_0}(0_+,u)-3/2)\frac{1}{r^{n+2}}\int_{(\partial B_r(x_0))^+} u^2=0,\quad \text{if }x_0\in\cR_u,
$$
since $N_{x_0}(0_+,u)=\kappa(x_0)=3/2$.

\end{remark}

\begin{proposition}[Unique type of $3/2$-homogeneous blowup]\label{prop:unique}
Let $u$ be a solution of the Signorini problem
\eqref{eq:u-harm}--\eqref{eq:compl-cond} in $B_2^+$, satisfying the assumptions
in Section~\ref{sec:assumptions}. Then there exist two positive
constants $c_u$ and $C_u$, depending only on $u$ such that if $x_0\in
\Gamma_u\cap B_1'$ and that 
$$
u_{r,x_0}(x)=\frac{u(x_0+rx)}{r^{3/2}}\rightarrow u_0(x)\quad\text{in }
C^{1,\alpha}_{\loc}(\R^{n}_+\cup \R^{n-1})
$$
over a sequence $r=r_j\to 0_+$, then 
$$
u_0(x)=C_0\Re (x'\cdot\nu'_{x_0}+ix_n)^{3/2}  
$$
with a constant $C_0$ satisfying
$$
c_u<C_0<C_u.
$$
\end{proposition}

\begin{proof} We have already noticed at the beginning of
  Section~\ref{sec:blowups} that $u_0$ is a nonzero global solution of the
  Signorini problem.
Besides, by the Weiss-type 
monotonicity formula, we will have
$$
W_0(\rho,u_0)=\lim _{r_j\rightarrow 0_+} W_0(\rho,u_{x_0,r_j})=\lim
_{r_j\rightarrow 0_+}W_{x_0}(\rho r_j,u)=W_{x_0}(0_+,u)=0,
$$
for any $\rho>0$. Hence, by Lemma~\ref{lem:weiss}, $u_0$ is a
homogeneous of degree $3/2$ in $\R^{n}_+\cup\R^{n-1}$. Then, by
Proposition~9.9 in \cite{AHN}, we must have the form
$$
u_{0}=C_0\Re (x'\cdot e'+ix_n)^{3/2}
$$
for some $C_0>0$ and a tangential unit vector $e'\in \partial
B'_1$. We claim that $e'=\nu'_{x_0}$. Indeed, we have
$$
\partial _\tau u_{r_j,x_0}\geq 0\quad \text{in $\overline{B_{R}^+}$,
 for any }\tau\in
C'_{\eta}(\nu'_{x_0}),
$$
provided $j\geq j_{R,h,x_0}$ so that $r_\eta(x_0)/r_j\geq R$. 
Passing to the limit $r_j\rightarrow 0$, we therefore have
\begin{equation*}
\partial _\tau u_{0} \geq 0 \quad \text{ in $\R^n_+\cup \R^{n-1}$ for any }\tau\in
C'_{\eta}(\nu'_{x_0}). 
\end{equation*}
Hence, for any $\eta>0$
\begin{equation}\label{eq:U0}
\tau\cdot e'\geq 0,\quad\text{whenever }\tau \in C'_{\eta}(\nu'_{x_0}). 
\end{equation}
This is possible only if $e'=\nu'_{x_0}$. Thus, we have the claimed
representation 
$$
u_0(x)=C_0\Re (x'\cdot\nu'_{x_0}+ix_n)^{3/2}.
$$
The estimates on constant $C_0$ now follow from the growth estimate
\eqref{eq:growth-est} and the nondegeneracy
Proposition~\ref{prop:hopf}, which are preserved under the
$3/2$-homogeneous blowup. 
\end{proof}

\begin{remark}\label{rem:unique}
This proposition also holds for the rescaling family with varying centers:
$$u_{r_j,x_j}(x)=\frac{u(x_j+r_jx)}{r_j^{3/2}},$$
where $x_j\in \Gamma_u\cap B_{1/2}'$, $x_j\rightarrow x_0\in \Gamma_u$ and $r_j\in (0,1/2)$.

Indeed, by Lemma~\ref{lem:weiss} we have
$$
x\mapsto W_x(r,u)\searrow 0\quad \text{pointwise on } \Gamma_u\cap \overline{B_{1/2}'} \text{ as } r\searrow 0_+.$$
Hence applying Dini's theorem form the classical analysis to the family of monotone continuous
functions $\{x\mapsto W_x(r,u)\}_{0<r<1/2}$ on the compact set
$\Gamma_u\cap \overline{B_{1/2}'}$, we have that the above convergence is
uniform on $\Gamma_u\cap \overline{B_{1/2}'}$. Hence passing to the
limit $j\rightarrow \infty$, we obtain
$$
W_0(\rho,u_{0})=\lim _{r_j\rightarrow 0_+} W_0(\rho,u_{x_j,r_j})=\lim
_{r_j\rightarrow 0_+}W_{x_j}(\rho r_j,u)=0,
$$
for any $\rho>0$. Arguing as in Proposition~\ref{prop:unique}, we
conclude that
$u_{0}(x)=C_0\Re (x'\cdot\nu'_{x_0}+ix_n)^{3/2}$.
\end{remark}

In order to get the uniqueness of the blowup limit, we need to show that the constant $C_0$ in Proposition~\ref{prop:unique} does not depend on the subsequence $r_j$ but only depends on $x_0$. 
This is a consequence of the following Monneau-type monotonicity
formula \cite{GP}. Without apriori knowledge on the  free boundary, this
formula is known to hold only at so-called singular points, i.e.,
$x_0\in \Gamma_u$ with $\kappa(x_0)=2m$, $m\in\N_+$. However, using the
$C^{1,\alpha}$ regularity of the free boundary, we will be able to
establish this result also at regular points.

\begin{lemma}[Monneau-type monotonicity formula]\label{lem:monneau}
Let $u$ be a solution of the Signorini problem
\eqref{eq:u-harm}--\eqref{eq:compl-cond} in $B_2^+$, satisfying the assumptions
in Section~\ref{sec:assumptions}. For any $x_0\in \Gamma_u\cap B'_1$,
$0<r<1$, and a positive constant  $c_0$, we define
$$M_{x_0}(r,u, c_0)=\frac{1}{r^{n+2}}\int_{(\partial B_r(x_0))^+}(u-c_0u_{x_0})^2,$$
where 
$$u_{x_0}=\Re(x'\cdot \nu'_{x_0}+ix_n)^{3/2}.$$
Then there exists a constant $\tilde C$ which depends on the
$C^{1,\alpha}$ norm of $f$, $C_u$ in \eqref{eq:r}, $c_0$, and
$\alpha$, such that 
$$
r\mapsto M_{x_0}(r,u,c_0)+\tilde C r^\alpha
$$
is monotone nondecreasing for $r\in (0,1)$.
\end{lemma}
\begin{proof}
For simplicity we assume $x_0=0$ and write $u_0=c_0\Re(x_{n-1}+ix_n)^{3/2}$, $M_0(r,u)=M_0(r,u,c_0)$. 

Letting $w=u-u_0$ and using the scaling properties of $M_0$, we have 
\begin{equation}\label{eq:M0}
\frac{d}{dr}M_{0}(r,u)=\frac{2}{r^{n+3}}\int_{(\partial B_r)^+} w\left(x\cdot \nabla w-\frac{3}{2}w\right).
\end{equation}
Next, we compute the Weiss energy functional
\begin{align*}
W_0(r,u)=\frac{1}{r^{n+1}}\int_{B_r^+}|\nabla(w+u_0)|^2 -\frac{3/2}{r^{n+2}}\int_{(\partial B_r)^+} (w+u_0)^2.
\end{align*}
By Remark~\ref{rem:weiss},
\begin{equation*}
W_0(0_+,u)=W_0(r,u_0)=0,
\end{equation*}
hence
\begin{equation}\label{eq:M1}
W_0(r,u)=\frac{1}{r^{n+1}}\int_{B_r^+}\left(|\nabla w|^2+2\nabla w\cdot\nabla u_0\right) -\frac{3/2}{r^{n+2}}\int_{(\partial B_r)^+}\left(w^2+2wu_0\right).
\end{equation}
An integration by parts gives
\begin{equation*}
\int_{B_r^+}\nabla w\cdot \nabla u_0=\int_{(\partial B_r)^+} w \partial_\nu u_0 + \int_{B'_r} w (-\partial_{x_n} u_0) -\int_{B_r^+} w \Delta u_0
\end{equation*}
Noticing that 
\begin{equation*}
\partial_\nu u_0=\frac{x\cdot \nabla u_0 }{r} \text{ on } (\partial B_r)^+\quad\text{and}\quad  \Delta u_0=0 \text{ in } B_r^+,
\end{equation*}
we have
\begin{equation}\label{eq:M2}
\frac{2}{r^{n+1}}\int_{B_r^+}\nabla w\cdot \nabla u_0=\frac{2}{r^{n+2}}\int_{(\partial B_r)^+} w(x\cdot \nabla u_0) +\frac{2}{r^{n+1}}\int_{B'_r} w (-\partial_{x_n} u_0).
\end{equation}
Similarly, integrating by parts and using $\Delta w=\Delta u-\Delta u_0=0$ in $B_r^+$, we obtain
\begin{equation}\label{eq:M3}
\frac{1}{r^{n+1}}\int_{B_r^+}|\nabla w|^2 =\frac{1}{r^{n+2}}\int_{(\partial B_r)^+} w(x\cdot \nabla w)+\frac{1}{r^{n+1}}\int_{B'_r} w(-\partial_{x_n}w).
\end{equation}
Combining \eqref{eq:M1}--\eqref{eq:M3}, we have 
\begin{multline*}
W_0(r,u)=\frac{1}{r^{n+2}}\int_{(\partial B_r)^+}w\left(x\cdot \nabla w-\frac{3}{2} w\right)+\frac{2}{r^{n+2}}\int_{(\partial B_r)^+} w\left( x\cdot u_0-\frac{3}{2} u_0\right)\\
 +\frac{1}{r^{n+1}}\int_{B'_r}w(-\partial_{x_n} w)+2w(-\partial_{x_n} u_0).
\end{multline*}
Since $u_0$ is homogeneous of degree $3/2$, the second integral above is zero. Moreover, using $u(\partial_{x_n}u)=u_0(\partial_{x_n}u_0)=0$ on $B'_r$, we have the third integral is equal to 
$$\frac{1}{r^{n+1}}\int_{B'_r}u(-\partial_{x_n}u_0)+u_0(\partial_{x_n}u).$$
Recalling \eqref{eq:M0}, we have
\begin{align*}
\frac{d}{dr}M_0(r,u)=\frac{2W_0(r,u)}{r}-\frac{2}{r^{n+2}}\int_{B'_r}\left(u(-\partial_{x_n}u_0)+u_0\partial_{x_n}u\right).
\end{align*}
Noticing that $W_0(r,u)>0$ for $0<r<1$ by Lemma~\ref{lem:weiss} and
$\partial_{x_n}u\leq 0$, $u_0\geq 0$ on $B'_1$, we have
\begin{equation*}
\frac{d}{dr}M_0(r,u)\geq \frac{2}{r^{n+2}}\int_{B'_r} u(\partial_{x_n}u_0).
\end{equation*}
Using the growth estimate \eqref{eq:r} for $u$ and explicit expression for $u_0$, we have
$$\frac{d}{dr}M_0(r,u)\geq - \frac{2}{r^{n+2}}\cdot C_uc_0r^{2} H^{n-1}\left(B'_r\cap \{u>0, x_{n-1}<0\}\right)$$
Since the free boundary $\Gamma_u=\partial_{B'_2}\{u>0\}$ is a
$C^{1,\alpha}$ graph 
and $0\in \Gamma_u$, we have
$$H^{n-1}(B'_r\cap\{u>0,x_{n-1}<0\})\leq c_1 r^{n-1+\alpha},$$
where $c_1$ is a constant depending on $\|f\|_{C^{1,\alpha}}$.
Hence
$$\frac{d}{dr}M_0(r,u)\geq -2c_1C_uc_0 r^{-1+\alpha}, $$
which implies the claim of the lemma with $\tilde C= 2c_1C_uc_0/\alpha$.
\end{proof}

We can now establish the main result of this section.

\begin{theorem}[Uniqueness of $3/2$-homogeneous blowup]\label{prop:uni}
Let $u$ be a solution of the Signorini problem in $B_2^+$, satisfying the assumptions in Section~\ref{sec:assumptions}. Then for $x_0\in \Gamma_u\cap B'_1$, there exists a constant $C_{x_0}>0$ such that 
$$\lim_{r\rightarrow 0_+}u_{r,x_0}(x)= C_{x_0}u_{x_0}(x) \text{ in } C^{1,\alpha}_{\loc}(\R^n_+\cup \R^{n-1}),$$
where
$$u_{x_0}(x)=\Re (x'\cdot \nu'_{x_0}+ix_n)^{3/2}.$$
Moreover, the function $x_0\mapsto C_{x_0}$ is continuous on $\Gamma_u\cap B'_1$. Furthermore, for any $\alpha\in(0,1/2)$ and $R>0$,
\begin{equation}\label{eq:uniformity0}
\sup_{x_0\in \Gamma_u\cap B_{1/2}'}\|u_{r,x_0}-C_{x_0}u_{x_0}\|_{C^{1,\alpha} (B_{R}^+\cup B'_R)}\rightarrow 0\quad \text{ as }r\rightarrow 0_+.
\end{equation}
\end{theorem}
\begin{proof}
We first show $C_{x_0}$ does not depend on the converging sequences. 

Given $x_0\in \Gamma_u\cap B_1'$, let $u_{r_j,x_0}$ be a converging sequence such that $u_{r_j,x_0}\rightarrow C_{x_0}u_{x_0}$ in $C^{1,\alpha}_{\loc}$ for some $C_{x_0}$ satisfying $c_u<C_{x_0}<C_u$. 
By Lemma~\ref{lem:monneau}, the mapping $r\mapsto M_{x_0}(r,u,C_{x_0})+\tilde Cr^\alpha$ is nonnegative, monotone nondecreasing on $(0,1)$, hence 
\begin{align*}
\lim_{r\rightarrow 0_+} M_{x_0}(r,u,C_{x_0})+\tilde Cr^\alpha&=\lim_{j\rightarrow \infty}M_{x_0}(r_j,u,C_{x_0})+\tilde Cr_j^\alpha\\
&=\lim_{j\rightarrow\infty} M_{x_0}(1,u_{r_j,x_0},C_{x_0})+\tilde Cr_j^\alpha =0.
\end{align*}
This implies $C_{x_0}$ does not depend on the converging sequences $u_{r_j,x_0}$.

Next we show that $C_{x_0}$ depends continuously on $x_0\in \Gamma_u\cap
B_1'$. Fix $x_0\in \Gamma_u\cap B_1'$. For any $\epsilon>0$, let
$r_\epsilon\in (0,1/2)$ be such that
$$M_{x_0}(r_\epsilon,u, C_{x_0})+\tilde Cr_\epsilon^\alpha=\int_{(\partial B_1)^+} (u_{r_\epsilon,x_0}-C_{x_0}u_{x_0})^2+\tilde Cr_\epsilon^\alpha<\epsilon.$$
For $r_\epsilon$ fixed, by the continuity of $u$ we have $x\mapsto \int_{(\partial B_1)^+} u_{r_\epsilon, x}$ is continuous on $\Gamma_u\cap B_1'$. Moreover, from the explicit formulation of $u_{x}$ as well as the $C^{0,\alpha}$ continuity of $x\mapsto \nu'_{x}$, the function $x\mapsto \int_{(\partial B_1)^+} u_x$ is continuous. 
Therefore, there exists a positive $\delta_\epsilon$ small enough, such that for all $x_1\in \Gamma_u$ with $|x_1-x_0|<\delta_\epsilon$, 
\begin{equation}\label{eq:uu1}
M_{x_1}(r_\epsilon, u, C_{x_0})=\int_{(\partial B_1)^+} (u_{r_\epsilon, x_1}-C_{x_0}u_{x_1})^2<2\epsilon.
\end{equation}
For $0<r<r_\epsilon$, by Lemma~\ref{lem:monneau} 
\begin{equation}\label{eq:Mu1}
M_{x_1}(r, u,C_{x_0})\leq M_{x_1}(r_\epsilon, u, C_{x_0})+\tilde Cr_\epsilon^\alpha<3\epsilon.
\end{equation}
Let $r\rightarrow 0_+$ in \eqref{eq:Mu1}, then
\begin{align}
\lim_{r\rightarrow 0_+} M_{x_1}(r,u,C_{x_0})&=\lim_{r\rightarrow 0_+}\int_{(\partial B_1)^+} (u_{r,x_1}-C_{x_0}u_{x_1})^2\label{eq:cont-c}\\
&=(C_{x_1}-C_{x_0})^2\int_{(\partial B_1)^+} (u_{x_1})^2 <3\epsilon.\notag
\end{align}
By the explicit expression of $u_{x}$, there is a constant $c_n>0$ such that $\int_{(\partial B_1)^+} (u_{x})^2\geq c_n>0$ for any $x\in \Gamma_u\cap B_1'$. This together with \eqref{eq:cont-c} gives
$$|C_{x_1}- C_{x_0}|< \sqrt{\frac{3\epsilon}{c_n}}, \quad \text{ for }x_1\in B_{\delta_\epsilon}'(x_0)\cap \Gamma_u.$$
This shows the  continuity of   $x_0\mapsto C_{x_0}$.

Finally, we show \eqref{eq:uniformity}. In fact, $x_0\mapsto M_{x_0}(r,u,C_{x_0})+\tilde Cr^\alpha$ is continuous on the compact set $\Gamma_u\cap \overline{B_{1/2}'}$ and it monotonically decreases to zero as $r$ decreases to zero for each $x_0$. Hence by Dini's theorem, $M_{x_0}(r,u,C_{x_0})+\tilde Cr^\alpha\rightarrow 0$ as $r\rightarrow 0_+$ uniformly on $\Gamma_u\cap \overline{B_{1/2}'}$. Thus one has
$$\sup_{x_0\in \Gamma_u\cap B_{1/2}'}\int_{(\partial B_1)^+}(u_{r,x_0}-C_{x_0}u_{x_0})^2\rightarrow 0, \quad r\rightarrow 0_+.$$

It is not hard to see that $(u_{r,x_0}-C_{x_0}u_{x_0})^\pm$ are subharmonic in $B_1$ (after the even reflection of $u$ about $x_n$). Hence by the $L^\infty-L^2$ estimates for the subharmonic functions we have \begin{equation}\label{eq:uniformity}
\sup_{x_0\in \Gamma_u\cap B_{1/2}'}\|u_{r,x_0}-C_{x_0}u_{x_0}\|_{L^\infty (B_{1/2}^+)}\rightarrow 0\quad \text{ as }r\rightarrow 0_+.
\end{equation}
By an interpolation of H\"older spaces, i.e.\ for some absolute constant $C>0$, 
$\|g\|_{C^{1,\alpha}}\leq C \|g\|^{\lambda}_{L^\infty}\|g\|_{C^{1,1/2}}^{1-\lambda}$ for $\alpha\in (0,1/2)$ and $\lambda=\lambda(\alpha)\in (0,1)$,
as well as a rescaling argument we obtain \eqref{eq:uniformity0}.
\end{proof}

The continuous dependence on $x_0$ of $C_{x_0}$ gives the uniqueness of the blowups with varying centers. 
\begin{corollary}\label{cor:varicenter}
Let $u$ be a solution of the Signorini problem
\eqref{eq:u-harm}--\eqref{eq:compl-cond} in $B_2^+$, satisfying the assumptions
in Section~\ref{sec:assumptions}. Let $x_j, x_0\in \Gamma_u\cap B_{1/2}'$, such that $x_j\rightarrow x_0$ as $j\rightarrow \infty$. Let $r_j\rightarrow 0$ as $j\rightarrow \infty$. Then for any $\alpha\in (0,1/2)$ and $u_{x_0}$ defined in Theorem~\ref{prop:uni},
$$u_{r_j,x_j}\rightarrow 	C_{x_0}u_{x_0}\text{ in } C^{1,\alpha}_{\loc}(\R^n_+\cup \R^{n-1}).$$
\end{corollary}
\begin{proof}
This follows from the uniform continuity of $x_0\mapsto \int_{(\partial B_1)^+} C_{x_0}u_{x_0}$ and Theorem~\ref{prop:uni}.
\end{proof}

\section{Partial hodograph transform}\label{sec:hodograph}
 Let $u$ be a solution of the
Signorini problem in $B_2^+$ satisfying the assumptions in
Section~\ref{sec:assumptions}. Following the idea in the classical obstacle problem \cite{KN}, we
would like to use the method of partial hodograph-Legendre
transforms to study the higher regularity of the free boundary in the
Signorini problem. Since $\Gamma_u$ has codimension two, the most
natural hodograph transformation to consider is the one with
respect to variable $x_{n-1}$ and $x_n$: 
\begin{align}
& T=T^u: B^+_1\cup B'_1\rightarrow \R^n,\label{eq:def-hod-1}\\
& y=T(x)=(x'',\partial_{x_{n-1}}u, \partial_{x_n}u).\label{eq:def-hod-2}
\end{align}
(The reader can easily check that if we do the partial hodograph
transform in $x_{n-1}$ variable only, it will still straighten
$\Gamma_u$, however, the image of $B^+$ will not have a flat boundary
and this will render this transformation rather useless.)

By doing so, we hope that there exists a small neighborhood $B_\rho$
of the origin, such that $T$ is one-to-one on $B_\rho ^+\cup
B'_\rho$. However, due to the $C^{1,1/2}$ regularity of $u$,
the mapping $T$ is only $C^{0,1/2}$ near the origin. Hence the simple inverse
function theorem (which typically requires the transformation $T$ to
be from class $C^1$) cannot be applied here. Instead,  we will make
use of the blowup profiles, which contain enough information to catch
the behavior of the solution near the free boundary points. 

Before stating the main results, we make several observations.

\begin{enumerate}[(i)]
\item By the assumptions in Section~\ref{sec:assumptions}, $\partial _{x_{n-1}}u>0$ in $B_1^+\cup \Omega_u$ and $\partial _{x_n}u \leq 0$ on $B'_1$. This together with the complementary boundary condition gives us
\begin{align*}
& T(\Lambda_u\cap B'_1)\subset \{y\in \R^n: y_{n-1}=0, y_n\leq 0\};\\
& T(\Omega_u\cap B'_1)\subset \{y\in \R^n: y_{n-1}> 0, y_n=0\};\\
& T(\Gamma_u\cap B'_1)\subset \{y\in \R^n: y_{n-1}=y_n=0\};\\
& T(B_1^+)\subset\{y\in \R^n: y_{n-1}>0, y_n<0\}.
\end{align*}

\definecolor{lightblue}{rgb}{0.945,0.895,0.95}%

\definecolor{lightpink}{rgb}{0.95,0.77,0.762}%

\begin{figure}[t]
\centering
\begin{picture}(300,163)(25,0)
\put(0,0){\includegraphics[height=163pt]{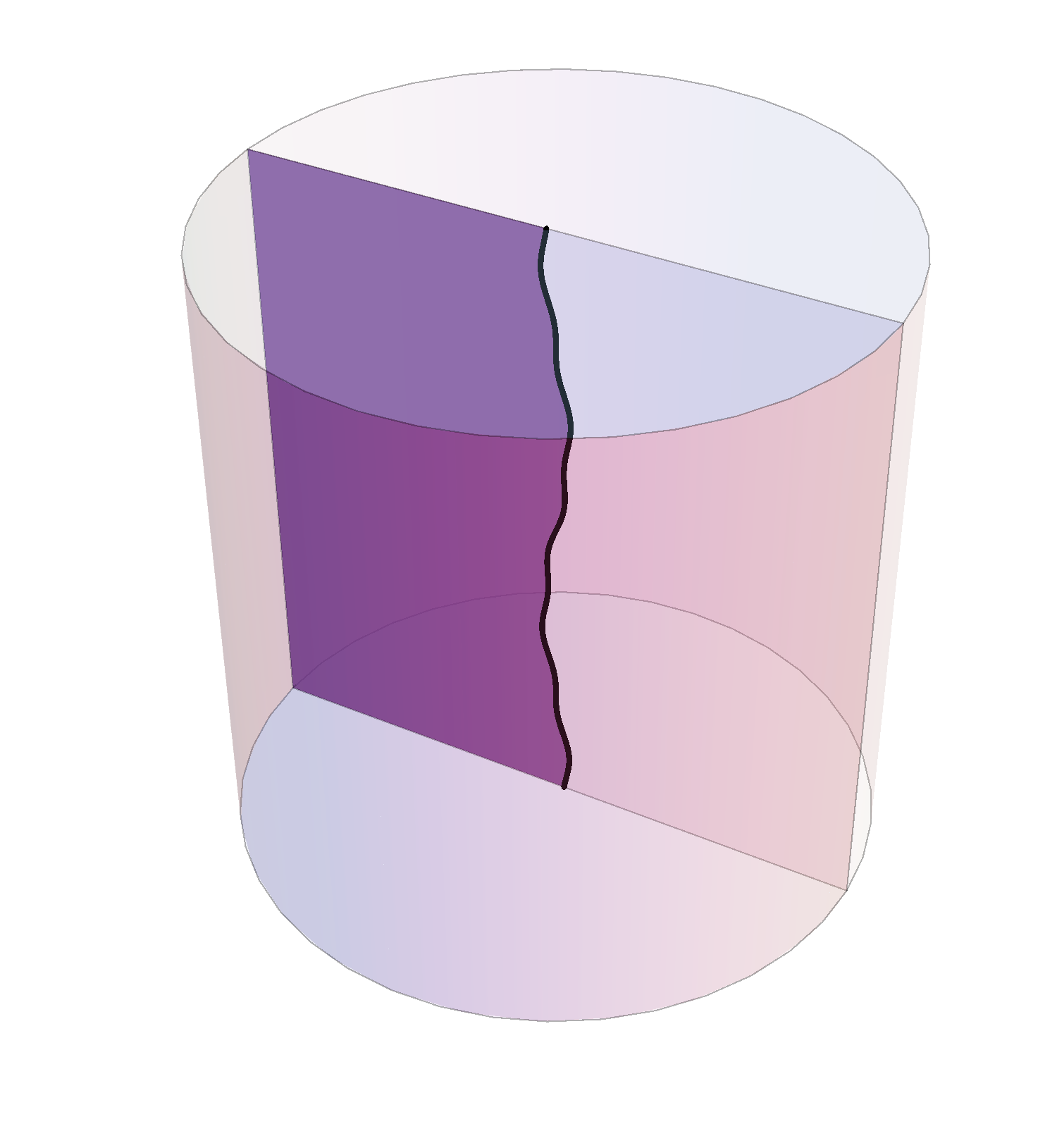}}
\put(55,90){\footnotesize \color{lightpink}{$\Lambda_u$}}
\put(90,85){\footnotesize $\Omega_u$}
\put(188,0){\includegraphics[height=163pt]{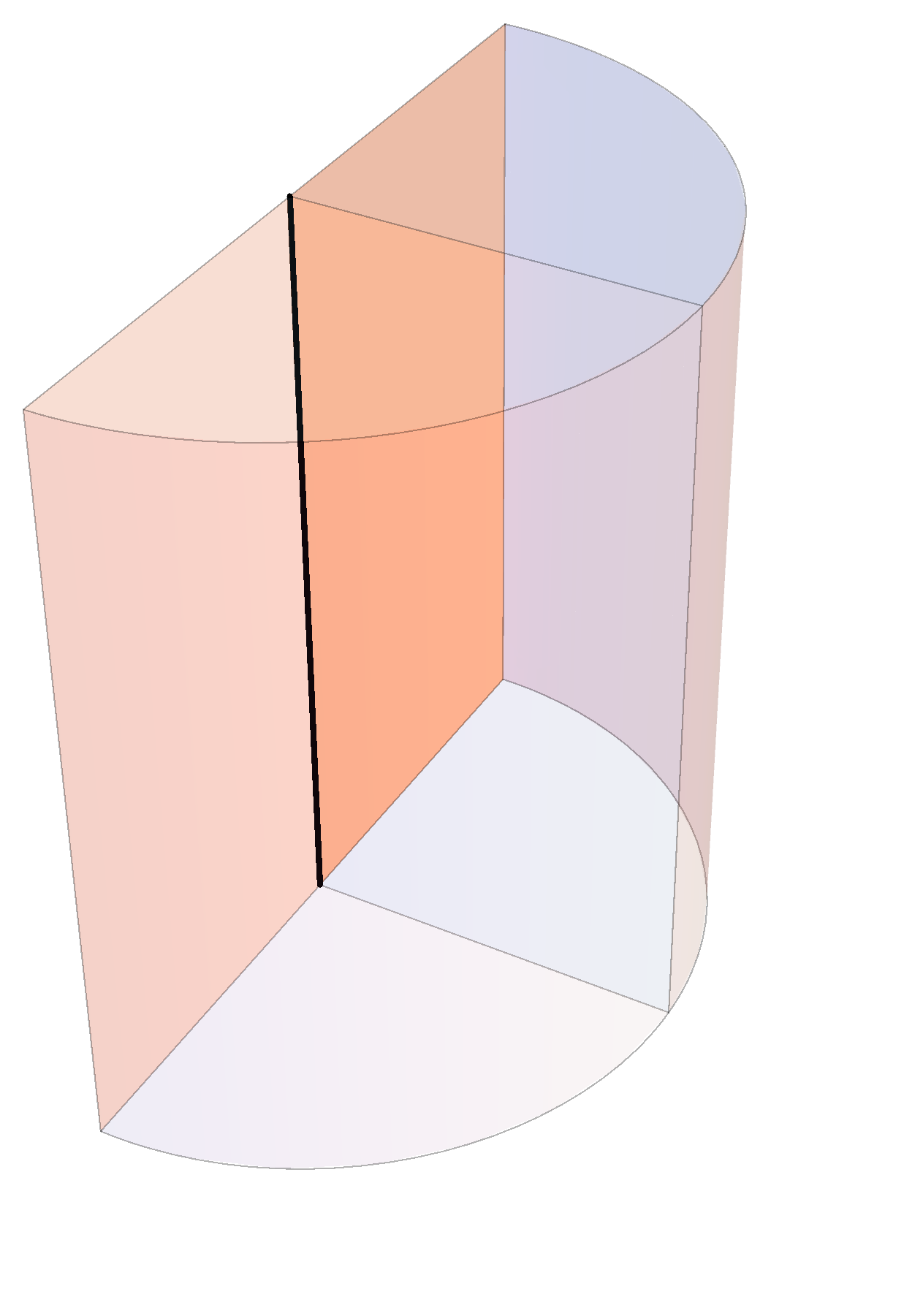}}
\put(226,120){\footnotesize $T(\Lambda_u)$}
\put(240,85){\footnotesize $T(\Omega_u)$}
\put(130,120){$\xymatrix{\ar@/^/[rr]^{T}_{\substack{y''=x''\\y_{n-1}=u_{x_{n-1}}\\y_n=u_{x_n}}} & &}$}
\end{picture}
\caption{The partial hodograph transform $T$. Shown for even
  extension of $u$ in $x_n$ variable.}
\end{figure}
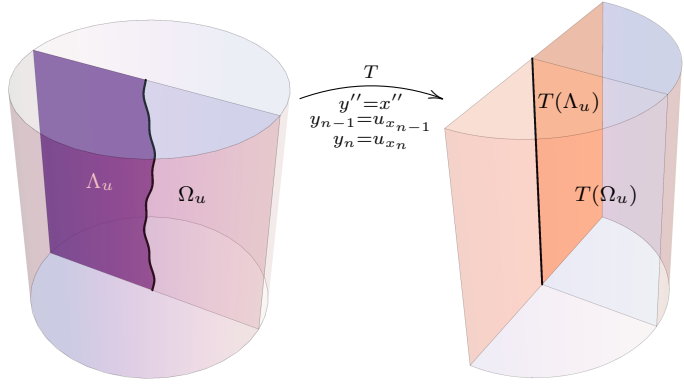

\item If we extend $u$ across $\{x_n=0\}$ to $B_1$ by the even symmetry or odd symmetry in $x_n$, then the resulting function will satisfy
\begin{align*}
&  \Delta u =0 \text{ in } B_1\setminus \Lambda_u,\quad \text{for even
  extension in }x_n;\\
&  \Delta u=0 \text{ in } B_1\setminus \overline{\Omega_u},\quad \text{for odd extension in }x_n.
\end{align*} 
Hence $u$ is analytic in $(B_1^+\cup B'_1)\setminus \Gamma_u$. From \eqref{eq:def-hod-2}, $T$ is also analytic in $(B_1^+\cup B'_1)\setminus \Gamma_u$.

\item To
  better understand the nature of $T$, consider the solution
  $u_0(x)=\Re(x_1+ix_2)^{3/2}$ of the Signorini problem and
  find an explicit formula for $T^{u_0}$. A simple computation shows that using
  complex notations the mapping $T^{u_0}$ is given by
$$
x_1+ix_2\mapsto \frac32(x_1-i x_2)^{1/2},
$$ 
where by the latter we understand the appropriate branch. Loosely
speaking, this tells that $T$ behaves like $\sqrt{z}$ function in the
last two variables.

The observation above also suggests us to compose $T$ with the mapping
\begin{align*}
\psi: \R^n&\rightarrow \R^n;\\
 (y'',y_{n-1},y_n)&\mapsto (y'',
y_{n-1}^2-y_n^2, -2y_{n-1}y_n),
\end{align*}
which can be expressed, by using complex notations (denote $\psi(y)$ by $w$), as
$$
 w''=y'', \quad w_{n-1}+iw_n=(y_{n-1}-i y_n)^2.
$$
More explicitly, alongside $T$, we consider the transformation $T_1=T_1^u=\psi\circ T$
\begin{align}
& T_1: B_1^+\cup B'_1\rightarrow \R^n,\label{eq:define_T1}\\
& w=T_1(x)=(x_1,\ldots, x_{n-2}, (\partial _{x_{n-1}}u)^2-(\partial _{x_n}u)^2, -2(\partial _{x_{n-1}}u)(\partial _{x_n}u)).\notag
\end{align}
Now $T_1$ maps $B_1^+$ to the upper half space $\{w_n> 0\}$ and $B'_1$ to the hyperplane $\{w_n=0\}$. Moreover, $T_1$ straightens the free boundary $\Gamma_u$ as $T$ and satisfies
\begin{align*}
& T_1(\Lambda_u)\subset\{w\in \R^n: w_{n-1}\leq 0, w_n=0\}=\Lambda_0\\
& T_1(\Omega_u)\subset \{w\in \R^n: w_{n-1}> 0, w_n=0\}=\Omega_0,\\
& T_1(\Gamma_u)\subset \{w\in \R^n: w_{n-1}= 0, w_n=0\}=\Gamma_0.
\end{align*} 

\item The advantage of $T_1$ is that by
  arguing as in (ii) above and making odd and even extensions of $u$ in
  $x_n$, we can extend it to a mapping on $B_1$. Moreover, this
  extension will be the same in both cases (unlike for the mapping
  $T$). In particular, this makes $T_1$ a single-valued mapping on
  $B_1$, which is real analytic in $B_1\setminus \Gamma_u$.
\end{enumerate}

 In what follows, we will prove the injectivity of $T_1$ in a
 neighborhood of the origin. For that we will need the make the
 following direct computations.

\definecolor{lightblue}{rgb}{0.945,0.895,0.95}
\definecolor{lightpink}{rgb}{0.95,0.77,0.762}

\begin{figure}[t]
\centering
\begin{picture}(338,163)(25,0)
\put(0,0){\includegraphics[height=163pt]{DDf.png}}
\put(55,90){\footnotesize \color{lightpink}{$\Lambda_u$}}
\put(90,85){\footnotesize $\Omega_u$}
\put(188,0){\includegraphics[height=163pt]{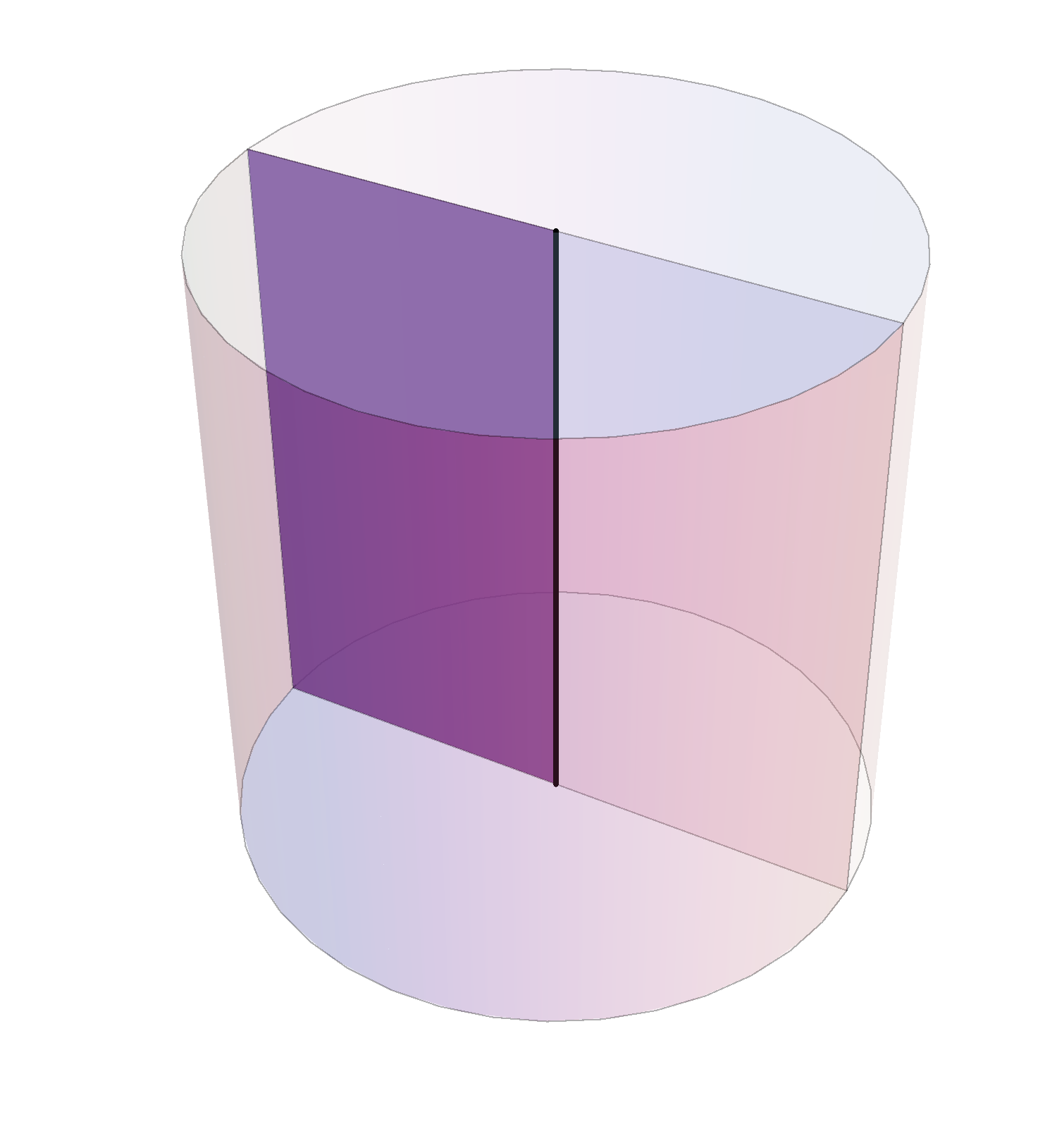}}
\put(243,90){\footnotesize \color{lightpink}{$\Lambda_0$}}
\put(278,85){\footnotesize $\Omega_0$}
\put(138,120){$\xymatrix{\ar@/^/[rr]^{T_1}_{\substack{w''=x''\\w_{n-1}=u_{x_{n-1}}^2-u_{x_n}^2\\w_n=-2u_{x_{n-1}}
      u_{x_n}}} & &}$}
\end{picture}
\caption{The modified partial hodograph transform $T_1=\psi\circ T$}
\end{figure}

\begin{proposition}\label{prop:computeu0}
Let $\hat u_0:\R^2\rightarrow \R$ be given by 
\begin{align*}
\hat u_0(x_1,x_2)=\Re(x_1+ix_2)^{3/2}=r^{3/2}\cos\left(3\theta/2\right),
\end{align*}
for $x_1+ix_2=r e^{i\theta}$, $\theta\in(-\pi,\pi]$. Then we have the
following identities
\begin{align*}
(\partial _{x_1} \hat u_0, \partial _{x_2} \hat u_0)&=(3r^{1/2}/2)(\cos (\theta/2), -\sin (\theta/2))\\
D^2\hat u_0(x_1,x_2)&=(3r^{-1/2}/4)\left(\begin{array}{cc}
\cos(\theta/2)& \sin(\theta/2)\\
\sin(\theta/2)& -\cos(\theta/2)\end{array}\right)\\
T_1^{\hat u_0}(x_1,x_2)&=(9/4)(x_1,x_2).
\end{align*}
\end{proposition}
\begin{proof} Direct computation.
\end{proof}

From now on, we will use a slightly different notation from
Theorem~\ref{prop:uni} to denote the blowup limit, i.e.\ for $x_0\in
\Gamma_u$, we let $$u_{x_0}(x)=C_{x_0}\Re(x'\cdot
\nu'_{x_0}+ix_n)^{3/2}.$$ The following proposition is a consequence
of Theorem~\ref{prop:uni} and the $C^{1}$ regularity of $f$.

\begin{proposition}\label{prop:uniformity2}
For any $\epsilon>0$, there exists $\delta>0$ depending on $\epsilon$ and $u$, such that for all $0<r<\delta$ and $x_0\in \Gamma_u\cap B_\delta' $, we have:
\begin{enumerate}[\em (i)]
\item $ \|u_{r,x_0}-u_0\|_{C^1(B_2^+\cup B'_2)}<\epsilon$;\\
\item $u_{r,x_0}$ has a harmonic extension at any $x\in B'_2$ with $|x_{n-1}|>1/4$ (by making even or odd reflection about $x_n$). Hence if we let
$$A^{1/2,1}_+:=\{(x'',x_{n-1},x_n):|x''|\leq 1,\ 1/2\leq \sqrt{x_{n-1}^2+x_n^2}\leq 1,\  x_n\geq 0\},$$
then for any multi-index $\alpha\in \Z_+^n$ with $|\alpha|=k\geq 2$ we have
$$ \|\partial^{\alpha} u_{r,x_0}-\partial^\alpha u_0\|_{L^\infty
  (A^{1/2,1}_+)}< C(n,k)\epsilon.$$
\end{enumerate}
\end{proposition}
\begin{proof}
(i) Given any $\epsilon>0$, by \eqref{eq:uniformity0} there is a positive constant $\delta_1$ depending on $u$ such that for any $0<r<\delta_1$,
$$\sup_{x_0\in \Gamma_u\cap B_{1/2}'}\|u_{r,x_0}-u_{x_0}\|_{C^1(B_2^+\cup B'_2)}<\epsilon/2.$$ On the other hand, there exists $\delta_2>0$ depending on $\epsilon$, modulus of continuity of $C_{x_0}$ and $C^{1,\alpha}$ norm of $f$ such that  $$\sup_{x_0\in \Gamma_u\cap B_{\delta_2}'}\|u_{x_0}-u_0\|_{C^1(B_2^+\cup B'_2)}< \epsilon/2.$$ Taking $\delta=\min\{\delta_1,\delta_2,1/2\}$ we proved (i).

(ii) We observe that for small enough $r$ and $|x_0|$, the rescaled
free boundary $$\Gamma_{u_{r,x_0}}=\{(x'', x_{n-1},0):
x_{n-1}=f_{r,x_0}(x'')\}, \quad
f_{r,x_0}(x''):=\frac{f(x''_0+rx'')-f(x''_0)}{r}$$ satisfies
$|f_{r,x_0}(x'')|<1/4$ when $|x''|<2$.  This follows immediately from
the assumption $f(0)=|\nabla'' f(0)|=0$.
The rest of (ii) then follows from the estimates for the higher order derivatives of harmonic functions.
\end{proof}

The next proposition is a consequence of Proposition~\ref{prop:computeu0} and Proposition~\ref{prop:uniformity2}, which is useful to understand the transformation $T$. Since $T$ fixes the first $n-2$ coordinates, it is more convenient to work on a ``tubular''  neighborhood of $\Gamma_u$ defined as follows: First we consider the projection map 
$$p:B_1\rightarrow \Gamma_u, \quad x\mapsto (x'', f(x''),0)$$
Since $f\in C^{1,\alpha}(B''_2)$, $p$ is continuous in $B_1$. Moreover, it is easy to verify that for some constant $c_1=c_1(n,\|f\|_{C^{1}})\geq 1$,
\begin{equation}\label{eq:compare}
d(x,\Gamma_u)\leq |x-p(x)|\leq c_1 d(x,\Gamma_u). 
\end{equation}
Next for $\delta\in (0,1/2)$, we let 
$$\mathcal{O}_\delta:=\{x: p(x)\in B_\delta'\cap \Gamma_u,\ |x-p(x)|<\delta\};$$
$$\mathcal{O}_\delta^+=\mathcal{O}_\delta \cap \{x_n>0\};\quad \mathcal{O}'_\delta=\mathcal{O}_\delta\cap \{x_n=0\}.$$
By the continuity of $p$ and \eqref{eq:compare}, $\mathcal{O}_\delta$ is a tubular neighborhood of the part of the free boundary $\Gamma_u$ lying in $B_\delta$.
\begin{proposition}\label{prop:square}
Given any $\epsilon>0$, let $\delta=\delta_\epsilon$ be the same constant as in Proposition~\ref{prop:uniformity2}. Then for any $x\in (\mathcal{O}^+_\delta\cup \mathcal{O}'_\delta)\setminus \Gamma_u$, we have
\begin{align*}
\textup{(i)}&\quad \left|\frac{\left| T(x)-T(p(x))\right|^2}{|x-p(x)|}-\frac{9C_0^2}{4}\right|
<2C_u\epsilon ;\\
\textup{(ii)}
&\quad  \left| |x-p(x)|\det(DT)(x)-\left(-\frac{9C_0^2}{16}\right)\right|<2C_0n\epsilon.
\end{align*}
\end{proposition}
\begin{proof}
(i) First we observe that
\begin{equation}\label{eq:ob1}
|T(x)-T(p(x))|^2=(\partial_{x_{n-1}}u)^2(x)+(\partial_{x_n}u)^2(x).
\end{equation}
For $x\in (\mathcal{O}^+_\delta\cup \mathcal{O}'_\delta)\setminus \Gamma_u$, let 
$$d=|x-p(x)|\in (0,\delta), \quad \xi =\frac{x-p(x)}{d}\in \partial B_1.$$
Applying Proposition~\ref{prop:uniformity2}(i) to $u_{d,p(x)}$ we have, 
\begin{equation}\label{eq:temp1}
|\nabla u_{d,p(x)}(\xi)- \nabla u_0(\xi)|<\epsilon.
\end{equation}
On the other hand, by the computation in Proposition~\ref{prop:computeu0}, 
\begin{equation}\label{eq:temp2}
(\partial_{x_{n-1}} u_0)^2(\zeta)+(\partial_{x_n}u_0)^2(\zeta)=\frac{9C_0^2}{4}, \quad \text{for any }\zeta\in \partial B_1.
\end{equation}
\eqref{eq:temp1} together with \eqref{eq:temp2} implies
$$\left|(\partial_{x_{n-1}}u_{d,p(x)})^2(\xi)+(\partial_{x_n}u_{d,p(x)})^2(\xi)- \frac{9C_0^2}{4}\right|\leq 2C_u \epsilon.$$ Rescaling back to $u$ and using \eqref{eq:ob1}, we obtain (i).

(ii) The proof of (ii) is similar to that of (i).  In $(\mathcal{O}^+_\delta\cup \mathcal{O}'_\delta)\setminus \Gamma_u$, $T$ is a smooth mapping, and
$$\det(DT)(x)=\partial_{x_{n-1}x_{n-1}}u(x)\partial_{x_nx_n}u(x)-(\partial_{x_{n-1}x_n}u)^2(x).$$
Given $x\in  (\mathcal{O}^+_\delta\cup \mathcal{O}'_\delta)\setminus \Gamma_u$,  consider $u_{d,p(x)}$ and rescaled point $\xi\in \partial B_1$ as above. By Proposition~\ref{prop:uniformity2}(ii), 
$$|D^2 u_{d,p(x)}(\xi)-D^2u_0(\xi)|\leq n\epsilon.$$
This together with the explicit expression for $D^2u_0$ in Proposition~\ref{prop:computeu0} gives 
\begin{equation}\label{eq:temp3}
\left|\det DT^{u_{d,p(x)}}(\xi)-\left(-\frac{9C_0^2}{16}\right)\right|<2C_0n\epsilon.
\end{equation}
It is easy to check the following rescaling property
$$\det (DT^{u_{r,x_0}})(\xi)=r \det (DT )(x_0+r\xi).$$
This combined with \eqref{eq:temp3} gives (ii).
\end{proof}

Now we are ready to prove the main theorem of the section. Let $$\mathcal{Q}=\{y\in \R^n: y_{n-1}\geq 0, y_n\leq 0\}.$$

\begin{theorem}\label{thm:injective}
There exists a small constant $\delta=\delta_u>0$, such that $T$ is a homeomorphism from $\mathcal{O}^+_\delta\cup \mathcal{O}'_\delta$ to $T(\mathcal{O}^+_\delta\cup \mathcal{O}'_\delta)$, which is relatively open in $\mathcal{Q}$. Moreover, if we extend $T$ to $B_1^-$ via the even (odd) reflection of $u$ about $x_n$, then it is a diffeomorphism from $\mathcal{O}_\delta\setminus \Lambda_u$ ($\mathcal{O}_\delta\setminus \overline{\Omega_u}$) onto an open subset in $\R^n$.
\end{theorem}

\begin{proof} By observation (iii) and (iv), instead of $T$, we consider the map 
$$T_1=\psi\circ T: B_1\rightarrow \R^n, $$
which is first defined in $B_1^+\cup B'_1$ as \eqref{eq:define_T1} and
then extended to $B_1$ via even or odd reflection of $u$ in $x_n$ variable. We will first show $T_1$ is a homeomorphism from $\mathcal{O}_\delta$ to $T_1(\mathcal{O}_\delta)$ for sufficiently small $\delta$. This is divided into three steps.

\textsl{Step 1.} There exists $\delta=\delta_u>0$, such that for any $0<r<\delta$ and $x_0\in \Gamma_u\cap B_\delta'$, $T_1 ^{u_{r,x_0}}$ is injective in $A^{1/2,1}$. Here 
$$A^{1/2,1}:=\{(x'',x_{n-1},x_n): |x''|\leq 1,\ 1/2\leq \sqrt{x_{n-1}^2+x_n^2}\leq 1\}.$$

In fact, by Proposition~\ref{prop:uniformity2} and the definition of $T_1$, for any $\epsilon>0$, there exists $\delta=\delta_{\epsilon,u}>0$ such that
$$\|T_1^{u_{r,x_0}}-T_1^{u_0}\|_{C^{1}(A^{1/2,1})}<\epsilon, \quad 0<r<\delta,\ x_0\in \Gamma_u\cap B_\delta'.$$
From Proposition~\ref{prop:computeu0}, 
$$T_1^{u_0}(x)=\left(x'', \frac{9C_0}{4}x_{n-1}, \frac{9C_0}{4}x_n\right)$$
is a nondegenerate linear map. Hence if we take $\epsilon=\epsilon(C_0)$ sufficiently small,  $T_1^{u_{r,x_0}}$ will be injective on the compact set $A^{1/2,1}$.

\textsl{Step 2.} $T_1$ is injective in $\mathcal{O}_\delta$ for $\delta$ chosen as above. 

It is clear that $T_1$ is injective on $\Gamma_u$, since it maps
$(x'',f(x''),0)$ to $(x'',0,0)$. Thus, it will be sufficient if we show the injectivity of $T_1$ in $\mathcal{O}_\delta\setminus \Gamma_u$. 

Indeed, suppose there exist $x_1, x_2\in \mathcal{O}_\delta\setminus \Gamma_u$ such that $T_1(x_1)=T_1(x_2)$. Necessarily $p(x_1)=p(x_2)$. Let $x_0=p(x_i)$, $d_i=|x_i-p(x_i)|$, $i=1,2$ and w.l.o.g  assume $d_1\geq d_2$. A simple rescaling gives $$T^{u_{d_1,x_0}}_1(\xi_1)=T^{u_{d_1,x_0}}_1(\xi_2), \quad \xi_i=\frac{x_i-x_0}{d_1}\in \overline{B_1}, \ i=1,2.$$
On the other hand, $T_1(x_1)=T_1(x_2)$ implies $$(\partial_{x_{n-1}}u)^2(x_1)+(\partial_{x_n}u)^2(x_1)=(\partial_{x_{n-1}}u)^2(x_2)+(\partial_{x_n}u)^2(x_2).$$ Then recalling \eqref{eq:ob1}, by Proposition~\ref{prop:square}(i) for small enough $\epsilon$ (by choosing $\delta$ small in step 1) one has
\begin{equation}\label{eq:ratio}
\frac{1}{2}\leq \frac{d_{2}}{d_{1}}\leq 1.
\end{equation}
Note that $d_2/d_1=|x_2-x_0|/d_1=|\xi_2|$ and $\xi_1, \xi_2$ have first $n-2$ coordinates equal to zero, thus $\xi_1, \xi_2\in A^{1/2,1}$. But this contradicts the injectivity of $T_1^{u_{d_1,x_0}}$ on $A^{1/2,1}$ obtained in step 1.

\textsl{Step 3.} $T_1$ is a homeomorphism from $\mathcal{O}_\delta$ to $T_1(\mathcal{O}_\delta)$. 

Now $T_1:\mathcal{O}_\delta\rightarrow \R^n$ is an injective map. Moreover, it is continuous because of the continuity of $Du$. Thus by the Brouwer invariance of domain theorem $T_1(\mathcal{O}_\delta)$ is open and $T_1$ is a homeomorphism between $\mathcal{O}_\delta$ and $T_1(\mathcal{O}_\delta)$. 

Now we proceed to show $T:\mathcal{O}^+_\delta\cup \mathcal{O}'_\delta\rightarrow T(\mathcal{O}^+_\delta\cup \mathcal{O}'_\delta)$ is a homeomorphism for $\delta$ chosen as above. Indeed, recalling $T_1=\psi\circ T$ we obtain the injectivity of $T$ from the injectivity of $T_1$. Next we note that $\psi$ is injective in $\mathcal{Q}$, which contains $T(B_1^+\cup B'_1)$ by observation (i), thus $T(U)=\psi^{-1}(T_1(U))\cap \mathcal{Q}$ for any subset $U\subset B_1^+\cup B'_1$. Thus, $T$ is an open map from $\mathcal{O}_\delta^+\cup \mathcal{O}'_\delta$ to $\mathcal{Q}$, since $T_1$ is open and $\psi$ is continuous. Combining with the continuity of $T$, we obtain that $T$ is an homeomorphism from $\mathcal{O}_\delta^+\cup \mathcal{O}'_\delta$ to  $T(\mathcal{O}_\delta^+\cup \mathcal{O}'_\delta)\subset \mathcal{Q}$.

Finally, we notice that $T$ is smooth on $\mathcal{O}_\delta\setminus \Lambda_u$ (or $\mathcal{O}_\delta\setminus \overline{\Omega}_u$) after even (or odd) extension of $u$ about $x_n$. Moreover, by Proposition~\ref{prop:square}(ii), $\det(DT)$ is nonvanishing there for sufficiently small $\delta$. Hence $T$ is a diffeomorphism by the implicit function theorem.
\end{proof}

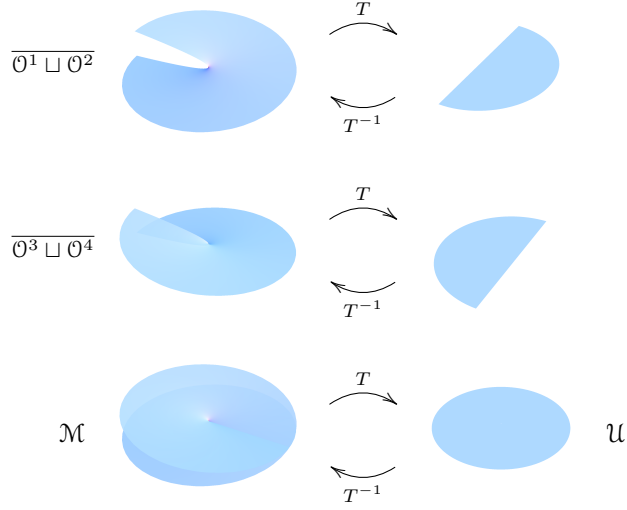
\begin{figure}[t]
  \centering
\begin{picture}(200,200)
  \put(0,134){\includegraphics[height=66pt]{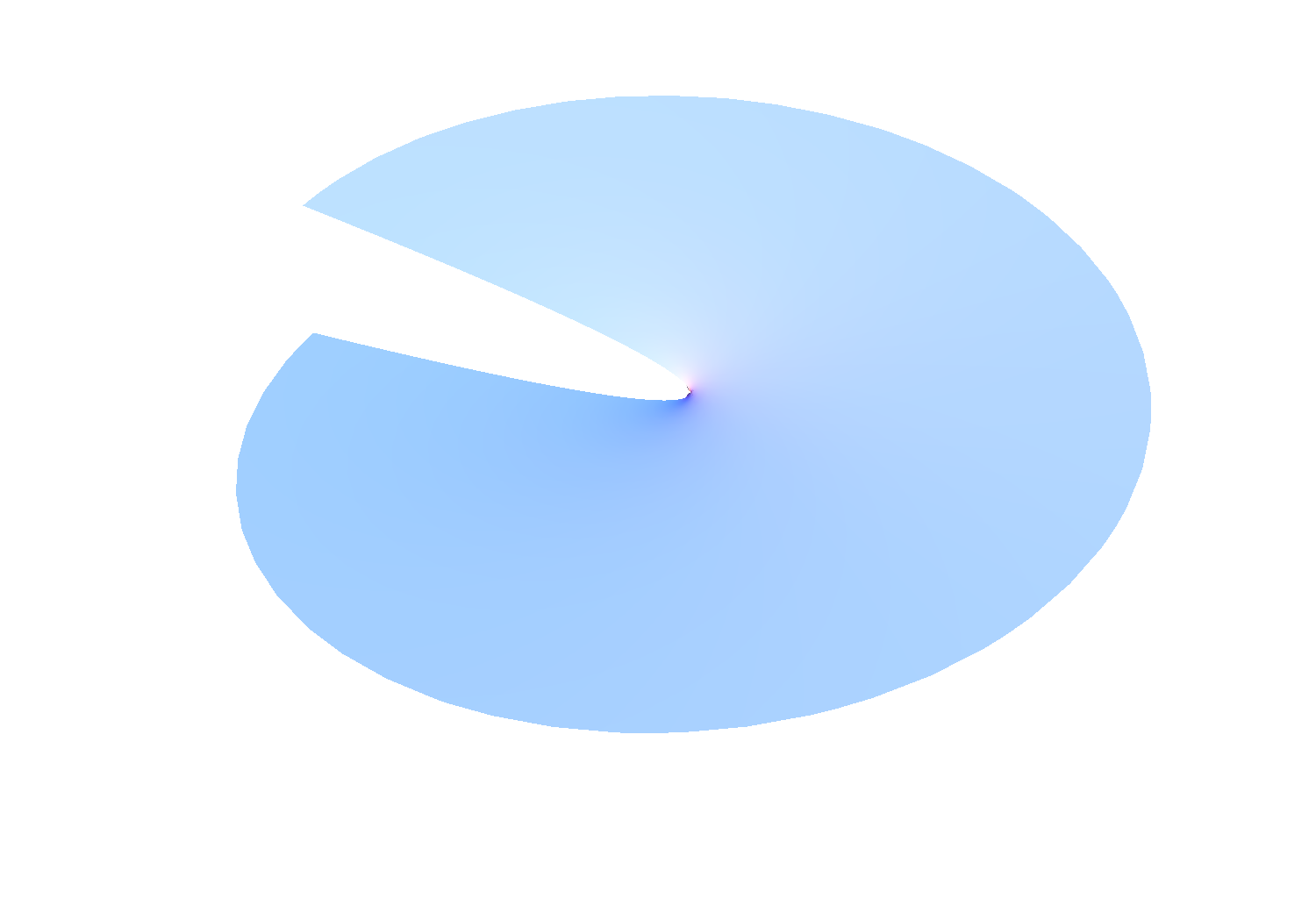}}
  \put(134,134){\includegraphics[height=66pt]{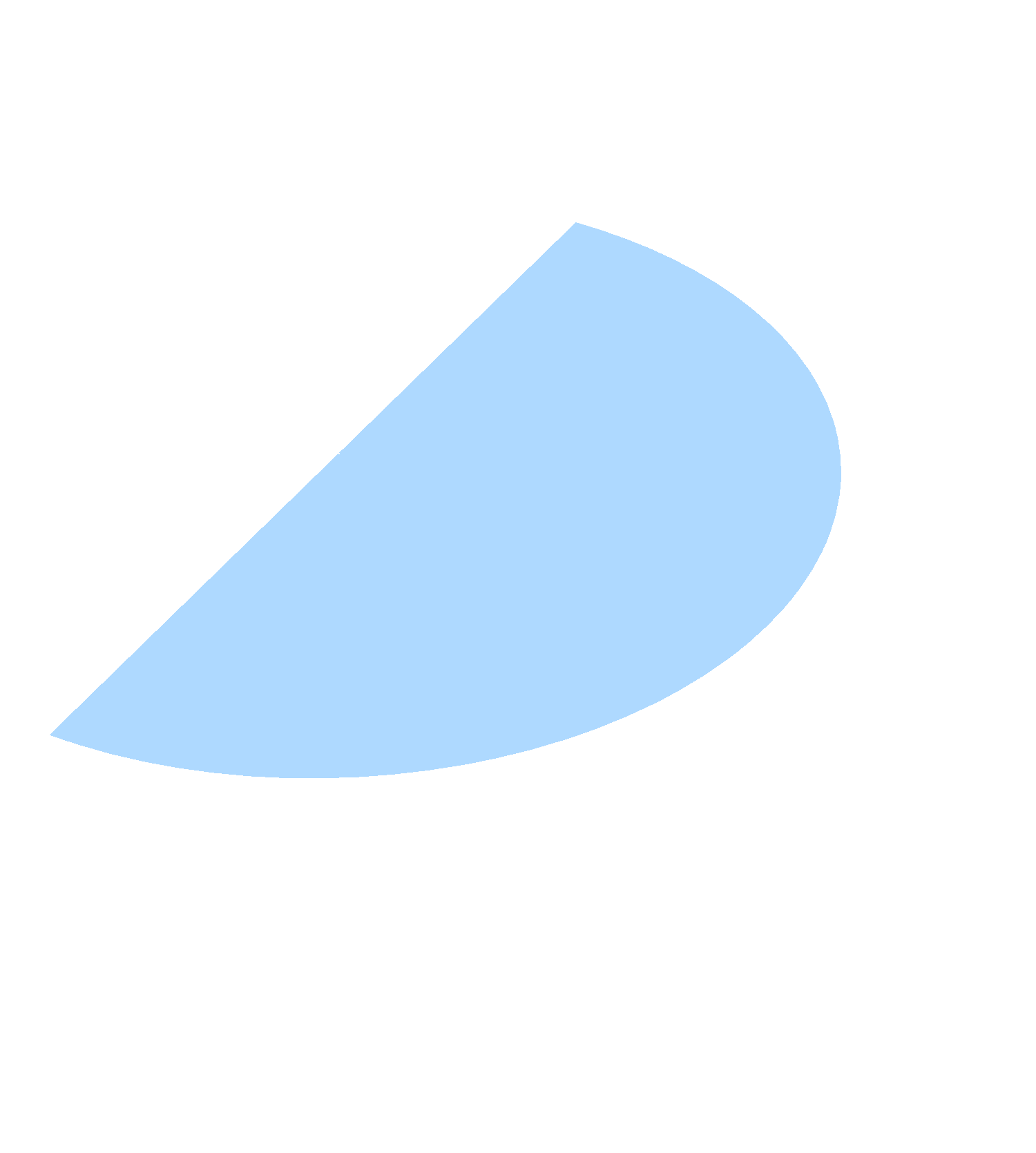}}
  \put(90,180){$\xymatrix{\ar@/^/[r]^T &}$}
  \put(90,160){$\xymatrix{& \ar@/^/[l]^{T^{-1}}}$}
  \put(-25,170){\small $\overline{\mathcal{O}^1\sqcup \mathcal{O}^2}$}
  \put(0,67){\includegraphics[height=66pt]{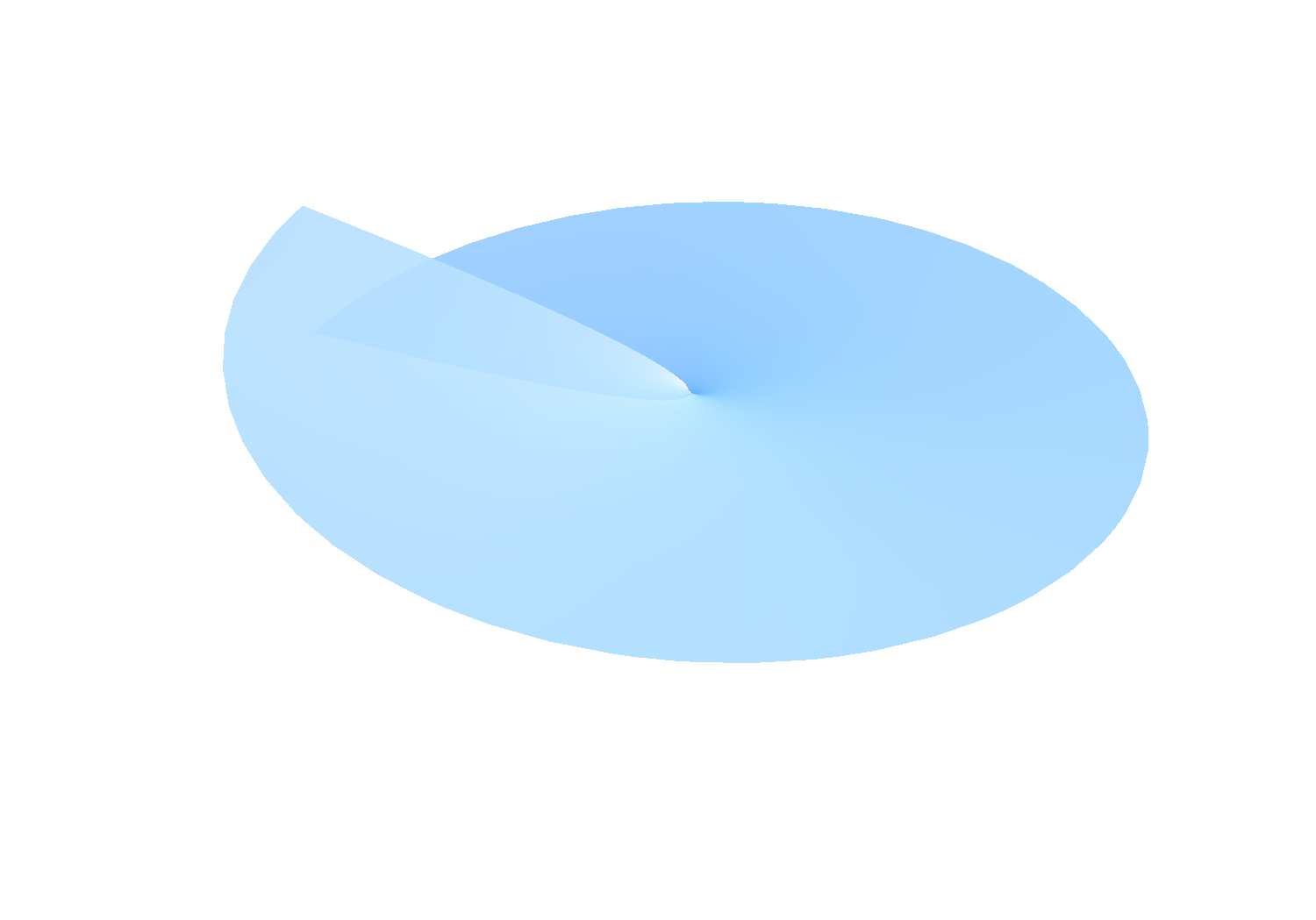}}
  \put(120,67){\includegraphics[height=66pt]{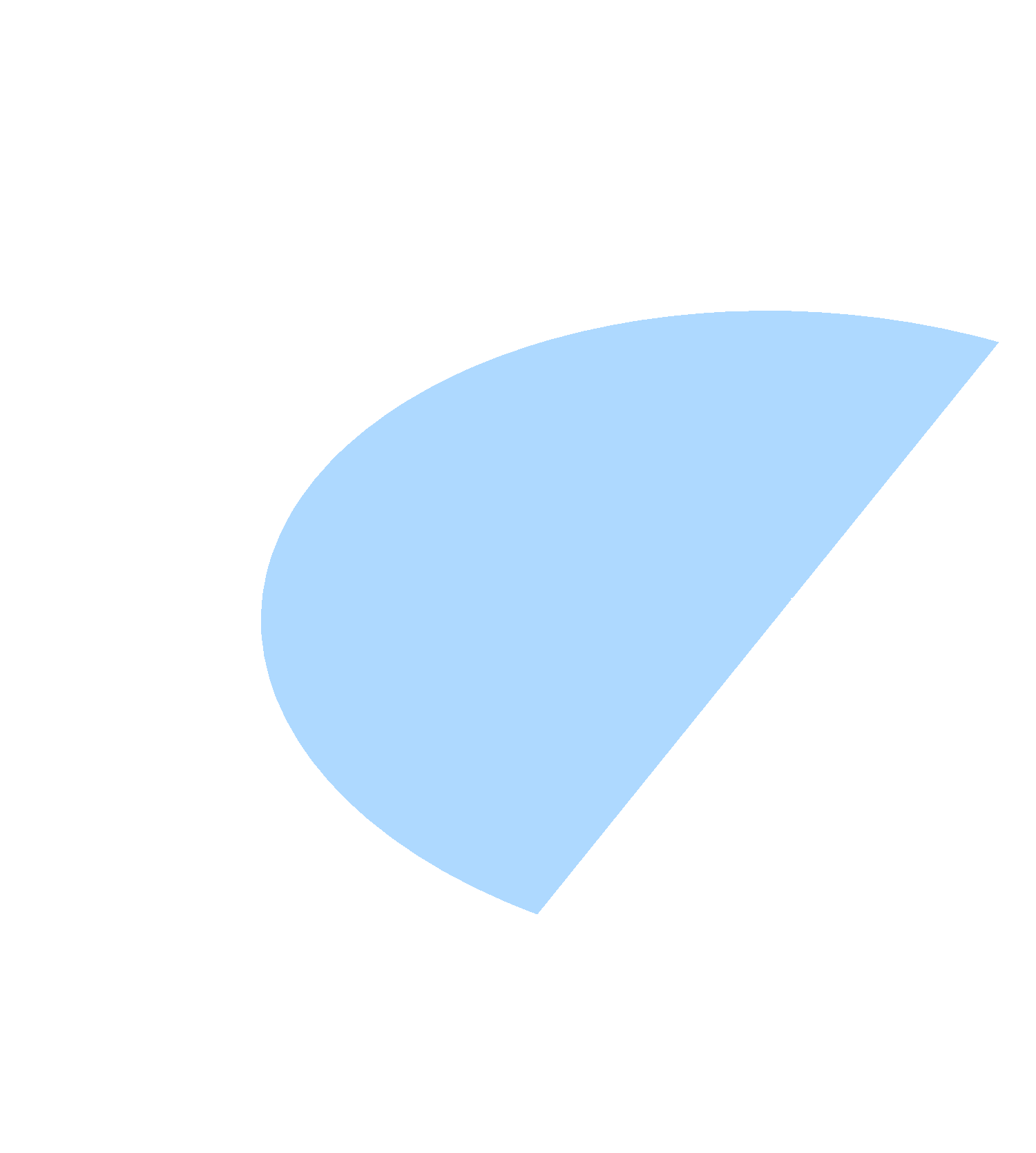}}
  \put(90,110){$\xymatrix{\ar@/^/[r]^T &}$}
  \put(90,90){$\xymatrix{& \ar@/^/[l]^{T^{-1}}}$}
  \put(-25,100){\small $\overline{\mathcal{O}^3\sqcup \mathcal{O}^4}$ }
  \put(0,0){\includegraphics[height=66pt]{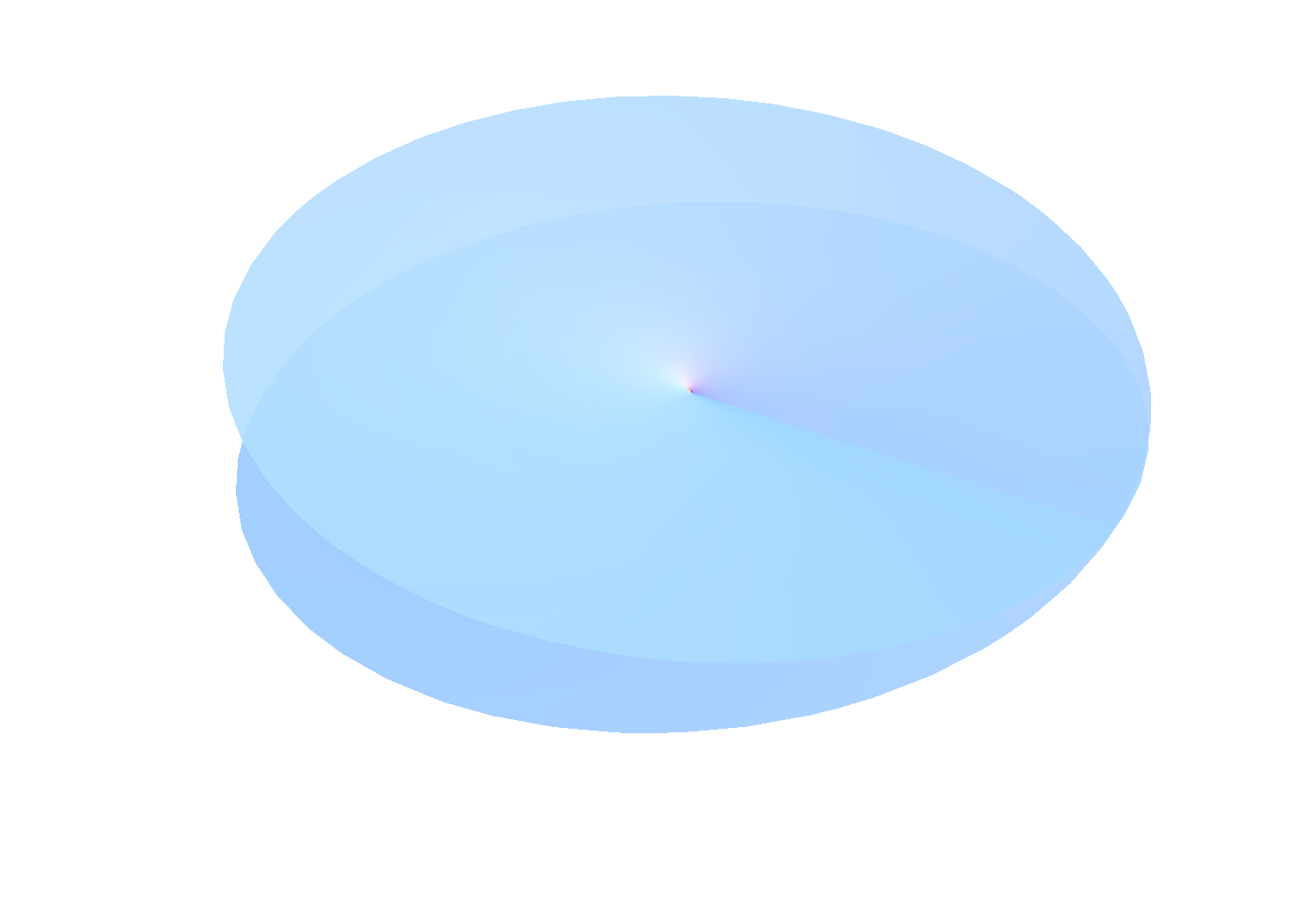}}
  \put(120,0){\includegraphics[height=66pt]{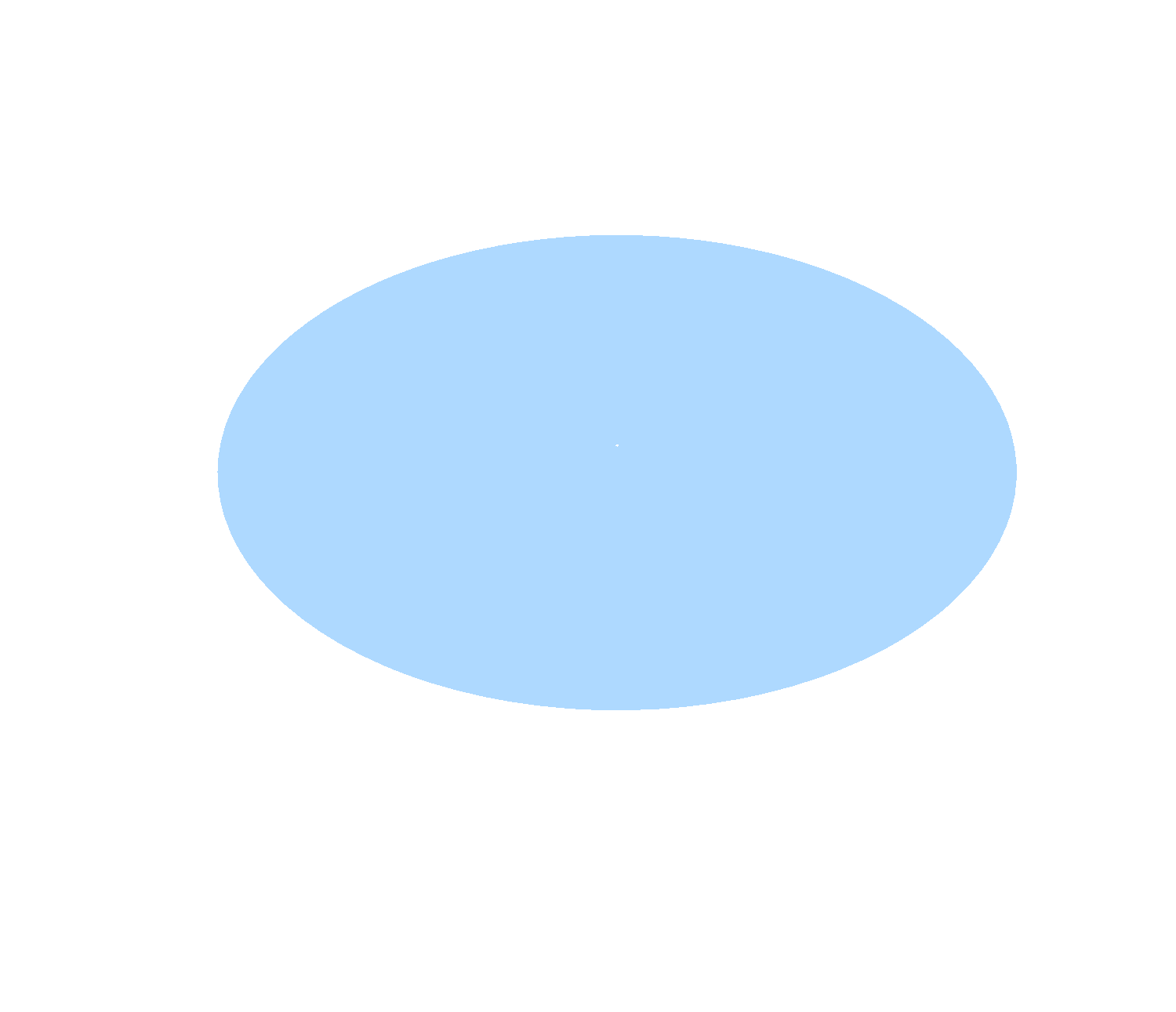}}
  \put(90,40){$\xymatrix{\ar@/^/[r]^T &}$}
  \put(90,20){$\xymatrix{& \ar@/^/[l]^{T^{-1}}}$}
  \put(-7,30){$\mathcal{M}$}
  \put(200,30){$\mathcal{U}$}
\end{picture}
  \caption{Manifold $\mathcal{M}$ and extension of $T$.}
  \label{fig:M}
\end{figure}

Next we construct a space $\mathcal{M}$ by gluing two copies of $\mathcal{O}_\delta^+\cup \mathcal{O}'_\delta$, $\mathcal{O}_\delta^-\cup \mathcal{O}'_\delta$ together properly and extend $u$ as well as $T$ to $\mathcal{M}$. The advantage of doing this is, the straightened free boundary $T(\Gamma_u)$ will be contained in $T(\mathcal{M})$, which is an open neighborhood of the origin in $\R^n$. This transfers the boundary singularity into the interior singularity, which will be easier for us to deal with. 

More precisely, we fix the $\delta$ chosen in Theorem~\ref{thm:injective} and consider $\{\mathcal{O}^i; i=1,2,3,4\}$, a family of subsets in $\R^n$, where
$$\mathcal{O}^1=\mathcal{O}^3:=\mathcal{O}^+_\delta\cup \mathcal{O}'_\delta, \quad \mathcal{O}^2=\mathcal{O}^4:=\mathcal{O}^-_\delta\cup \mathcal{O}'_\delta.$$
Define $u^i: \mathcal{O}^i\rightarrow \R$ as follows: $u^1=u$, where $u$ is the solution to the Signorini problem satisfying the assumptions in the introduction; $u^2$ is the even reflection of $u$ about $x_n$; $u^3=-u^1$ and $u^4=-u^2$. Let $T^i=T^{u^i}:\mathcal{O}^i\rightarrow \R^n$ be the corresponding partial hodograph-Legendre transform defined in \eqref{eq:def-hod-2}.

Consider the disjoint union of $\mathcal{O}^i$: $\sqcup_{i=1}^4\mathcal{O}^i$,  and denote the elements of it by $(x,i)$ for $x\in \mathcal{O}^i$, $i=1,\cdots, 4$. Define 
\begin{alignat*}{2}
\sqcup _{i=1}^4 u^i &: \sqcup^4_{i=1}\mathcal{O}^i\rightarrow \R,&\quad& (\sqcup _{i=1}^4 u^i)((y,j))=u^j(y);\\
\sqcup^4_{i=1}T^i&: \sqcup^4_{i=1}\mathcal{O}^i\rightarrow \R^n,&&(\sqcup^4_{i=1}T^i)((y,j))=T^j(y).
\end{alignat*} 

Now we define an equivalence relation on $\sqcup _{i=1}^4 \mathcal{O}^i$ as follows: $$(x,i)\sim (y,j)\quad\text{iff}\quad (\sqcup^4_{i=1}T^i)((x,i))=(\sqcup^4_{i=1}T^i)((y,j)).$$ It is easy to check from the definition of $T^i$ and Theorem~\ref{thm:injective} that this equivalence relation identifies the points $(x,i)$ and $(x,j)$ if (i) $x\in \Lambda_u$, $i=2,j=3$ or $i=4,j=1$; (ii) $x\in \overline{\Omega_u}$, $i=1,j=2$ or $i=3,j=4$. In particular, for $x\in \Gamma_u$, $(x,i)$ and $(x,j)$ are identified for all $i,j=1,2,3,4$.  

Let $\mathcal{M}:=\sqcup _{i=1}^4 \mathcal{O}^i \slash \sim$ denote the quotient space. Consider on $\mathcal{M}$: 
\begin{alignat*}{2}
\overline{u}&:=\overline{\sqcup _{i=1}^4 u^i} : \mathcal{M}\rightarrow \R,& \quad&\overline{u}(\overline{(y,j)})=u^j(y);\\
\overline{T}&:=\overline{\sqcup^4_{i=1}T^i}: \mathcal{M}\rightarrow \R^n,&&\overline{T}(\overline{(x,i)})=T^i(x).
\end{alignat*} 
It is immediate that $\overline{T}$ is continuous and injective.  Moreover, it is open from $\mathcal{M}$ to $\R^n$ by Theorem~\ref{thm:injective} and the special way we glue $\mathcal{O}^i$.  Hence we obtain that $\overline{T}$ is a homeomorphism from $\mathcal{M}$ to $\overline{T}(\mathcal{M})$. In particular, $\overline{T}(\mathcal{M})$ is an open neighborhood of the origin in $\R^n$, which contains the straightened free boundary $T(\Gamma_u)$.

We still denote the set $\{\overline{(x,i)}\in \mathcal{M}:x\in \Gamma_u\}$ by $\Gamma_u$. It is not hard to observe that $\mathcal{M}\setminus \Gamma_u$ is a double cover of $\mathcal{O}_\delta\setminus \Gamma_u\subset \R^n$ with the covering map $\phi: \overline{(x,i)}\mapsto x$. Hence $\mathcal{M}\setminus \Gamma_u$ can be given a smooth structure which makes $\phi$ into a local diffeomorphism. In the local coordinate charts $\phi_\alpha:U_\alpha\rightarrow \mathcal{O}_\delta\setminus \Gamma_u$, we have $(\overline{u}\circ\phi^{-1}_\alpha)(x)=u(x)$, where $u$ is the extended function via the even or odd reflection about $\Lambda_u$ or $\Omega_u$, hence $\overline{u}$ is continuous on $\mathcal{M}$, smooth in $\mathcal{M}\setminus \Gamma_u$ and $\Delta \overline{u}=0$ there. Similarly, one can compute $(\overline{T}\circ \phi^{-1}_\alpha)(x)=(x'', \partial_{n-1} u(x), \partial_n u(x))$, which is a diffeomorphism on $\phi_\alpha(U_\alpha)$ by Theorem~\ref{thm:injective} (apply Theorem~\ref{thm:injective} for the extended $u$). This shows that $\overline{T}: \mathcal{M}\setminus \Gamma_u \rightarrow \overline{T}(\mathcal{M}\setminus \Gamma_u)$ is a diffeomorphism.

From now on, with slight abuse of the notation we will still denote $\overline{u}$ by $u$ and $\overline{T}$ by $T$.
In the following, we will simply write $\partial_k u$, $DT$, etc.\ while having in mind that we are taking the derivatives in the local coordinates.

\section{Partial Legendre transform and a nonlinear PDE}
\label{sec:legendre-funct-nonl}
In this section we study the partial Legendre transform of $u$ and the fully nonlinear PDE it satisfies. We let $$\U:=T(\mathcal{M}),$$ which is an open neighborhood of the origin and 
$$\cP:=T(\Gamma_u)$$
be the straightened free boundary.

\subsection{Partial Legendre transform}
\label{sec:legendre-transf-its}
For $y\in \U$, we define the partial Legendre transform of $u$ by the identity
\begin{equation}
\label{eq:legendrefun}
v(y)=u(x)-x_{n-1}y_{n-1}-x_ny_n,
\end{equation}
where $x=T^{-1}(y)\in \mathcal{M}$. It is immediate to check the following properties of $v$:
\begin{enumerate}[(i)]
\item $v$ is odd about $y_{n-1}$ and even about $y_n$.

\item $v$ is continuous in $\U$, smooth in $\U\setminus \mathcal{P}$ and $v=0$ on $\mathcal{P}$.

\item A direct computation shows that in ${\U}\setminus \mathcal{P}$
\begin{equation}\label{eq:derivative}
\begin{split}
&\partial_{y_i} v=\partial_{x_i} u, \quad i=1,\ldots,n-2\\
&\partial_{y_{n-1}} v=-x_{n-1},\quad \partial_{y_n} v=-x_n.
\end{split}
\end{equation}
Hence $T^{-1}$ can be written as
\begin{align*}
x=T^{-1}(y)=(y_1,\ldots,y_{n-2},-\partial_{y_{n-1}} v, -\partial_{y_n} v).
\end{align*}
The Jacobian matrix of $v$ is then
\begin{equation}\label{eq:DT-1}
DT^{-1}(y)=\begin{pmatrix}
I_{n-2} & 0 \\
A & B
\end{pmatrix},\quad
DT=\begin{pmatrix}
I_{n-2} & 0 \\
-B^{-1}A & B^{-1}
\end{pmatrix},
\end{equation}
where
\begin{align*}
A&=\begin{pmatrix}
-\partial_{y_{n-1}y_1}v &-\partial_{y_{n-1}y_2} v & \ldots & -\partial_{y_{n-1}y_{n-2}} v\\
-\partial_{y_ny_1}v & -\partial_{y_ny_2}v& \ldots & -\partial_{y_ny_{n-2}}v
\end{pmatrix},\\
B&=\begin{pmatrix}
-\partial_{y_{n-1}y_{n-1}}v &-\partial_{y_{n-1}y_n}v \\
-\partial_{y_ny_{n-1}}v & -\partial_{y_ny_n}v
\end{pmatrix}.
\end{align*}
Since $u\in C^1(\mathcal{M}\setminus \Gamma_u)$ and its differential has an continuous extension to $\Gamma_u$, this together with the continuity of $T^{-1}$ and \eqref{eq:derivative} imply that $v\in C^1(\U)$. 

\item The restriction of $T^{-1}$ to $\cP$ is given by 
$$
T^{-1}(y'',0,0)=(y'',-\partial_{y_{n-1}}v(y'',0,0),0)\in \Gamma_u,
$$
which gives a local parametrization of the free boundary $\Gamma_u$. 
Thus, the
regularity of the free boundary is directly related to the regularity
of $\partial_{y_{n-1}}v$, restricted to $\cP$.
\end{enumerate}

\subsection{Fully nonlinear equation for $v$}
A direct computation using \eqref{eq:derivative} and \eqref{eq:DT-1} shows that 
\begin{multline*}
\partial_{x_ix_i}u=\partial_{y_iy_i}v-(\partial_{y_{n-1}y_i}v, \partial_{y_ny_i}v)\left(\begin{array}{@{}cc@{}}
\partial_{y_{n-1}y_{n-1}}v& \partial_{y_{n-1}y_n}v \\
\partial_{y_{n-1}y_n}v & \partial_{y_ny_n}v
\end{array}
\right)^{-1}\binom{\partial_{y_iy_{n-1}}v}{\partial_{y_iy_n}v}\\
i=1,\cdots, n-2
\end{multline*}
\begin{equation*}
\begin{pmatrix}
\partial_{x_{n-1}x_{n-1}} u & \partial_{x_{n-1}x_n}u\\
\partial_{x_n x_{n-1}}u & \partial_{x_nx_n} u
\end{pmatrix} = 
-\begin{pmatrix}
\partial_{y_{n-1}y_{n-1}} v & \partial_{y_{n-1}y_n} v\\
\partial_{y_ny_{n-1}}v &\partial_{y_ny_n} v
\end{pmatrix}^{-1}.
\end{equation*}
Since $\Delta u=0$ in $\mathcal{M}\setminus
\Gamma_u$, the Legendre function $v$
satisfies the following fully nonlinear equation in $\U\setminus \cP$ 
\begin{multline}\label{eq:mainequation}
\tilde{F}(D^2_yv)=\sum_{i=1}^{n-2}\partial_{y_iy_i}v-\operatorname{tr}\left(\begin{array}{@{}cc@{}}
\partial_{y_{n-1}y_{n-1}}v & \partial_{y_{n-1}y_n}v \\
\partial_{y_{n-1}y_n}v & \partial_{y_ny_n}v
\end{array}
\right)^{-1}\\-\sum_{i=1}^{n-2}(\partial_{y_iy_{n-1}}v, \partial_{y_iy_n}v) \left(\begin{array}{@{}cc@{}}
\partial_{y_{n-1}y_{n-1}}v& \partial_{y_{n-1}y_n}v \\
\partial_{y_{n-1}y_n}v & \partial_{y_ny_n}v
\end{array}
\right)^{-1}
\binom{\partial_{y_iy_{n-1}}v}{\partial_{y_iy_n}v}=0.
\end{multline}
Multiplying both sides of \eqref{eq:mainequation} by 
$$-J(v)=-\det \left(\begin{array}{@{}cc@{}}
\partial_{y_{n-1}y_{n-1}}v & \partial_{y_{n-1}y_n}v \\
\partial_{y_{n-1}y_n}v & \partial_{y_ny_n}v
\end{array}
\right),$$
we can write it in the form
\begin{multline*}
F(D^2v)=-J(v)\tilde{F}(D^2v)=\partial_{y_{n-1}y_{n-1}}v+\partial_{y_ny_n}v-J(v)\sum_{i=1}^{n-2}\partial_{y_iy_i} v\\+\sum_{i=1}^{n-2}((\partial_{y_iy_{n-1}} v)^2\partial_{y_ny_n}v-2\partial_{y_iy_{n-1}}v\partial_{y_iy_n}v\partial_{y_{n-1}y_n}v+(\partial_{y_iy_n}v)^2\partial_{y_{n-1}y_{n-1}}v)=0, 
\end{multline*}
which can be further rewritten as
\begin{equation}\label{eq:mainequation2}
F(D^2v)=\partial_{y_{n-1}y_{n-1}}v+\partial_{y_ny_n}v-\sum_{i=1}^{n-2} \det (V^i)=0,
\end{equation}
where $V^i$, $i=1,\ldots,n-2$, is the $3\times 3$ matrix
$$
V^i=\left(
    \begin{array}{@{}lll@{}}
      \partial_{y_iy_i}v & \partial_{y_iy_{n-1}}v & \partial_{y_iy_n}v \\
      \partial_{y_{n-1}y_i}v &\partial_{y_{n-1}y_{n-1}}v & \partial_{y_{n-1}y_n}v \\
     \partial_{y_ny_i}v & \partial_{y_ny_{n-1}}v & \partial_{y_ny_n}v \\
    \end{array}
  \right)_{3\times 3}.
$$

\subsection{Blowup of $v$ at $\mathcal{P}$}
In order to study the asymptotic of higher derivatives of $v$ at the straightened free boundary $\cP$, we study the blowup of $v$ at points on $\mathcal{P}$.

Let $v_{y_0}$ be the Legendre function of $u_{x_0}$ as in \eqref{eq:legendrefun}, where $y_0=T(x_0)=(x_0'',0,0)\in \cP$. It is not hard to compute that
\begin{equation}\label{eq:computev}
v_{y_0}(y)=-\frac{4}{27C_{x_0}^2}\left (\left(\frac{y_{n-1}}{(\nu'_{x_0})_{n-1}}\right)^3-3\left(\frac{y_{n-1}}{(\nu'_{x_0})_{n-1}}\right)y_n^2\right)+(\nu''_{x_0}\cdot y'') \left(\frac{y_{n-1}}{(\nu'_{x_0})_{n-1}}\right)
\end{equation}
where
\begin{equation}\label{eq:nu}
\nu'_{x_0}=(\nu''_{x_0}, (\nu'_{x_0})_{n-1})=\frac{(-\nabla f (x''_0), 1)}{\sqrt{1+|\nabla f(x''_0)|^2}}
\end{equation}
is the unit outer normal of $\Lambda_u$ at $x_0$. In particular, at the origin, 
$$v_0(y)=\left(-\frac{4}{27C_0^2}\right)\left(y_{n-1}^3-3y_{n-1}y_n^2\right).$$
For $y=(y'',y_{n-1},y_n)\in \R^n$ and $r>0$, we consider the non-isotropic dilation 
\begin{equation}\label{eq:nonisotropic}
\delta_r y=(r^2 y'', ry_{n-1},ry_n)
\end{equation}
and the rescaling at $y_0$ (note $v(y_0)=0$)
\begin{equation}\label{eq:rescale_v}
v_{r,y_0}(y)=\frac{v(y_0+\delta_r y)}{r^3}.
\end{equation}
From \eqref{eq:legendrefun}, \eqref{eq:rescale_v} and \eqref{eq:homgen-rescal}, one can easily check that
\begin{equation}\label{eq:Legendre_vr}
v_{r,y_0}(y)=u_{r^2,x_0}(x)-x_{n-1}y_{n-1}-x_ny_n, \quad x=(T^{u_{r^2,x_0}})^{-1}(y).
\end{equation}
Here rescaling family $u_{r^2,x_0}$ and $T^{u_{r^2,x_0}}$ are defined on $\mathcal{M}_{r^{-2}, x_0}$, where $\mathcal{M}_{r^{-2},x_0}$ are the topological spaces obtained by gluing four rescaled copies $$\mathcal{O}^i_{r^{-2},x_0}:=(\mathcal{O}^i-x_0)/r^2$$ together as in the construction of $\mathcal{M}$.

To study the convergence of $v_{r,y_0}$ to  $v_{y_0}$, we first show a lemma which concerns about in the local coordinate charts the uniform convergence of $(T^{u_{r^2,x_0}})^{-1}$ to $(T^{u_{x_0}})^{-1}$. The following two facts are easy to verify:

\begin{enumerate}[(i)]
\item $T^{u_{r^2,x_0}}$ is bijective and $T^{u_{r^2,x_0}}(\mathcal{M}_{r^{-2},x_0})=\delta_{r^{-1}}(\U)$, where $\delta_{r^{-1}}$ is the non-isotropic dilation in \eqref{eq:nonisotropic}. In particular, we have $T^{u_{x_0}}(\mathcal{M}_{\infty,x_0})=\R^n$. 

\item Let 
$$\mathcal{Q}^1:=\{y\in \R^n:y_{n-1}\geq 0, y_n\leq 0\},\ \mathcal{Q}^2:=e^{i\pi/2}\mathcal{Q}^1, \ \mathcal{Q}^3:=e^{i\pi/2}\mathcal{Q}^2, \ \mathcal{Q}^4:=e^{i\pi/2}\mathcal{Q}^3$$
and let 
$$\mathcal{O}^{i,\ast}_{r^{-1},x_0}:=\{\overline{(x,i)}\in \mathcal{M}_{r,x_0}: x\in \mathcal{O}^i_{r^{-1},x_0}\}, \quad i=1,2,3,4.$$
Then for any sufficiently small $r$ and $x\in \Gamma_u\cap \mathcal{O}_\delta$, we have $T^{u_{r^2,x_0}}(\mathcal{O}^{i,\ast}_{r^{-2},x_0})\subset \mathcal{Q}^i$, i.e. $T^{u_{r^2,x_0}}$ always maps the copy $\mathcal{O}^{i,\ast}_{r^{-2},x_0}$ into the corresponding ``quarter'' domain $\mathcal{Q}^i$. 
\end{enumerate}

Let $\phi^i: \mathcal{O}^{i,\ast}_{r^{-1},x_0}\rightarrow \mathcal{O}^{i}_{r^{-1},x_0}$ with $\phi^i(\overline{(x,i)})=x$ denote the restriction of the covering map to $\mathcal{O}^{i,\ast}_{r^{-1},x_0}$. 

\begin{lemma}\label{lem:dilation}
Let $y_0\in \cP$ and $x_0=T^{-1}(y_0)$. Then
$$\phi^i\circ (T^{u_{r^2,x_0}})^{-1}\rightarrow \phi^i\circ(T^{u_{x_0}})^{-1}, \quad r\rightarrow 0_+$$
uniformly in any compact subset $K\subset \mathcal{Q}^i$, $i=1,2,3,4$. Moreover, this convergence is also uniform for $y_0$ varying in a compact subset of $\cP$.
\end{lemma}
\begin{proof}
Given $y_0\in \cP$ and a compact set $K\subset\mathcal{Q}^i$, by (i) and (ii) above, there is a positive $r_1=r_1(y_0,K)$, such that $K\subset T^{u_{r^2,x_0}}(\mathcal{O}^{i,\ast}_{r^{-2},x_0})$ for any $r<r_1$. 

By the rescaled version of Proposition~\ref{prop:uniformity2}(i), there exists $r_0\in (0,r_1)$ small such that $\phi^i\circ (T^{u_{r^2,x_0}})^{-1}(K)\subset \mathcal{O}^{i}_{r_0^{-2},x_0}$, which is a bounded subset in $\R^n$, for any $0<r<r_0$.

We know from the $C^1_{\loc}$ convergence of $u_{r^2,x_0}$ to $u_{x_0}$ that
$$T^{u_{r^2,x_0}}\circ (\phi^i)^{-1}\rightarrow T^{u_{x_0}}\circ (\phi^i)^{-1} \text{ uniformly on } \mathcal{O}^{i}_{r_0^{-2},x_0}.$$ Moreover, the limit $\phi^i\circ (T^{u_{x_0}})^{-1}$ is  continuous  on $\mathcal{Q}^i$ by a direct computation. Therefore, $\phi^i\circ(T^{u_{r^2,x_0}})^{-1}\rightarrow \phi^i\circ (T^{u_{x_0}})^{-1}$ uniformly in $K$. 

By Theorem~\ref{prop:uni} and the continuous dependence of $(T^{u_{x_0}})^{-1}$ on $x_0$, we have the above convergence is uniform for $y_0$ in any compact subset of $\cP$.
\end{proof}

Next we show the following compactness results.
\begin{proposition}\label{prop:compact-v}  Let $K$ be a compact subset in $\R^n\setminus \{y_{n-1}=y_n=0\}$. Then for any multi-index $\alpha$, $$\partial^\alpha_y v_{r,y_0} \rightarrow \partial^\alpha_y v_{y_0}, \quad r\rightarrow 0_+$$ uniformly in $K$.
Moreover, the above convergence is also  uniform for $y_0$ varying in a compact subset of $\cP$. 
\end{proposition}
\begin{proof} Given $y_0\in \cP$, let $x_0=T^{-1}(y_0)\in \Gamma_u$. For $|\alpha|=0$ and $1$, using \eqref{eq:Legendre_vr} and \eqref{eq:derivative}, we can easily conclude from Lemma~\ref{lem:dilation} together with the uniform convergence of $u_{r^2,x_0}$ to $u_{x_0}$ that, $\partial^\alpha v_{r,y_0}$ converges to $\partial^\alpha v_{y_0}$ uniformly in $K$. 

For $|\alpha|\geq 2$, using \eqref{eq:derivative} and \eqref{eq:DT-1} one can express $\partial^\alpha v_{r,y_0}$ in terms of $\partial^{\alpha' }u_{r^2,x_0}$ with $|\alpha'|\leq |\alpha|$, i.e.\ for fixed $\alpha$, 
\begin{equation}\label{eq:v_in_u}
\partial^\alpha_y v_{r,y_0}(y)=f_\alpha(\partial^{\alpha'}_x u_{r^2,x_0}(x))\big|_{x=(T^{u_{r^2,x_0}})^{-1}(y)},
\end{equation}
where $f_\alpha$ is some polynomial. 
For $K\subset \R^n\setminus\{y_{n-1}=y_n=0\}$ compact, $(T^{u_{x_0}})^{-1}(K)$ is also compact and $(T^{u_{x_0}})^{-1}(K)\cap \Gamma_{u_{x_0}}=\emptyset$.  By the local uniform convergence of $(T^{u_{r^2,x_0}})^{-1}$ to $(T^{u_{x_0}})^{-1}$ (Lemma~\ref{lem:dilation}) as well as the flatness of the free boundary $\Gamma_{u_{x_0}}$ (i.e.\ the Hausdorff distance between $\Gamma_{u_{r^2,x_0}}$ and $\Gamma_{u_{x_0}}$ goes to zero as $r$ goes to zero), there exists $K'\subset \R^n$ compact and $r_0=r_0(K, K')$ small, such that for all $r<r_0$, we have 
\begin{equation}\label{eq:separation1}
\left(\phi^i\circ (T^{u_{r^2,x_0}})^{-1}\right)(K\cap \mathcal{Q}^i)\subset K', \quad i=1,2,3,4
\end{equation}
and 
\begin{equation}\label{eq:separation2}
K' \cap \Gamma_{u_{r^2,x_0}}=\emptyset.
\end{equation}
Note that \eqref{eq:separation2} implies that for any $r<r_0$ and $i=1,2,3,4$, $u_{r^2,x_0}\circ (\phi^i)^{-1}$  are harmonic in $K'$. Thus for any multi-index $\alpha$, we have $\partial^\alpha (u_{r^2,x_0}\circ (\phi^i)^{-1})\rightarrow \partial^\alpha (u_{x_0}\circ (\phi^i)^{-1})$ uniformly in $K'$. This combined with \eqref{eq:v_in_u} ,\eqref{eq:separation1} and Lemma~\ref{lem:dilation} gives the conclusion.

Due to Theorem~\ref{prop:uni} and Lemma~\ref{lem:dilation}, the above convergence is uniform in $y_0$ varying the compact subset of $\cP$.
\end{proof}

From Proposition~\ref{prop:compact-v} one can get continuous extension of higher order (properly weighted) derivatives  of $v$ at $\cP$. 
\begin{corollary}\label{cor:extension}
For each $y_0\in \cP$, $x_0=T^{-1}(y_0)$, we extend the following functions to $y_0$ by setting  
\begin{align*}
&\textup{(i)}\quad \sqrt{y_{n-1}^2+y_n^2}\partial_{ij} v(y_0)= 0, \quad i, j=1, \ldots,n-2;\\
&\textup{(ii)}\quad \partial_{i,n-1} v (y_0)=\frac{(\nu_{x_0})_i}{(\nu_{x_0})_{n-1}}=\partial_{i}f(x''_0), \quad \partial_{i, n} v(y_0)=0,\quad i=1,\ldots,n-2;\\
&\textup{(iii)}\quad \partial_{\alpha, \beta} v(y_0)= 0, \quad \alpha, \beta=n-1,n;\\
&\textup{(iv)}\quad\partial_{n-1,n-1,n-1} v(y_0)=  -\frac{8}{9(\nu_{x_0})_{n-1}^3C_{x_0}^2}, \quad \partial_{n-1,n,n} v(y_0)= \frac{8}{9(\nu_{x_0})_{n-1}C_{x_0}^2};\\
&\qquad \ \partial_{n-1,n-1,n} v(y_0)= 0, \quad \partial_{n,n,n} v(y_0)= 0.
\end{align*}
Then after such extension the above functions are continuous on $\U$. 
\end{corollary}
\begin{proof} 
The proof is based on the following two facts (a) (b) and a blowup argument.

(a) For $i,j\in\{1,\ldots, n-2\}$, $\alpha, \beta,\gamma\in \{n-1,n\}$,
\begin{align*}
&\partial_{i,j} v_{r,y_0}(y)=r(\partial_{i,j}v)(y_0+ry),\quad \partial_{i,\alpha} v_{r,y_0}(y)=(\partial_{i,\alpha} v)(y_0+ry),\\
&\partial_{\alpha, \beta} v_{r,y_0}(y)=r^{-1}(\partial_{\alpha,\beta}v)(y_0+ry),\quad \partial_{\alpha, \beta, \gamma} v_{r,y_0}(y)=(\partial_{\alpha,\beta,\gamma}v)(y_0+ry).
\end{align*}

(b)
 From the $C^{1,\alpha}$ regularity of $\Gamma_u$ and the explicit expression of $v_{y_0}$, we have the map $y_0\mapsto \partial^\alpha v_{y_0}$ is continuous from $\cP$ to $C_0(K)$, where $K\subset \R^n\setminus \{y_{n-1}=y_n=0\}$ compact. This together with Proposition~\ref{prop:compact-v} gives that for any multi-index $\alpha$
$$\lim_{\substack{\hat{y_0}\in \cP, \hat{y_0}\rightarrow y_0\\r\rightarrow 0_+}}\|\partial^\alpha v_{r,\hat{y_0}}-\partial^\alpha v_{y_0}\|_{C_0(K)}=0$$ 

We proceed to show the extended functions are continuous at $y_0\in \cP$. First they are continuous on $\cP$ from the $C^{0,\alpha}$ dependence of $\nu'_{x_0}$ on $x_0$.  Next for $y\notin \cP$, we use $y''$ to denote $(y'',0,0)$ and let
$$r=|y-y''|=\sqrt{y_{n-1}^2+y_n^2},\quad \eta=\frac{y-y''}{r}=\frac{(0,y_{n-1},y_n)}{r}.$$
Then 
$$v(y)=v_{r,y''}(\eta), \quad \eta\in \{y''=0, y_{n-1}^2+y_n^2=1\}.$$
As $y\rightarrow y_0$, we have $r\rightarrow 0_+$ and $y''\rightarrow y_0$. Thus by (b) above, for a fixed $K\subset \R^n\setminus \{y_{n-1}=y_n=0\}$, which is compact and contains the set $\{y''=0, y_{n-1}^2+y_n^2=1\}$, we have
\begin{equation}\label{eq:converge-v}
\partial^\alpha v_{r,y''}\rightarrow \partial^\alpha v_{y_0} \text{ in } C_0 (K) \text{ as } y\rightarrow y_0.
\end{equation}

From the explicit expression of $v_{y_0}$ (see \eqref{eq:computev}) we have
\begin{align*}
&D^2 v_{y_0}(y)=\begin{pmatrix}
0&(\nu''_{x_0})^t/(\nu'_{x_0})_{n-1}& 0\\
\nu''_{x_0}/(\nu'_{x_0})_{n-1} &-8y_{n-1}/9(\nu'_{x_0})^3_{n-1}C_{x_0}^2 &  8y_n/9(\nu'_{x_0})_{n-1}C_{x_0}^2\\
0& 8y_n/9(\nu'_{x_0})_{n-1}C_{x_0}^2& 8y_{n-1}/9(\nu'_{x_0})_{n-1}C_{x_0}^2
\end{pmatrix}\\
&\partial_{n-1,n-1,n-1} v_{y_0}\equiv-\frac{8}{9(\nu'_{x_0})_{n-1}^3C_{x_0}^2},\quad \partial_{n-1,n,n} v_{y_0}\equiv\frac{8}{9(\nu'_{x_0})_{n-1}C_{x_0}^2}\\
&\partial_{n-1,n-1,n}v_{y_0}=\partial_{n,n,n}v_{y_0}\equiv 0
\end{align*}
This together with (a) and \eqref{eq:converge-v} gives (i)-(iv).
\end{proof}

\section{Smoothness}\label{sec:smoothness}
\subsection{Subelliptic structure}\label{sec:subelliptic-coeff}
In this section we show the fully nonlinear equation  $F(D^2v)=0$ in \eqref{eq:mainequation2} has a subelliptic structure in $\mathcal{U}$. 

Let $\mathcal{S}^{n\times n}$ be the space of $n\times n$ symmetric matrices and we may consider $F$ as a smooth function on $\mathcal{S}^{n\times n}$. Let $F_{ij}(P)=\frac{\partial F}{\partial p_{ij}}(P)$ for $P=(p_{ij})\in \mathcal{S}^{n\times n}$ be the linearization of $F$ at $P$. A direct computation shows that for $i,j=1,\ldots, n-2$ and $y\in \U\setminus \cP$, the linearization of $F$ at $D^2v$ has the form
\begin{align*}
F_{ij}(D^2v)&=0, \quad  i\neq j\\
F_{ii}(D^2 v)&=-(\partial_{n-1,n-1} v\partial_{nn}v-(\partial_{n-1,n}v)^2)\\
& =-\left( \frac{\partial_{n-1,n-1} v}{y_{n-1}}\frac{\partial_{nn} v}{y_{n-1}}\right) y_{n-1}^2 +\frac{(\partial_{n-1,n} v)^2}{y_n^2} y_n^2\\
F_{i, n-1}(D^2v)&=\partial_{n-1, i} v\partial_{nn}v-\partial_{n-1,n}v\partial_{n,i}v\\
&=\left(\partial_{n-1,i} v \frac{\partial_{nn} v}{y_{n-1}}\right) y_{n-1}-\left(\partial_{n,i} v\frac{\partial_{n-1,n}v}{y_n}\right) y_n\\
F_{i,n}(D^2v)&=-(\partial_{n-1,i} v\partial_{n,n-1}v-\partial_{n-1,n-1}v\partial_{n,i}v)\\
&=\left(\partial_{n-1,i}v\frac{\partial_{n,n-1}v}{y_{n} }\right) y_{n}-\left(\frac{\partial_{n-1,n-1}v}{y_{n-1}} \partial_{n,i}v\right) y_{n-1}\\
F_{n-1,n-1}(D^2 v)&=1-\sum_{i=1}^{n-2}\left(\partial_{ii}v\partial_{nn}v-\partial_{i,n}v\partial_{n,i}v\right)\\
F_{n-1,n}(D^2 v)&=\sum_{i=1}^{n-2}\left(\partial_{ii}v\partial_{n,n-1}v-\partial_{i,n-1}v\partial_{n,i}v\right)\\
F_{n,n}(D^2 v)&=1-\sum_{i=1}^{n-2}\left(\partial_{ii}v\partial_{n-1,n-1}v-\partial_{i,n-1}v\partial_{n-1,i} v\right).
\end{align*}
Observe that one can write $\{F_{ij}(D^2v)\}_{i,j}=ABA^t$, with $A$ and $B$  symmetric matrices of the following forms:
\begin{align*}
A=\begin{pmatrix}
Y & 0 & 0&\cdots & 0\\
0 & Y & 0& \cdots & 0\\
\vdots & \vdots & \ddots &\ddots & \vdots\\
0 & 0 & \cdots & Y & 0\\
0 & 0 & 0&\cdots & I_2
\end{pmatrix}_{n\times 2(n-1)}
\end{align*}
\begin{align*}
B=(b_{ij})=\begin{pmatrix}
B_0&0& 0&\cdots & B_{1}\\
0& B_0 &0&\cdots & B_{2}\\
\vdots &\vdots& \ddots &\ddots & \vdots\\
0& 0 & \cdots & B_0 & B_{n-2}\\
(B_{1})^t&(B_{2})^t &\cdots &\cdots &B_{n-1}
\end{pmatrix}_{2(n-1)\times 2(n-1)}
\end{align*}
where 
\begin{align*}
Y&=(y_{n-1},y_n), \\
B_0&=\begin{pmatrix}
b_0(y) & 0\\
0& \widetilde{b_0}(y)
\end{pmatrix} , \quad B_{i}=\begin{pmatrix} b_{i,1}(y)&b_{i,2}(y)\\ \widetilde{b_{i,1}}(y)&\widetilde{b_{i,2}}(y)\end{pmatrix}, \quad i=1,\ldots, n-1
\end{align*}
with 
\begin{alignat*}{2}
b_0(y)&=- \frac{\partial_{n-1,n-1} v}{y_{n-1}}\frac{\partial_{nn} v}{y_{n-1}},& \quad \widetilde{b_0}(y)&=\frac{(\partial_{n-1,n} v)^2}{y_n^2};\\
b_{i,1}(y)&=\partial_{n-1,i} v \frac{\partial_{nn} v}{y_{n-1}},&\quad \widetilde{b_{i,1}}(y)&=-\partial_{n,i} v\frac{\partial_{n-1,n}v}{y_n},\\
b_{i,2}(y)&=-\partial_{n,i}v\frac{\partial_{n-1,n-1}v}{y_{n-1}} ,&\quad \widetilde{b_{i,2}}(y)&=\partial_{n-1,i}v\frac{\partial_{n,n-1}v}{y_{n} },
\end{alignat*}
for $i=1,\ldots, n-2$,
and for $i=n-1$,
\begin{align*}
B_{n-1}=\begin{pmatrix}
F_{n-1,n-1}& F_{n-1,n}\\
F_{n,n-1}& F_{n,n}
\end{pmatrix}.
\end{align*}
Note that $B(y)$ is smooth in $\mathcal{U}\setminus \mathcal{P}$ due to the smoothness of $v$ there. Moreover, by Corollary~\ref{cor:extension}(iii)(iv) and the intermediate value theorem,
\begin{align*}
&\frac{\partial_{n-1,n-1} v(y)}{y_{n-1}}\rightarrow -\frac{8}{9(\nu'_{x_0})_{n-1}^3C_{x_0}^2};\quad
\frac{\partial_{n,n}v(y)}{y_{n-1}}\rightarrow \frac{8}{9(\nu'_{x_0})_{n-1}C_{x_0}^2};\\
&\frac{\partial_{n-1,n}v(y)}{y_n}\rightarrow \frac{8}{9(\nu'_{x_0})_{n-1}C_{x_0}^2},
\end{align*} 
thus $B(y)$ has a continuous extension on $\mathcal{P}$. In particular, 
\begin{equation}\label{eq:B_0}
B(0)=a^2I_{2(n-1)}, \quad a=\frac{8}{9C_0^2}.
\end{equation}
Hence, $B(y)$ is positive definite in a small neighborhood of the origin,  which implies that the linearized operator $F_{ij}(D^2v)\partial_{ij}$ has a subelliptic structure near the origin. 
Moreover, using \eqref{eq:B_0} we have $$AB(0)A^t=\begin{pmatrix}
a^2(y_{n-1}^2+y_n^2) I_{n-2} & 0\\
0& I_2
\end{pmatrix}$$
which is the coefficient matrix for the \emph{Baouendi-Grushin type operator}:
$$\mathcal{L}_a=a^2(y_{n-1}^2+y_n^2)\sum_{i=1}^{n-3}\partial^2_{i, i} +\partial^2_{n-1,n-1}+\partial^2_{n,n}.$$
This indicates us to view the linearization of $F$ in a neighborhood
of the origin as a perturbation of Baouendi-Grushin type operator.

\subsection{$L^p$ estimates}
We first recall the classical $L^p$ estimate for the Baouendi-Grushin operator. For $(x,t)\in \R^{m+n}$, $x\in \R^m$, $t\in \R^n$, the Baouendi-Grushin operator is
$$\mathcal{L}_0= \Delta_x + |x|^2 \Delta_t, \quad |x|=\sqrt{|x_1|^2+\cdots+|x_m|^2}.$$
In order to study the weak solution associated with $\mathcal{L}_0$, it is natural to consider the following function space associated with the H\"ormander vector fields
\begin{align*}
X_i&=\partial_{x_i}\quad i=1,\ldots, m;\\
X_{m+ij}&=x_i\partial_{t_j}\quad i=1,\ldots, m; \ j=1,\ldots, n.
\end{align*}
For $k,p\geq 1$, $\Omega\subset \R^{m+n}$ a bounded open subset, we define
\begin{multline}
M^{k,p}(\Omega)=\{u\in L^p(\Omega): X_{j_1}\cdots X_{j_s} u\in L^p(\Omega),\\ 1\leq s\leq k,\ j_1,\ldots, j_s\in \{1,\ldots, m+mn\}\}.
\end{multline}
By Theorem 1 in \cite{Xu}, $M^{k,p}(\Omega)$ is a separable Banach space for $1\leq p<\infty$ with the norm
$$\|u\|_{M^{k,p}(\Omega)}=\|u\|_{p,\Omega}+\sum _{1\leq s\leq k}\|X_{j_1}\cdots X_{j_s} u\|_{p,\Omega}.$$
Denote by $M_0^{k,p}(\Omega)$ the closure of $C_0^\infty(\Omega)$ in $M^{k,p}(\Omega)$. It is not hard to prove by using mollifier that if $u\in M^{k,p}(\Omega)$ and has compact support in $\Omega$, then $u\in M^{k,p}_0(\Omega)$, $1\leq p<\infty$. 

In this paper only the spaces with $k=1,2$ are involved. We list them below separately:
\begin{align*}
M^{1,p}(\Omega)=\{u\in L^p(\Omega)&: \partial_{x_i} u,\ |x|\partial_{t_j}u\in L^p(\Omega)\};\\
M^{2,p}(\Omega)=\{u\in L^p(\Omega)&: \partial_{x_i} u,\ \partial_{t_j} u \in L^p(\Omega); \\
&\quad \partial^2_{x_ix_j} u,\ |x|^2 \partial^2_{t_it_j}u,\ |x|\partial^2_{t_ix_j}u \in L^p(\Omega)\}.
\end{align*}
To simplify the notation, we will denote
\begin{align*}
\|\partial^1_w u\|_{p,\Omega}&:=\|\nabla _x u\|_{p,\Omega}+\||x|\nabla_t u\|_{p,\Omega}\\
\|\partial^2_w u\|_{p,\Omega}&:=\sum_{i,j=1,\ldots,m}\|\partial^2_{x_ix_j} u\|_{p,\Omega}+\sum_{\substack{i=1,\ldots,m\\j=1,\ldots,n}}\||x|\partial^2_{x_it_j}u\|_{p,\Omega}\\&\qquad +\sum_{i,j=1,\ldots,n}\||x|^2 \partial^2_{t_it_j}u\|_{p,\Omega}
+\sum_{j=1,\ldots,n}\|\partial_{t_j}u\|_{p,\Omega}.
\end{align*}

We will need the Sobolev embedding theorem for $M^{1,p}_0(\Omega)$ and the $L^p$ estimate for $\mathcal{L}_0$. Similar results for more general subelliptic operators can be found in lots of literature like \cite{Xu} and \cite{SCA}. Since our case is much simpler, we provide a relatively shorter and self-contained proof in the Appendix.

\begin{lemma}[Sobolev embedding for $M_0^{1,p}(\Omega)$]\label{lem:embedding}
Let $\Omega $ be a bounded domain in $\R^{m+n}$.
\begin{enumerate}[\em (i)]
\item If $1\leq p< m+2n$, then $M_0^{1,p}(\Omega) \hookrightarrow L^q(\Omega)$ for $p, q$ satisfying $\frac{1}{q}+\frac{1}{m+2n}=\frac{1}{p}$, i.e.\ there exists $C=C(p)>0$ such that for all $u\in M_0^{1,p}(\Omega)$ 
$$ \|u\|_{q,\Omega}\leq C \left (\|\nabla _x u\|_{p, \Omega}+\| |x|\nabla_t u\|_{p,\Omega}\right).$$ 
\item  If $p>m+2n$, then $M_0^{1,p}(\Omega)\hookrightarrow L^\infty(\Omega)$, i.e.\ there exists $C=C(p, n, m)>0$ such that for all $u\in M^{1,p}_0(\Omega)$, 
$$\|u\|_{\infty, \Omega}\leq C\left(\|\nabla _x
  u\|_{p,\Omega}+\||x|\nabla_t u\|_{p,\Omega}\right).$$
\end{enumerate}
\end{lemma}
\begin{proof}
See Appendix.
\end{proof}

\begin{theorem}[$L^p$ estimates for $\mathcal{L}_0$]
\label{thm:Grushin}
Let $\Omega $ be a domain in $\R^{m+n}$, $m\geq 2$ and $m$ is even. Let $u\in M^{2,p}_0(\Omega)$ with $1<p<\infty$. Then there is a positive constant $C_p$ which only depends on $p$ such that
\begin{equation}\label{eq:Grushin_a}
\|\partial_w^2 u\|_{p,\Omega}\leq C_p\|\mathcal{L}_0 u\|_{p,\Omega}.
\end{equation}
\end{theorem}
\begin{proof}
See Appendix.
\end{proof}

Next we state the local $L^p$ estimates for the perturbed operator
\begin{equation}\label{eq:defL}
\mathcal{L}:=\sum_{i,j=1,\ldots, m+n} \ell_{ij}(x,t)\partial_{ij}.
\end{equation}
where $\{\ell_{ij}(x,t)\}_{(m+n)\times (m+n)}$ can be decomposed into the form $ABA^t$ with
\begin{align*}
A=\begin{pmatrix}
I_m & 0 &\cdots &0\\
0& X & \ddots &\vdots\\
\vdots & \ddots & \ddots & 0\\
0 & \cdots & 0 & X
\end{pmatrix}_{(m+n)\times (m+mn)}, \quad  X=(x_1,\ldots, x_m)_{1\times m}
\end{align*}
and $B=(b_{ij})_{(m+mn)\times (m+mn)}$ a positive definite matrix with continuous entries and for some small positive $\delta_0$ it satisfies 
\begin{equation}
|b_{ij}(x,t)-\delta_{ij}|\leq \delta_0 \label{eq:defL3}
\end{equation}
where $\delta_{ij}=1$ if $i=j$ and $\delta_{ij}=0$ if $i\neq j$.

From now on, we will work on the following scale-invariant ``cylinder'' (w.r.t $\mathcal{L}_0$) centered at the origin: for $r>0$
$$\C_r=\{(x,t)\in \R^{m+n}: |x|<r, |t|< r^2 \}.$$
 
\begin{proposition}\label{prop:perturb-Grushin}
Let $u\in M^{2,p}(\C_{r})$ satisfy 
\begin{equation}\label{eq:perturb-Grushin1}
\mathcal{L} u=f \text{ a.e.\ in } \C_{r}, \quad f\in L^p(\C_{r}),\quad 1<p<\infty,
\end{equation}
where $\mathcal{L}$ is a perturbed Baouendi-Grushin operator given by \eqref{eq:defL}.
Then if \eqref{eq:defL3} is satisfied for sufficiently small $\delta_0=\delta_0(p,m,n)>0$, there exists $C=C(p)>0$ such that for any $\sigma\in (0,1)$,
\begin{equation}\label{eq:perturb-Grushin2}
\|\partial^2_w u\|_{p,\C_{\sigma r}}\leq \frac{C}{(1-\sigma)^2r^2}\left(r^2\|f\|_{p,\C_{r}}+\|u\|_{p,\C_{r}}\right).
\end{equation}
\end{proposition}
\begin{proof}
We write 
$$\mathcal{L}_0 u= (\mathcal{L}_0-\mathcal{L})u +\mathcal{L}u.$$
If $\supp u\Subset \C_{r}$, then by \eqref{eq:Grushin_a}
\begin{align*}
\|\partial^2_w u\|_{p, \C_{r}}&\leq C_p\left(\|\mathcal{L}u\|_{p, \C_{r}}+\|(\mathcal{L}_0-\mathcal{L})u\|_{p,\C_{r}}\right)\\
&\leq C_p \left (\|\mathcal{L}u\|_{p, \C_{r}}+\sup_{\C_{r}} | b_{ij}(x,t) - b_{ij}(0,0)| \|\partial^2_w u\|_{p, \C_{r}}\right)\\
&\leq C_p\left( \|\mathcal{L}u\|_{p,\C_{r}}+\delta_0\|\partial^2_w u\|_{p,\C_{r}}\right).
\end{align*}
Hence if we choose $\delta_0=1/(2C_p)$ in \eqref{eq:defL3}, for any $u\in M^{2,p}_0(\mathcal{C}_r)$ we have
\begin{equation}
\|\partial_w^2u\|_{p,\C_{r}}\leq \frac{C_p}{1-C_p\delta_0}\|\mathcal{L}u\|_{p,\C_{r}}\leq 2C_p\|\mathcal{L}u\|_{p,\C_{r}}.\label{eq:grushin1}
\end{equation}

Now we remove the compact support condition. Let $\sigma\in (0,1)$ be fixed and let $\sigma'=(1+\sigma)/2$. Let $\eta(x,t)=\eta_1(|x|)\eta_2(|t|)$ be a smooth cut-off function in $\C_r$, where $0\leq \eta _i\leq 1$ satisfy $\eta_1=1$ when $ |x|\leq \sigma r$, $\eta_1=0$ when $|x|\geq \sigma' r$; $\eta_2=1$ when $|t|<\sigma r^2$, $\eta_2=0$ when $|t|>\sigma' r^2$. Moreover, $|\eta'_1|\leq 2/(1-\sigma')r$, $|\eta'_2|\leq 2/(1-\sigma')r^2$, $|\eta''_1|\leq 4/(1-\sigma')^2r^2$, $|\eta''_2|\leq 4/(1-\sigma')^2r^4$. Let $v= u\eta^2$, then
\begin{align*}
\mathcal{L}v= \eta \mathcal{L}u + [\mathcal{L}, \eta^2] u.
\end{align*}
By \eqref{eq:grushin1} we have
\begin{equation}\label{eq:perturb2}
\|\partial^2_w u\|_{p,\C_{\sigma r}}\leq 2C_p\left(\|\mathcal{L}u\|_{p, \C_{\sigma' r}}+\|[\mathcal{L},\eta^2]u\|_{p,\C_{\sigma' r}}\right).
\end{equation}
Compute $[\mathcal{L},\eta^2]u$ directly. Using the estimates for the coefficient matrix $b_{ij}$ with $\delta_0$ chosen less than 1, as well as the cut-off functions, we obtain the following (for simplicity we write $\|\cdot\|_p=\|\cdot\|_{p,\C_{\sigma' r}}$)
\begin{multline}\label{eq:perturb_a}
\quad \|[\mathcal{L},\eta^2]u\|_{p}
\leq C\left(\frac{1}{(1-\sigma')r}\|\nabla_x u\|_p+\frac{1}{(1-\sigma')r^2}\||x|^2 \nabla_t u\|_p\right)+\\
+\frac{C}{(1-\sigma')^2r^2}\|u\|_p+\frac{C}{(1-\sigma')r}\||x|\nabla_t u\|_p,
\end{multline}
where $C$ is some absolute constant.

Now using the interpolation between the classical Sobolev spaces (for $r=1$) and Young's inequality we have for any $\epsilon>0$,
\begin{align*}
(1-\sigma)\|\nabla_x u\|_{p,\C_\sigma}&\leq  \epsilon(1-\sigma)^2  \|D^2_{x,x} u\|_{p,\C_\sigma}+\frac{C}{\epsilon}\|u\|_{p,\C_\sigma},\\
(1-\sigma)\||x|^2\nabla_t u\|_{p,\C_\sigma}&\leq  \epsilon(1-\sigma)^2\||x|^2D^2_{t,t} u\|_{p,\C_\sigma}+\frac{C}{\epsilon}\||x|^2u\|_{p,\C_\sigma},\\
(1-\sigma)\||x|\nabla_t u\|_{p, \C_\sigma}&\leq \epsilon(1-\sigma)^2\||x|^2 D^2_{t,t} u\|_{p, \C_\sigma}+\frac{C}{\epsilon}\|u\|_{p, \C_\sigma}.
\end{align*}
Hence by rescaling and then taking the supreme in $\sigma$, we have                                                    
\begin{multline}\label{eq:inter11}
\sup_{0<\sigma<1} (1-\sigma)r \|\nabla _x u\|_{p,\C_{\sigma r}}\\
\leq \epsilon \sup_{0<\sigma<1} (1-\sigma)^2r^2 \|D^2_{x,x} u\|_{p,\C_{\sigma r}}+ \frac{C}{\epsilon}\sup_{0<\sigma<1}\|u\|_{p,\C_{\sigma r}};
\end{multline}
\begin{multline}\label{eq:inter22}
\sup_{0<\sigma<1} (1-\sigma) \||x|^2\nabla _t u\|_{p,\C_{\sigma r}}\\
\leq \epsilon \sup_{0<\sigma<1} (1-\sigma)^2r^2 \||x|^2\nabla^2_t u\|_{p,\C_{\sigma r}}+ \frac{C}{\epsilon r^2}\sup_{0<\sigma<1}\||x|^2 u\|_{p,\C_{\sigma r}}.
\end{multline}
\begin{multline}\label{eq:inter33}
\sup_{0<\sigma<1}(1-\sigma)r\||x|\nabla_t u\|_{p,\C_{\sigma r}}\\\leq \epsilon \sup_{0<\sigma<1} (1-\sigma)^2r^2 \||x|^2\nabla^2_t u\|_{p,\C_{\sigma r}}+\frac{C}{\epsilon}\sup_{0<\sigma<1}\|u\|_{p,\C_{\sigma r}}.
\end{multline}
Combining \eqref{eq:perturb2}-\eqref{eq:inter33} and choosing $\epsilon$, $\delta_0<1/(2C_p)$ small enough, depending on $p$, we obtain
the inequality \eqref{eq:perturb-Grushin2}.
\end{proof}

\subsection{Smoothness at the free boundary points}\label{sec:smoothness-free-bound}
In this section we show that the Legendre function $v$ which satisfies the fully nonlinear PDE \eqref{eq:mainequation2} is smooth in a neighborhood of the origin. 

We will work on the non-isotropic cylinder at the origin:
\begin{align*}
\C_r=\{(y'',y_{n-1},y_n):|y''|<r^2, \sqrt{y_{n-1}^2+y_n^2}<r\}, \quad n\geq 3, r>0.
\end{align*}
Before proving the main theorem, we make the following two remarks:
\begin{enumerate}[(i)]
\item By Corollary~\ref{cor:extension} and the discussion in
  Section~\ref{sec:subelliptic-coeff}, there is $r_0>0$ small enough
  such that $v\in M^{2,p}(\C_{r_0})$ for any $1\leq p<\infty$ and the
  linearized operator $F_{ij}(D^2v)\partial_{ij}$ can be viewed as a
  perturbation of the Baouendi-Grushin type operator in $\C_{r_0}$. 

\item We note the following rescaling property: if $v$ solves \eqref{eq:mainequation2} in $\C_{r_0}$, then $\tilde{v}(y)=c_0v(\delta_r y)/r^3$ with $c_0>0$ and $\delta_r$ the non-isotropic dilation in \eqref{eq:nonisotropic} will solve
\begin{equation*}
\partial_{y_{n-1}y_{n-1}}\tilde{v}+\partial_{y_ny_n}\tilde{v}-(c_0)^{-2}\sum_{i=1}^{n-2} \det (\tilde{V}^i)=0\quad  \text{ in } \C_{r_0/r}.
\end{equation*}
Hence by multiplying a nonzero constant, we may assume that the coefficient matrix $F_{ij}(D^2v)$ is of the form $ABA^t$ in $\C_{r_0}$ with $B$ continuous and satisfying \eqref{eq:defL3} for sufficiently small $\delta_0$, where $\delta_0$ is chosen such that the $L^p$ estimate (Proposition~\ref{prop:perturb-Grushin}) applies.  
\end{enumerate}

The idea to show the smoothness is then to apply iteratively \eqref{eq:perturb-Grushin2} to the first order difference quotient of $\partial^\alpha v$, but each step we need to be careful that the non-homogeneous RHS coming from differentiation is bounded in $L^p$. 

For notation simplicity, in what follows we will discuss the case when $n=3$. Then equation \eqref{eq:mainequation2} is simply
\begin{equation}\label{eq:equ-for-v}
F(D^2v)=\partial_{22}v+\partial_{33}v-\det(D^2v)=0.
\end{equation}
The arguments for $n>3$ are the same.

\begin{theorem}\label{thm:smoothness-v}
Let $v$ be the Legendre function of $u$ defined in \eqref{eq:legendrefun}. Let $r_0>0$ such that in $\C_{r_0}$, Corollary~\ref{cor:extension} holds for $v$ and $F_{ij}(D^2v)$ can be written in the form of \eqref{eq:defL} with $b_{ij}$ satisfying \eqref{eq:defL3} for sufficiently small $\delta_0$. Then $v$ is smooth at the origin. 
\end{theorem}
\begin{proof}
We let 
$$\Delta ^h_i v:=\frac{v(\cdot+he_i)-v}{h} , \quad  h\neq 0,\quad i=1,2,3$$
denote the first difference quotient of $v$ in $e_i$ direction.  
\step{Step 1:}  Show $\partial_i v\in M^{2,p}(\C_{r})$ for any $1<p<\infty$, $0<r<r_0$. By Corollary~\ref{cor:extension}, it is enough to show it for $\partial_1 v$.

In fact, $v\in M^{2,p}(\C_{r_0})$ implies that $\Delta_1^h v\in M^{2,p}(\C_{r_0-h})$ and $\Delta_1^h v=0$ on $\C'_{r_0-h}$ for $0<h<r<<r_0$. Moreover, by taking $\Delta^h_1$ on both sides of \eqref{eq:equ-for-v} we get $\Delta_1^h v$ satisfies
\begin{equation*}
\tilde F_{ij}(D^2v , h)\partial_{ij}\Delta^h_1 v=0
\end{equation*}
with 
\begin{equation*}
\tilde F_{ij}(D^2v, h)=\int^1_0 F_{ij}(D^2v+ t(D^2v(\cdot+he_1)-D^2v))\ dt.
\end{equation*}
Since a translation in $e_1$ direction does not change the subelliptic structure of the operator (by Corollary~\ref{cor:extension} and the $C^{0,\alpha}$ dependence of $\nu_{x_0}$ on $x_0$), i.e. $\tilde F_{ij}( D^2v,h)\partial_{ij}$ is still a perturbed Baouendi-Grushin operator in the form of \eqref{eq:defL}, with $b_{ij}$ satisfying \eqref{eq:defL3} in $\C_{r_0-h}$. Then by Proposition~\ref{prop:perturb-Grushin} there exists $C=C(p)>0$ such that
for any $\sigma \in(0,1)$
\begin{equation}\label{eq:st1}
\|\partial^2_w\Delta_1^h v\|_{p, \C_{\sigma (r_0-h)}}\leq \frac{C}{(1-\sigma)^2(r_0-h)^2}\|\Delta_1^h v\|_{p,\C_{r_0-h}}.
\end{equation}
Note that the RHS of \eqref{eq:st1} is uniformly bounded in $h$. Moreover, $\|\partial^2_w \Delta_1^h v\|_p=\|\Delta_1^h (\partial^2_w v)\|_p$ and $\|\partial^2_w v\|_{p,\C_{r_0}}<\infty$ (here we slightly abuse of the notation to let $\partial^2_w v$ denote the weighted second order derivatives and first order derivatives w.r.t. $y_{n-1}$ and $y_n$). Thus $\partial_1 v\in M^{2,p}(\C_r)$ for any $0<r<r_0$ with $\|\partial_1 v\|_{M^{2,p}(\C_r)}$ depending on $r_0-r$, $r_0$ and $p$.

\step{Step 2:} Show $\partial_{11} v\in M^{2,p}(\C_r)$, $0<r<r_0$ . 

Take $\partial_1$ to both sides of \eqref{eq:equ-for-v}. From step 1, $\partial_1v\in M^{2,p}(\C_{r})$ and it satisfies
\begin{equation}\label{eq:partial-v}
F_{ij}(D^2v)\partial_{ij}\partial_1 v=0.
\end{equation}
Applying $\Delta_1^h$ to \eqref{eq:partial-v} with $0<h<<r_0/8$, then $\Delta^h_1 \partial_1 v\in M^{2,p}(\C_{r -h})$ and it satisfies
\begin{equation}\label{eq:differencequotient}
F_{ij}(D^2 v)\partial_{ij}(\Delta^h_1 \partial_1 v)=f\quad \text{in } \C_{r-h}
\end{equation}
with
\begin{equation*}
f=- \Delta ^h_1 (F_{ij}) \partial_{ij}\partial_1 v(\cdot+he_1).
\end{equation*}

To estimate the $\|f\|_{p}$, we first notice that $f$ (up to a translation $\tau_{e_1h}$) is a summation of the following terms:
\begin{align*}
&I_1=\partial_{111}v \partial_{\beta'\gamma '}v \Delta^h_1\partial_{\beta\gamma } v , \\
&I_2=\partial_{11\beta} v \left( \Delta ^h_1 \partial_{1\gamma} v \partial_{\beta'\gamma'} v+ \partial_{1\gamma} v \Delta^h_1 \partial_{\beta'\gamma'} v\right), \\
&I_3= \partial_{1\beta\gamma}v \left(\Delta ^h_1\partial_{11}v \partial_{\beta'\gamma'} v+\partial_{11} v\Delta ^h_1 \partial_{\beta'\gamma'}v \right),\quad \beta, \gamma, \beta', \gamma'=2, 3.
\end{align*}

Next, since $(y_2^2+y_3^2)^{-1/2}\in L^s( \C_r\cap \{y_1=\text{const}\})$ for any $s\in (1,2)$, then by H\"older's inequality, for $q'\in (1,2)$ satisfying $\frac{1}{q'}=\frac{1}{p}+\frac{1}{s}$ we have
\begin{align}
&\|(y_2^2+y_3^2)^{1/2}\partial_{111}v\|_{q',\C_r}\leq \|(y_2^2+y_3^2)^{-1/2}\|_{s,\C_r}\|(y_2^2+y_3^2)\partial_{111}v\|_{p,\C_r}\label{eq:hold1};\\
&\|\partial_{11\beta} v\|_{q', \C_{r}}\leq \|(y_2^2+y_3^2)^{-1/2}\|_{s,\C_r}\|(y_2^2+y_3^2)^{1/2} \partial_{11\beta} v\|_{p, \C_{r}},\quad \beta=2,3 \label{eq:hold2}.
\end{align}

Apply H\"older to estimate $I_1$. For some $q$ satisfying $1/q=1/q'+1/p$ we have
\begin{multline}
\|I_1\|_{q,\C_{r-h}}\leq \label{eq:estimate-I1}\\ \left\|(y_2^2+y_3^2)^{1/2} \partial_{111}v\right\|_{q',\C_{r-h}} \left\|(y_2^2+y_3^2)^{-1/2}\partial_{\beta'\gamma'} v\right\|_{\infty, \C_{r-h}} \left \|\Delta^h_1\partial_{\beta\gamma} v\right\|_{p,\C_{r-h}}.
\end{multline}
By Corollary~\ref{cor:extension}(iv),  the second term on the RHS of \eqref{eq:estimate-I1} is bounded.  From the boundedness of $\|\partial_w^2\partial_1 v\|_{p,\C_r}$ shown in step 1, the third term is uniformly bounded in $h$. Hence combining \eqref{eq:hold1} we have $\|I_1\|_{q,\C_{r-h}}$ is uniformly bounded in $h$. Similarly by using Corollary~\ref{cor:extension}, \eqref{eq:hold2} and step 1, we have $\|I_2\|_{q,\C_{r-h}}$, $\|I_3\|_{q,\C_{r-h}}$ are uniformly bounded in $h$. Therefore,  applying the $L^p$ estimate (Proposition~\ref{prop:perturb-Grushin}) to \eqref{eq:differencequotient}, one can find a constant $C_q$ independent of $h$, such that for any $\sigma\in (0,1)$,
\begin{equation*}
\|\partial^2_w(\Delta ^h_1 \partial_1 v)\|_{q, \C_{\sigma(r-h)}}
\leq \frac{C _q}{(1-\sigma)^2}\|f\|_{q, \C_{r-h}}+\frac{C_q}{(r-h)^2}\|\Delta^h_1 \partial_1 v\|_{q, \C_{r-h}}.
\end{equation*}
Since the RHS is uniformly bounded in $h$, this implies
\begin{equation}\label{eq:st2}
\|\partial^2_w \partial_{11} v\|_{q,\C_{\sigma r}}\leq C_{q,\sigma, r}.
\end{equation}

From \eqref{eq:st2} we know
$(y_2^2+y_3^2)^{1/2}\partial_{111}v$, $\partial_{11\beta}v \in M^{1,q}(\mathcal{C}_{\sigma r})$, $\beta=2,3$.
Multiplying a cut-off function to extend the functions to $\R^n$ and applying the Sobolev embedding lemma~\ref{lem:embedding}(i) we have, for $q_1=4q/(4-q)$ (with $4$ the homogeneous dimension associated with $F_{ij}(D^2v)\partial_{ij}$),
$$\|(y_2^2+y_3^2)^{1/2}\partial_{111}v\|_{q_1, \C_{\sigma r}},~ \|\partial_{11\beta}v\|_{q_1, \C_{\sigma r}}\leq C_{q,\sigma,r}.$$
Repeat the above arguments starting from \eqref{eq:estimate-I1} with $q'$ replaced by $q_1$ (note $q_1>q'$ if $p>4$). After finite steps $m$ (which only depends on the dimension) we will get
$(y_2^2+y_3^2)^{1/2}\partial_{111}v$, $\partial_{11\beta}v\in M^{1,q}(\C_{\sigma^m r})$  with $q$ larger than the homogeneous dimension $4$, and hence by embedding lemma~\ref{lem:embedding}(ii) are in $L^\infty(\C_{\sigma^mr})$. Applying Proposition~\ref{prop:perturb-Grushin} again we obtain $\|\partial^2_w\partial_{11} v\|_{p,\C_{\sigma^m r}}< C_{p,\sigma, r, m}$ for $1<p<\infty$.
Noting that $r\in (0,r_0)$ is chosen arbitrary, we complete the proof for step 2.

\step{Step 3.}  Show $\partial^\alpha v\in M^{2,p}(\C_r)$, $|\alpha|=2$ with $\alpha_2+\alpha_3\geq 1$.

First from step 1, $\|\partial^\alpha v\|_{p, \C_r}<C_{p,r,r_0}$ for $|\alpha|=3$ with $\alpha_2+\alpha_3\geq 2$. This together with the boundedness of $\|\partial^2_w\partial_{11}v\|_{p,\C_{r}}$ obtained in step 2 gives 
\begin{equation}\label{eq:threediff}
\|\partial^{\alpha} v\|_{p, \C_{r}}<C_{p,r,r_0}\quad \text{for all }\alpha \text{ with }|\alpha|=3.
\end{equation}

Next we estimate $\Delta^h_k\partial_\beta v$ with $k=1,2,3$ and $\beta=2,3$. Similar as \eqref{eq:differencequotient}, $\Delta^h_k\partial_\beta v$ satisfies
$$F_{ij}\partial_{ij}(\Delta^h_k\partial_\beta v)= f$$ with
$$f= -\Delta_k^h(F_{ij})\partial_{ij}\partial_\beta v(\cdot+he_k).$$
By \eqref{eq:threediff} we immediately have $\|f\|_{p,\C_{ r-h}}$ is uniformly bounded in $h$ for any $1<p<\infty$. Applying the $L^p$ estimate (Proposition~\ref{prop:perturb-Grushin}) we have $\|\partial^2_w(\Delta^h_k\partial_\beta v)\|_{p,\C_{ r-h}}$ is uniformly bounded in $h$, and therefore completes the proof for step 3.

\step{Step 4.} Show $\partial^\alpha v\in M^{2,p}(\C_{r})$ for $|\alpha|=k>2$, with $\|\partial^\alpha v\|_{M^{2,p}(\C_r)}$ depending on $p,r,r_0,|\alpha|$ and dimension. 

This is done in the same way as for $|\alpha|=2$. More precisely, we first consider the equation of $\Delta ^h_1 \partial_1^{k-1} v$ in $\C_{r}$
$$F_{ij}(D^2v)\Delta^h_1(\partial_1^{k-1} v) + l.o.t. = 0$$
with $\|l.o.t.\|_{q, \C_{r}}$ bounded uniformly in $h$ for some $q\in (1,2)$. Then we apply the Sobolev embedding lemma and the $L^p$ estimate iteratively to obtain the boundedness of $\|l.o.t.\|_{\infty, \C_{r'}}$ with some $r'\in (0,r)$, as well as the boundedness of $\|\partial_w^2\partial^k_1 v\|_{p, \C_{r'}}$ for any $1<p<\infty$. In particular, this combined with the fact that $\partial^\alpha v\in M^{2,p}(\C_r)$ for all $|\alpha|\leq k-1$ gives that $\|\partial^\alpha v\|_{p, \C_{r'}}$ is bounded with $|\alpha|=k+1$ for any $1<p<\infty$. Next, we consider the equation for $\Delta_j^h\partial^\alpha v$, $j=1,2,3$, $|\alpha|=k-1$ and $\alpha_2+\alpha_3\geq 1$. Similar as in step 3, one easily get the uniform boundedness of $\|l.o.t\|_{p,\C_{r'}}$ for any $1<p<\infty$ due to the boundedness of $\| \partial^\alpha v\|_{p,\C_{r'}}$, $|\alpha |=k+1$. Applying $L^p$ estimate again we obtain the conclusion.
\end{proof}

\begin{corollary}\label{cor:freeboundary}
Let $u$ be a solution to the Signorini problem in $B_2^+$, and let $\Gamma_u$ be the regular set of the free boundary. Then $\Gamma_u$ is smooth. 
\end{corollary}
\begin{proof}
Take $x_0\in \Gamma_u$ a regular free boundary point, in a coordinate chart centered at $x_0$, by \cite{ACS} $\Gamma_u$ can be locally expressed as the graph of a $C^{1,\alpha}$ function $y_2=f(y_1)$ with $f(0)=f'(0)=0$. Consider in a neighborhood of $x_0$ the partial hodograph-Legendre transform and the corresponding Legendre function $v$ defined in section~\ref{sec:legendre-transf-its}. By Theorem~\ref{thm:smoothness-v}, $v$ is smooth at the origin. Since
$$f(y_1)=-\partial_2 v(y_1,0,0), \quad f(0)=-\partial_2 v(0).$$
Hence the smoothness of $v$ at the origin implies the smoothness of $f$ at $0$.
\end{proof}

\section{Real Analyticity}
\label{sec:analyticity}

In this section, we show that the Legendre transform $v$ is analytic in a neighborhood of the origin (Theorem~\ref{thm:fb-regul}). 

\begin{theorem}\label{thm:main} Let $v$ be the Legendre transform
  defined in Section~\ref{sec:legendre-funct-nonl}. Then $v$ is real analytic in a neighborhood of the origin.
\end{theorem} 

We first make some more assumptions and observations.\\
1. For simplicity we work on $\R^3$. By the scaling invariant property mentioned at the beginning of Section~\ref{sec:smoothness-free-bound} we may assume $v$ is smooth in $\C_2$ and solves the fully nonlinear equation \eqref{eq:equ-for-v} there. Denote
$$M_0:=\sup_{|\alpha|\leq 4}\|\partial^2_w\partial^\alpha v\|_{\infty, \C_2}.$$
From Corollary~\ref{cor:extension}(iii)(iv) and the intermediate value theorem, we have
\begin{equation}\label{eq:ss}
\left\| \frac{\partial_{\beta\beta} v}{y_2}\right\|_{\infty, \C_2},\ \left\| \frac{\partial_{1\beta\beta} v}{y_2}\right\|_{\infty, \C_2},\ \left\| \frac{\partial_{23} v}{y_3}\right\|_{\infty, \C_2},\ \left\| \frac{\partial_{123} v}{y_3}\right\|_{\infty, \C_2}\leq M_0, \quad \beta=2,3.
\end{equation}\\
2. The following multi-index notation will be used: \\For a multi-index $\alpha=(\alpha_j)_{j=1,\ldots, n}\in \Z_+^n$, $$\alpha ! =\prod_{j=1}^n \alpha_j!\quad (0!=1)$$ Let $\alpha=(\alpha_j)$, $\beta=(\beta_j)$ be two multi-index in $\Z_+^n$, we say $\alpha\leq \beta$ iff
$\alpha_j\leq \beta_j$, $j=1,\ldots, n$;  and $\alpha<\beta$ iff $\alpha \leq \beta $, $|\alpha|<|\beta|$, where $|\alpha|=\sum^n_1\alpha_j$.\\
3. The strategy to prove the analyticity is as follows: Given $\alpha \in \Z_+^3$, taking $\partial^\alpha$ on both sides of \eqref{eq:equ-for-v} and using the summation convention on $i, j$s, we obtain that $\partial^\alpha v$ satisfies
\begin{multline}\label{eq:deri-equ}
F_{ij}(D^2 v) \partial_{ij} \partial^\alpha v\\
+\sum_{\substack{\alpha^1+\alpha^2+\alpha^3=\alpha\\\alpha^i<\alpha,i=1,2,3}} \frac{\alpha !}{\alpha^1 !\alpha^2 !\alpha^3 !} \sum_{\sigma\in S_3}\sgn(\sigma)(\partial ^{\alpha^1} \partial_{1\sigma_1} v) (\partial^{\alpha^2} \partial_{2\sigma_2} v) (\partial^{\alpha^3}\partial_{3\sigma _3}v)=0,
\end{multline}
where $S_3$ is the set of all permutations of $\{1,2,3\}$.
We will apply Proposition~\ref{prop:perturb-Grushin} for \eqref{eq:deri-equ} to get a fine estimate of the $L^p$ norm of $\partial^\alpha v$.  In order to do so, usually one needs a sequence of domains with properly shrinking radius as well as the corresponding sequence of cut-off functions. In this paper, we use  the trick introduced in \cite{Kato} to avoid this technical trouble. In the following we take and fix a cut-off function $\eta\in C^\infty(\C_2)$ which satisfies
\begin{align*}
&\eta(y)=\eta^1(y_1)\eta^2(y_2,y_3), \quad 0\leq \eta^i\leq 1,\\
&\eta^1=\eta^2=1 \text{ if } |y_1|\leq 1, \sqrt{y_2^2+y_3^2}\leq 1,\\
&\eta^1=\eta^2=0 \text{ if } |y_1|>\frac{3}{2}, \sqrt{y_2^2+y_3^2}>\frac{3}{2},\\
&|\partial \eta^i|, |\partial ^2 \eta^i|\leq C_\eta, \quad i=1,2.
\end{align*}
We will estimate $\|\eta^{|\alpha|-2}\partial^2_w\partial^\alpha v\|_{p, \C_{3/2}}$ with $|\alpha|=k$ by $\|\eta^{|\alpha|-2}\partial^2_w\partial^\alpha v\|_{p, \C_{3/2}}$ with $|\alpha|<k$. \\
4. From now on we will simply write $\|\cdot\|_p$ if the integral domain is $\C_{3/2}$. We will fix a $p$ larger than the homogeneous dimension $4$. By a universal constant we mean an absolute constant which only depends on $C_\eta$, $M_0$, dimension (which is $3$ in our setting) or $p$ chosen, in particular independent of $k$.\\
5. The following observation will be useful in the proof:  
\begin{align*}
\supp (\partial_j \eta)\subset \{y:y_2^2+y_3^2>1\}\quad \text{ for }j=2,3,
\end{align*}
which implies
$$|\partial_j \eta(y)|\leq (y_2^2+y_3^2)^{1/2} C_\eta\quad \text{ for }j=2,3 . 
$$
Hence for multi-index $\alpha$ with $|\alpha|=k\geq 4$, 
\begin{align}
&\left|\|\partial^1_w(\eta^{|\alpha|-2}\partial ^\alpha u)\|_p-\|\eta^{|\alpha|-2}\partial^1_w \partial^\alpha u\|_p\right|\label{eq:remark1}\\
&\qquad\leq|\alpha|C_\eta \|\eta^{|\alpha|-3}(y_2^2+y_3^2)^{1/2}\partial^\alpha v\|_p \notag\\
&\left|\|\partial ^2_w(\eta^{|\alpha|-2} \partial^\alpha u)\|_p-\|\eta^{|\alpha|-2} \partial^2_w \partial^\alpha u\|_p\right| \label{eq:important-remark}\\
&\qquad\leq 2|\alpha|C_\eta\|\eta^{|\alpha'|-2}\partial^2_w\partial^{\alpha'} u\|_p+|\alpha|^2C_\eta\|\eta^{|\alpha''|-2}\partial^2_w\partial^{\alpha''} u\|_p\notag
\end{align}
for some $\alpha'$, $\alpha''$ with $|\alpha'|=k-1$ and $|\alpha''|=k-2$.

The following proposition is the main proposition of this section.
\begin{proposition}\label{prop:induction}
There exist universal constants $R$, $0<R<1$ and $\overline{C}>2$ such that for any $k\geq 4$
\begin{align*}
 \textup{(i)}\quad & \|\eta^{k-2} \partial^2 _w \partial^k_1 v\|_p\leq R^{-(k-4)}k^{k-4}; \\
 \textup{(ii)}\quad &\|\eta^{k-2}\partial^2_w \partial ^\alpha v\|_p\leq
 R^{-(k-3)}k^{k-3},\quad\text{for all }
\alpha \text{ with } |\alpha|=k  \text{ and } \alpha_2+\alpha_3\geq 1.
\end{align*}
\end{proposition}

Theorem~\ref{thm:main} will follow from Proposition~\ref{prop:induction}. Indeed, by Proposition~\ref{prop:induction} there exists a universal $R>0$ such that  
$$\|\partial^{\alpha} v\|_{p, \C_1}\leq R^{-|\alpha|} |\alpha|^{|\alpha|}, \quad |\alpha |\geq 4.$$
Hence by the classical Sobolev embedding $W^{1,p}\hookrightarrow L^\infty$ for $p>4$ chosen above, one has 
\begin{equation}\label{eq:Taylor2}
\sup_{y\in \C_1}|\partial^{\alpha} v(y)|\leq C R^{-|\alpha|} |\alpha|^{|\alpha|}.
\end{equation}
Hence $v$ is in Gevrey class $G^1$, which is the same as the class of
real analytic functions.

Before proving Proposition~\ref{prop:induction}, we first show a lemma on the $L^\infty$-norm of $\partial^\alpha v$, which roughly speaking is a consequence of the Sobolev embedding lemma (Lemma~\ref{lem:embedding}(ii)). 

\begin{lemma}\label{lem:embedding2}
For $k\geq 5$, assume Proposition~\ref{prop:induction} holds for $k-1$ and $k$. Then there is a universal constant $C_2>M_0>0$ such that
\begin{align*}
 \textup{(i)} \quad & \|\eta^{k-2} \partial ^{k}_1 v\|_\infty,~ \|\eta^{k-2}\partial ^k_1 \partial_{\beta} v\|_{\infty}\leq C_2R^{-(k-4)}k^{k-4},\quad \beta=2,3;\\
\textup{(ii)}\quad &\|\eta^{k-2}\partial^\alpha v\|_{\infty}\leq C_2R^{-(k-3)}k^{k-3}, \quad |\alpha|=k+1,~ \alpha_2+\alpha_3\geq 2;\\
 \textup{(iii)} \quad &\|\eta^{k-2}\partial_1^{k-2}\partial_{\beta\beta}v/y_2\|_{\infty},\|\eta^{k-2}\partial_1^{k-2}\partial_{\beta\gamma} v / y_3\|_{\infty}\leq C_2R^{-(k-3)}k^{k-3}, \\
 &\text{for }\beta,\gamma=2,3 \text{ and }\beta\neq \gamma.
\end{align*}
\end{lemma}

\begin{proof}
(i) First by Lemma~\ref{lem:embedding2}(ii), there is a universal constant $C$ such that 
$$\|\eta^{k-2}\partial^k_1\partial_\beta v\|_\infty \leq C\|\partial^1_w (\eta^{k-2}\partial^k_1\partial_\beta v)\|_p, \quad \beta=2,3.$$
Applying \eqref{eq:remark1} to the RHS of above inequality, we have 
\begin{align*}
\|\eta^{k-2}\partial^k_1\partial_\beta v\|_\infty &\leq C\left(\|\eta^{k-2}\partial^1_w\partial^k_1\partial_\beta v\|_p+C_\eta k\|\eta^{k-3}\partial^2_w \partial^{k-1}_1 v\|_p\right)\\
&\leq C\left(\|\eta^{k-2}\partial^2_w\partial_1^k v\|_p+C_\eta k\|\eta^{k-3}\partial^2_w \partial^{k-1}_1 v\|_p\right)
\end{align*}
By Proposition~\ref{prop:induction}(i) for $k-1$ and $k$, there exists a universal constant $C_2$, which can be chosen larger than $M_0$ such that
\begin{equation}\label{eq:tem}
\|\eta^{k-2}\partial ^k_1 \partial_{\beta} v\|_{\infty}\leq C_2R^{-(k-4)}k^{k-4}.
\end{equation}
The estimate for $\|\eta^{k-2} \partial ^{k}_1 v\|_\infty$ follows from the classical Sobolev embedding, \eqref{eq:remark1}, \eqref{eq:tem} and the assumption.

(ii) Similar to (i). 

(iii) Use Corollary~\ref{cor:extension} and (ii) above.
\end{proof}

\begin{proof}[Proof of Proposition~\ref{prop:induction}]
This is done by induction. Assume (i) and (ii) hold for $4, \ldots, k-1$. We want to show they hold for $k$.

Let $\alpha$ be a multi-index with $|\alpha|=k$. From \eqref{eq:deri-equ}, $\eta^{k-2} \partial ^\alpha v$ satisfies the following equation
\begin{equation*}
F_{ij} \partial_{ij} (  \eta^{k-2} \partial^\alpha v)= I+II,
\end{equation*}
where
\begin{align*}
I&=-\sum_{\substack{\alpha^1+\alpha^2+\alpha^3=\alpha\\0\leq\alpha^i<\alpha,i=1,2,3}} \frac{\alpha !}{\alpha^1 !\alpha^2 !\alpha^3 !} \sum_{\sigma\in S_3}\sgn(\sigma)(\partial ^{\alpha^1} \partial_{1\sigma_1} v) (\partial^{\alpha^2} \partial_{2\sigma_2} v) (\partial^{\alpha^3}\partial_{3\sigma _3}v)\eta^{k-2}\\
II&=2F_{ij} \partial_i (\eta^{k-2})\partial _j\partial ^\alpha v+ F_{ij}\partial _{ij} (\eta^{k-2})\partial^\alpha v.
\end{align*}
By the $L^p$ estimate for the perturbed Baouendi-Grushin operator (Proposition~\ref{prop:perturb-Grushin}), there is a universal $C_3>0$ such that
\begin{equation}\label{eq:sumI}
\|\partial ^2_w (\eta^{k-2}\partial ^\alpha v)\|_{p}\leq C_3\left (\| I \|_{p}+ \|II \|_{p}+\|\eta^{k-2}\partial^\alpha v\|_p\right).
\end{equation}
 
The estimate of $\|II\|_p$ is standard. In fact, 
\begin{equation*}
\|\sum_{i,j}F_{ij} \partial_i (\eta^{k-2}) \partial _j \partial ^\alpha v\|_p
\leq 2M_0^2 (k-2)C_\eta\|\eta^{k-3}\partial_w^2 (\partial ^{\tilde{\alpha}} v)\|_p,
\end{equation*}
for some $\tilde{\alpha}<\alpha$ with $|\tilde{\alpha}|=k-1$. 
By the induction assumption (i)(ii) for $k-1$, the above RHS is bounded by
\begin{align*}
 2 M_0^2(k-2)C_\eta\begin{cases} 
                              R^{-(k-5)}(k-1)^{k-5} & \text{ if } \alpha'=(k-1,0,0)\\
                              R^{-(k-4)}(k-1)^{k-4} & \text{ otherwise}.
\end{cases}
\end{align*}
Similarly there is some $\tilde{\tilde{\alpha}}<\alpha$ with $|\tilde{\tilde{\alpha}}|=k-2$ such that
\begin{align*}
&\| \sum_{i,j}F_{ij}\partial _{ij}( \eta^{k-2})\partial^\alpha v\|_p
\leq 2M^2(k-2)(k-3)C_\eta\|\eta^{k-4}\partial^2_w(\partial^{\tilde{\tilde{\alpha}}}v)\|_p\\
&\qquad \leq 2M_0^2(k-2)(k-3)C_\eta \begin{cases}
                                        R^{-(k-6)}(k-2)^{k-6} & \text{ if } \alpha''=(k-2,0,0)\\
                                        R^{-(k-5)}(k-2)^{k-5} & \text{ otherwise}.
\end{cases}
\end{align*}
Hence 
\begin{align}
\|II\|_p&\leq \|2\sum_{i,j}F_{ij} \partial_i (\eta^{k-2})\partial _j \partial ^\alpha v\|_p+\| \sum_{i,j}F_{ij}\partial _{ij} (\eta^{k-2})\partial^\alpha v\|_p \notag\\
&\leq 6M_0^2C_\eta \begin{cases} 
                                    R^{-(k-5)} k^{k-4}  & \text{ if } \alpha=(k,0,0)\\
                                    R^{-(k-4)} k^{k-3} & \text{ otherwise}. 
\end{cases}\label{eq:IIp}
\end{align}

Next we estimate $\|I\|_p$.

\step{Proof of (i).}
Let $\ell_1=|\alpha_1|$, $\ell_2=|\alpha_2|$, $\ell_3=|\alpha_3|$. Then
\begin{equation*}
\|I\|_p\leq \sum_{\substack{\ell_1+\ell_2+\ell_3=k\\0\leq \ell_i<k}}\frac{k!}{\ell_1 !\ell_2 ! \ell_3!}\sum_{\sigma\in S_3}\left\|(\partial^{\ell_1}_1\partial_{1\sigma_1}v)(\partial^{\ell_2}_1\partial_{2\sigma_2}v)(\partial^{\ell_3}_1\partial_{3\sigma_3} v)\eta^{k-2}\right\|_p.
\end{equation*}
We discuss the following two cases:
\case {Case 1.} $\ell_1\geq\max\{\ell_2,\ell_3\}$. 

If $\sigma_1=1$, then 
\begin{align*}
&\left\|(\partial^{\ell_1}_1\partial_{1\sigma_1}v)(\partial^{\ell_2}_1\partial_{2\sigma_2}v)(\partial^{\ell_3}_1\partial_{3\sigma_3} v)\eta^{k-2}\right\|_p\\
&\qquad\leq \left\|\eta^{\ell_1-2}(y_2^2+y_3^2)\partial^{\ell_1}_1\partial_{11}v\right\|_p \left\|\eta^{\ell_2} \frac{\partial^{\ell_2}_1\partial_{2\sigma_2} v}{(y_2^2+y_3^2)^{1/2}}\right\|_\infty \left\|\eta^{\ell_3} \frac{\partial^{\ell_3}_1\partial_{3\sigma_3} v}{(y_2^2+y_3^2)^{1/2}}\right\|_\infty\\
&\qquad\leq \left\|\eta^{\ell_1-2}\partial^2_w\partial^{\ell_1}_1 v\right\|_p\left\|\eta^{\ell_2} \frac{\partial^{\ell_2}_1\partial_{2\sigma_2} v}{(y_2^2+y_3^2)^{1/2}}\right\|_\infty \left\|\eta^{\ell_3} \frac{\partial^{\ell_3}_1\partial_{3\sigma_3} v}{(y_2^2+y_3^2)^{1/2}}\right\|_\infty,
\end{align*}
which is by Lemma~\ref{lem:embedding2}(iii) and the induction assumption (i) for $k-1$ bounded by
\begin{equation*}
R^{-(\ell_1-4)}\ell_1^{\ell_1-4} C_2R^{-(\ell_2-1)}\ell_2^{\ell_2-1}C_2R^{-(\ell_3-1)}\ell_3^{\ell_3-1}= (C_2)^2R^{-(k-6)}\ell_1^{\ell_1-4}\ell_2^{\ell_2-1}\ell_3^{\ell_3-1}.
\end{equation*}

If $\sigma_1\neq 1$, then 
\begin{align*}
&\left\|(\partial^{\ell_1}_1\partial_{1\sigma_1}v)(\partial^{\ell_2}_1\partial_{2\sigma_2}v)(\partial^{\ell_3}_1\partial_{3\sigma_3} v)\eta^{k-2}\right\|_p\\
&\qquad\leq\left\|\eta^{\ell_1-2}(y_2^2+y_3^2)^{1/2}\partial^{\ell_1}_1\partial_{1\sigma_1} v\right\|_p\left\|\eta^{\ell_2+\ell_3}\frac{(\partial_1^{\ell_2}\partial_{2\sigma_2}v)(\partial_1^{\ell_3}\partial_{3\sigma_3}v)}{(y_2^2+y_3^2)^{1/2}}\right\|_\infty,
\end{align*}
which is by Lemma~\ref{lem:embedding2}(iii)(i) and the induction assumption (i) for $k-1$ bounded by
\begin{equation*}
(C_2)^2R^{-(k-7)}\ell_1^{\ell_1-4}\ell_2^{\ell_2-1}\ell_3^{\ell_3-1}.
\end{equation*}
Hence 
\begin{align*}
&\sum_{\substack{\ell_1+\ell_2+\ell_3=k\\0\leq \ell_i<k\\\ell_1\geq \max\{\ell_2,\ell_3\}}}\frac{k!}{\ell_1 !\ell_2 ! \ell_3!}\sum_{\sigma\in S_3}\left\|(\partial^{\ell_1}_1\partial_{1\sigma_1}v)(\partial^{\ell_2}_1\partial_{2\sigma_2}v)(\partial^{\ell_3}_1\partial_{3\sigma_3} v)\eta^{k-2}\right\|_p\\
&\qquad \leq \sum_{\substack{\ell_1+\ell_2+\ell_3=k\\0\leq \ell_i<k\\\ell_1\geq \max\{\ell_2,\ell_3\}}}\frac{k!}{\ell_1 !\ell_2 ! \ell_3!}6(C_2)^2R^{-(k-6)}\ell_1^{\ell_1-4}\ell_2^{\ell_2-1}\ell_3^{\ell_3-1}.
\end{align*}
By Stirling's formula and the fact that $\ell_3\geq \max\{\ell_1,\ell_2\}\implies \ell_3\geq k/3$, there is a universal constant $C_4>0$ such that
$$\sum_{\substack{\ell_1+\ell_2+\ell_3=k\\0\leq \ell_i<k\\\ell_1\geq \max\{\ell_2,\ell_3\}}}\frac{k!}{\ell_1 !\ell_2 ! \ell_3!}\ell_1^{\ell_1-4}\ell_2^{\ell_2-1}\ell_3^{\ell_3-1}\leq C_4k^{k-4}.$$
Hence
\begin{align*}
&\sum_{\substack{\ell_1+\ell_2+\ell_3=k\\0\leq \ell_i<k\\\ell_1\geq \max\{\ell_2,\ell_3\}}}\frac{k!}{\ell_1 !\ell_2 ! \ell_3!}\sum_{\sigma\in S_3}\left\|(\partial^{\ell_1}_1\partial_{1\sigma_1}v)(\partial^{\ell_2}_1\partial_{2\sigma_2}v)(\partial^{\ell_3}_1\partial_{3\sigma_3} v)\eta^{k-2}\right\|_p\\
&\qquad\leq 6(C_2)^2C_4R^{-(k-6)}k^{k-4}.
\end{align*}

\case{Case 2.} $\ell_2\geq\max\{\ell_1,\ell_3\}$ or $\ell_3\geq\max\{\ell_1,\ell_2\}$. 

We only discuss when $\ell_2\geq \max\{\ell_1, \ell_3\}$. Similarly we consider $\sigma_2=1$ and $\sigma_2\neq 1$. If $\sigma_2=1$, the estimate is indeed included in Case 1. If $\sigma_2\neq 1$, then we simply have
\begin{align*}
&\left\|(\partial^{\ell_1}_1\partial_{1\sigma_1}v)(\partial^{\ell_2}_1\partial_{2\sigma_2}v)(\partial^{\ell_3}_1\partial_{3\sigma_3} v)\eta^{k-2}\right\|_p\\
&\qquad\leq\left\|\eta^{\ell_1}\partial^{\ell_1}_1\partial_{1\sigma_1}v\right\|_\infty \left\|\eta^{\ell_2}\partial^{\ell_2}_1\partial_{2\sigma_2}v\right\|_p\left\|\eta^{\ell_3-1}\partial^{\ell_3}_1\partial_{3\sigma_3} v\right\|_\infty\\
&\qquad\leq (C_2)^2R^{-(k-8)}\ell_1^{\ell_1-2}\ell_2^{\ell_2-4}\ell_3^{\ell_3-2}.
\end{align*}

Arguing similarly as in Case 1 we have 
\begin{align*}
&\sum_{\substack{\ell_1+\ell_2+\ell_3=k\\0\leq \ell_i<k\\\ell_1\leq \max\{\ell_2,\ell_3\}}}\frac{k!}{\ell_1 !\ell_2 ! \ell_3!}\sum_{\sigma\in S_3}\left\|(\partial^{\ell_1}_1\partial_{1\sigma_1}v)(\partial^{\ell_2}_1\partial_{2\sigma_2}v)(\partial^{\ell_3}_1\partial_{3\sigma_3} v)\eta^{k-2}\right\|_p\\
&\qquad \leq 12(C_2)^2C_4R^{-(k-6)}k^{k-4}.
\end{align*}

Combining the above two cases we have, 
\begin{equation}\label{eq:firstI1}
\|I\|_p\leq 18(C_2)^2 C_4 R^{-(k-6)}k^{k-4}.
\end{equation}

We combine \eqref{eq:sumI}, \eqref{eq:IIp} for $\alpha=(k,0,0)$ and \eqref{eq:firstI1}, and use the induction assumption (i) for $\alpha=(k-1,0,0)$ to estimate $\|\eta^{k-2} \partial^\alpha v\|_p$. Then
\begin{align}
&\|\partial^2_w (\eta^{k-2} \partial^\alpha v)\|_p\label{eq:temp11}\\
&\qquad\leq C_3\left(18(C_2)^2C_4R^{-(k-6)}k^{k-4}+6M_0^2C_\eta R^{-(k-5)}k^{k-4}+(k-1)^{k-5}R^{-(k-5)}\right)\notag\\
&\qquad\leq C_5 R^{-(k-5)}k^{k-4},\notag
\end{align}
where $C_5=C_3(18(C_2)^2C_4+6M_0^2C_\eta+1)$ is a universal constant.
By \eqref{eq:important-remark} for $\alpha=(k,0,0)$,
\begin{align*}
&\|\eta^{k-2}\partial^2_w\partial^k_1 v\|_p \\
&\qquad\leq\|\partial^2_w (\eta^{k-2} \partial^k_1 v)\|_p+2C_\eta k\|\eta^{k-3}\partial^2_w\partial^{k-1}_1v\|_p+ C_\eta k^2\|\eta^{k-4}\partial^2_w\partial^{k-2}_1 v\|_p.
\end{align*}
Thus combining \eqref{eq:temp11} with induction assumption (i) for $k-1$ and $k-2$, we have
\begin{align}
&\|\eta^{k-2}\partial^2_w\partial^k_1 v\|_p \label{eq:estimate1}\\
&\qquad\leq C_5R^{-(k-6)}k^{k-4}+2C_\eta R^{-(k-5)}k(k-1)^{k-5}+C_\eta R^{-(k-6)}k^2(k-2)^{k-6}\notag\\
&\qquad\leq C_6R^{-(k-5)}k^{k-4}\notag,
\end{align}
where $C_6=C_5+3C_\eta$. Choosing $R=C_6^{-1}$ we proved (i).

\step{Proof of (ii).} The key step is to estimate $\|I\|_p$, which is done similarly as for (i). 

More precisely, Lemma~\ref{lem:embedding2}(i) and \eqref{eq:estimate1} (with $R=C_6^{-1}$) imply
$$\|\eta^{k-2}\partial_1^{k}\partial_\ell v\|_p\leq C_2 R^{-(k-4)}k^{k-4}, \quad \ell=1,2,3.$$
This together with the induction assumption (ii) up to $k-1$ yields, for any multi-index $\alpha$, $3\leq |\alpha|\leq k+1$,
\begin{equation}\label{eq:proofii}
\|\eta^{|\alpha|-3}\partial^\alpha v\|_p\leq C_2R^{-(|\alpha|-5)}(|\alpha|-1)^{|\alpha|-5}.
\end{equation}

In the estimate of $\|I\|_p$, we first consider when $|\alpha^1|\geq \max\{|\alpha^2|,|\alpha^3|\}$: by Lemma~\ref{lem:embedding2}(ii) and \eqref{eq:proofii},
\begin{align*}
&\left\|(\partial^{\alpha^1}\partial_{1\sigma_1} v)(\partial^{\alpha^2}\partial_{2\sigma_2}v)(\partial^{\alpha^3}\partial_{3\sigma_2}v)\eta^{k-2}\right\|_p\\
&\qquad\leq  \left\|\eta^{|\alpha^1|-1}\partial^{\alpha_1}\partial_{1\sigma_1}v\right\|_p \left\|\eta^{|\alpha^2|-1}\partial^{\alpha^2}\partial_{2\sigma_2} v\right\|_\infty \left\|\eta^{|\alpha^3|-1}\partial^{\alpha^3}\partial_{3\sigma_3} v\right\|_\infty\\
&\qquad\leq (C_2)^3R^{-(k-7)}\left(|\alpha^1|+1\right)^{|\alpha^1|-3}\left(|\alpha^2|\right)^{|\alpha^2|-2}\left(|\alpha^3|\right)^{|\alpha^3|-2}.
\end{align*}
Hence
\begin{align}
&\sum_{\substack{\alpha^1+\alpha^2+\alpha^3=\alpha\\0\leq |\alpha^i|<k\\|\alpha^1|\geq \max\{|\alpha^2|,|\alpha^3|\}}}\hskip-2em \frac{\alpha !}{\alpha^1 !\alpha^2 !\alpha^3 !} \sum_{\sigma\in S_3}\left\|(\partial^{\alpha^1}\partial_{1\sigma_1} v)(\partial^{\alpha^2}\partial_{2\sigma_2}v)(\partial^{\alpha^3}\partial_{3\sigma_2}v)\eta^{k-2}\right\|_p \notag\\
&\qquad \leq 6(C_2)^3R^{-(k-7)}\hskip-2em\sum_{\substack{\alpha^1+\alpha^2+\alpha^3=\alpha\\0\leq |\alpha^i|<k\\|\alpha^1|\geq \max\{|\alpha^2|,|\alpha^3|\}}}\hskip-2em \frac{\alpha !}{\alpha^1 !\alpha^2 !\alpha^3 !}  \left(|\alpha^1|+1\right)^{|\alpha^1|-3}\left(|\alpha^2|\right)^{|\alpha^2|-2}\left(|\alpha^3|\right)^{|\alpha^3|-2}\label{eq:ii}
\end{align}
Since for $\alpha$ with $|\alpha|=k$ fixed, we have the following identity (e.g. Proposition 2.1 in \cite{Kato}): 
\begin{equation*}
\sum_{\substack{\alpha^1+\alpha^2+\alpha^3=\alpha\\0\leq \alpha^i\leq\alpha}}\frac{\alpha !}{\alpha^1 !\alpha^2 !\alpha^3 !}=\sum_{\substack{\ell_1+\ell_2+\ell_3=k\\0\leq \ell_i\leq k}}\frac{k!}{\ell_1!\ell_2!\ell_3!} ,
\end{equation*}
then if we let $|\alpha^1|=\ell_1$, $|\alpha^2|=\ell_2$, $|\alpha^3|=\ell_3$, the RHS of \eqref{eq:ii} is bounded by 
$$6(C_2)^3R^{-(k-7)}\sum_{\substack{\ell_1+\ell_2+\ell_3=k\\0\leq \ell_i<k\\\ell_1\geq\max\{\ell_2,\ell_3\}}}\frac{k!}{\ell_1 !\ell_2 !\ell_3 !}(\ell_1+1)^{\ell_1-3}\ell_2^{\ell_2-2}\ell_3^{\ell_3-2}.$$
By Stirling's formula and if we still use $C_4$ to denote the universal constant from it, then the above quantity is bounded by
$$6(C_2)^3C_4R^{-(k-7)}k^{k-3}.$$
The arguments for the case $|\alpha^1|\leq\max\{|\alpha^2|,|\alpha^3|\}$ are exactly the same. Hence
$$\|I\|_p\leq 18 (C_2)^3C_4R^{-(k-7)}k^{k-3}.$$

The rest of the proof for (ii) is the same as for (i) and we do not repeat here.
\end{proof}

\section{Appendix}
In the Appendix, we give a short proof for Lemma~\ref{lem:embedding}
(the Sobolev embeddings for $M^{1,p}_0$, see also \cite{Hei}) and
Theorem~\ref{thm:Grushin}  (the $L^p$ estimates for Baouendi-Grushin operator).
Both statements are well known, and available in much greater
generality. Checking that the general results apply, however, requires
familiarity with the theory of subelliptic operators. For completeness
and the convenience of the reader we provide complete proofs.

\begin{proof}[Proof of Lemma~\ref{lem:embedding}]
We first prove (i). Suppose that $u\in C^1_0(\Omega)$ and extend $u$ to be zero outside $\Omega$. For every $\sigma_1\in \R^m$, $\sigma_2\in \R^n$ with $|\sigma_1|, |\sigma_2|\leq 1$,
\begin{equation}\label{eq:embedding}
u(x,t)=-\int_0^\infty \partial_s u( x+ \sigma_1 s, t+\sigma_2\gamma(s)) ds,
\end{equation}
where $\gamma :[0,\infty)\rightarrow \R$ is a $C^1$ function satisfying $\gamma(0)=0$, $\dot{\gamma}(s)=|x+\sigma_1 s|$. Let 
$$f(x,t):=|\nabla_x u(x,t)|+|x||\nabla_t u(x,t)|,$$
then by a direct computation
$$|\partial_s u( x+ \sigma_1 s, t+\sigma_2\gamma(s))|\leq f(x+\sigma_1 s, t+\sigma_2 \gamma (s)).$$
Hence
\begin{equation*}
|u(x,t)|\leq c_n\int_0^\infty \int_{|\sigma_2|\leq 1} f(x+\sigma_1 s, t+\sigma_2 \gamma(s)) d\sigma_2 ds,
\end{equation*}
which by a change of variable $\eta_2=\sigma_2 \gamma(s)$ gives
\begin{equation}\label{eq:embedding1}
 |u(x,t)|\leq c_n\int_0^\infty \frac{1}{|\gamma(s)|^n}\int_{|\eta_2|\leq \gamma(s)} f(x+\sigma_1 s, t+ \eta_2) d\eta_2 ds.
\end{equation}
By Young's inequality for convolution and Fubini, for $1\leq p<q<\infty$, 
\begin{equation*}
\|u(x,\cdot)\|_{q}\leq c_n \int_0^\infty \frac{1}{|\gamma(s)|^{n(1/p-1/q)}}\|f(x+\sigma_1 s, \cdot)\|_{p} ds
\end{equation*}
Observe that for $x\neq 0$,
\begin{align*}
\gamma(s)&=\int_0^s \dot{\gamma}(\tau)d\tau=\int_0^s|x+\sigma_1 \tau| d\tau \\
&\geq \int_0^s \left|x+\frac{-x}{|x|}\tau \right| d\tau \\
&=\int_0^s\left| |x|-\tau\right| d\tau \geq \frac{s^2}{4}, 
\end{align*}
for any $\sigma_1\in \R^m \text{ with } |\sigma_1|=1$, and for $x=0$, $\gamma(s)=s^2/2$. Hence if we let
$$v(x):=\|u(x,\cdot)\|_q, \quad g(x):=\|f(x,\cdot)\|_p$$
after integrating over $\sigma_1 \in S^{m-1}$, by Fubini and a change of variable we have
\begin{align}
v(x)&\leq c_n\int_0^\infty \frac{1}{s^{2n(1/p-1/q)}}g(x+\sigma_1 s) ds\label{eq:embedding2}\\
&= c_nc_m\int_0^\infty \int_{|\sigma_1|=1}\frac{1}{s^{2n(1/p-1/q)}}g(x+\sigma_1 s) d\sigma ds\notag\\
&=c_{n,m}\int_{\R^m} \frac{1}{|\eta_1|^{2n(1/p-1/q)+(m-1)}}g(x+\eta_1)d\eta_1\notag
\end{align}
By the Hardy-Littlewood-Sobolev inequality, for the chosen $q$ satisfying 
$$\frac{1}{q}+\frac{1}{m+2n}=\frac{1}{p}$$
we have
$$\|v\|_q\leq C_{n,m,p}\|g\|_p.$$

Next we prove (ii). Since this property is local in nature, given $(x,t)\in \Omega$ we may assume by multiplying a characteristic function that $u=0$ in $\Omega\setminus B_1(x,t)$. Hence the integration in \eqref{eq:embedding} can be written from $0$ to $M$ for some $M>0$ large enough.

Applying H\"older's inequality to \eqref{eq:embedding1} we have
\begin{equation*}
|u(x,t)|\leq c_n \int_0^M \frac{1}{|\gamma(s)|^{n/p}}\|f(x+\sigma_1 s, \cdot)\|_p ds.
\end{equation*}
Then arguing similarly as in \eqref{eq:embedding2} we have
\begin{equation*}
|u(x,t)|\leq c_{n,m}\int_{\{\eta_1\in \R^m: |\eta_1|\leq M\}}\frac{1}{|\eta_1|^{2n/p + (m-1)}}g(x+\eta_1) d\eta_1.
\end{equation*}
Then (ii) follows from H\"older's inequality because $p>m+2n$. 
\end{proof}

Next we give a short proof for the $L^p$ estimate for Baouendi-Grushin operator $\mathcal{L}_0$. The idea is to treat $\mathcal{L}_0$ as the projected operator of sub-Laplacian on the Heisenberg-Reiter type group onto certain quotient space. The transference method in \cite{Coi} links the $L^p$ estimate on the group to the $L^p$ estimate for $\mathcal{L}_0$ in the quotient space.

\begin{proof}[Proof of Theorem~\ref{thm:Grushin}]
We give a proof for $m=2$. The proof for general $m$ is the same. We extend $u$ to all $\R^{m+n}$ by setting $u=0$ in $\R^{m+n}\setminus \Omega$. 

Consider $\R^2\times \R^{n\times 2}\times \R^n$ equipped with the group law
\begin{align*}
(x,y,t)\circ (\xi, \eta,\tau)=\left(x+\xi, y+\eta, t+\tau+\frac{1}{2}(y\xi-\eta x)\right)
\end{align*}
where $\R^{n\times 2}$ is the space of $n\times 2$ real matrices, and $y\xi, \eta x$ are understood as the matrix multiplication. 
Let $G:=(\R^2\times \R^{n\times 2}\times \R^n,\circ)$. $G$ is an example of Heisenberg-Reiter group, which is by \cite{Reiter} a nilpotent Lie group of step $2$, with dilation $\delta_\lambda(x,y,t)=(\lambda x, \lambda y, \lambda^2 t)$ and homogeneous dimension $Q=2+2n+2n=2+4n$. It is also immediate that the Lebesgue measure $dxdydt$ is a left and right Haar measure on $G$. 

Now we recall several facts about the nilpotent Lie groups with dilations. 
A direct computation shows that the horizontal vector fields in the Lie algebra $\mathfrak{g}$ that agree at the origin with $\partial_{x_i}$ and $\partial_{y_{i,j}}$ are 
\begin{align*}
X_j&=\partial_{x_j}+\frac{1}{2}\sum_{s=1}^n y_{sj}\partial_{t_s}, \quad j=1,2\\
Y_{i,j}&=\partial_{y_{i,j}}-\frac{1}{2}x_j\partial_{t_i}, \quad i=1, \ldots, n;\ j=1,2.
\end{align*}
Consider the sub-Laplacian in $G$
$$\Delta_{\mathcal{H}}:=\sum_{j=1,2}X_j^2+\sum_{\substack{i=1,\ldots, n;\\ j=1,2}} Y_{i,j}^2.$$
It is easy to check that $X_i$, $Y_{i,j}$ are H\"ormander vector fields, then $\Delta_{\mathcal{H}}$ is hypoelliptic. By Theorem 2.1 in \cite{Folland}, there exists a unique fundamental solution of type $2$ (i.e.\ smooth away from $0$ and homogeneous of degree $2-Q$) for $\Delta_{\mathcal{H}}$, which we denote by $\Psi$. For $u\in C^\infty_0(G)$ we have $u=(\Delta_{\mathcal{H}}u)\ast \Psi$. Since $X_j$ and $Y_{i,j}$ are left-invariant, then 
\begin{equation}\label{eq:kap1}
P(X, Y) u= (\Delta_{\mathcal{H}} u)\ast P(X, Y) \Psi,
\end{equation}
where $P$ is a quadratic polynomial in $X_j$ and $Y_{i,j}$. By $L_p$ estimates for the singular integral in homogeneous groups (see for example XIII 5.3 in \cite{Stein}), we have 
\begin{equation}\label{eq:HLpestimate2}
\|w\ast P(X,Y)\Psi\|_p\leq C_p\|w\|_p, \quad \forall w\in L^p(G),
\end{equation}
and in particular, using \eqref{eq:kap1} we have 
\begin{equation}\label{eq:HLpestimate}
\|P(X, Y)u\|_p\leq C_p\|\Delta_{\mathcal{H}} u\|_p.
\end{equation}

Now let $H:= \{(0,y,0): y\in \R^{n\times 2}\}$. One can easily verify that $H$ is a closed subgroup of $G$ with a bi-invariant measure the Lebesgue measure $dh=dy=dy_{1,1}\ldots dy_{n,2}$. 
Let $H\backslash G=\{Hg: g\in G\}$ be the quotient space and $\pi: g\mapsto Hg$ the natural quotient mapping. Given $f\in C^\infty_0(G)$, define
\begin{equation*}
\overline{f}(Hg)=\int_H f(h \circ g) dh
\end{equation*}
By Theorem 15.21 in \cite{HR}, the above correspondence $f\mapsto \overline{f}$ defines a linear mapping of $C^\infty_0(G)$ onto $C^\infty_0(H\backslash G)$. 

We identify $H\backslash G$ with $\R^2\times \{0\}\times \R^{n}$. A direct computation gives
\begin{equation}\label{eq:pi}
\pi\left((x,y,t+\frac{1}{2}yx)\right)=\pi\left(  (0,y,0)\circ (x,0,t)\right)=(x,0,t).
\end{equation}
Since the Lebesgue measure $dg=dxdydt$ and $dh=dy$ are bi-invariant on $G$ and $H$ correspondingly, then by Theorem 15.24 in \cite{HR}, there exists a unique right-invariant measure (up to a constant) on $H\backslash G\cong \R^2\times\{0\}\times\R^n$, which is necessarily the Lebesgue measure $dxdt$.

Consider the following vector fields, which are the `push-down' vector fields of $X_j$, $Y_{i,j}$ on $H\backslash G\cong \R^2\times \{0\}\times \R^n$
$$X^b_1=\partial_{x_1},\quad X^b_2=\partial_{x_2},$$
$$Y^b_{i,j}=-x_j \partial_{t_i},\quad i=1,\ldots, n;\ j=1,2.$$
Note that the Baouendi-Grushin operator can be written as 
$$\mathcal{L}_0=\sum_{j=1,2}(X^b_j)^2+\sum_{\substack{i=1,\ldots, n;\\ j=1,2}} (Y^b_{i,j})^2.$$
To see that they are indeed the push-down vector fields on $H\backslash G$, we notice that for $u\in C^\infty_0(G)$, 
\begin{align*}
\int_{\R^{2n}}X_1 u ((0,y,0)\circ (x,0,t))dy&=\int_{\R^{2n}}\frac{d}{d\tau}u((0,y,0)\circ(x,0,t)\circ(\tau,0,0))|_{\tau=0}dy \\
&=\frac{d}{d\tau}\int_{\R^{2n}}u((0,y,0)\circ(x,0,t)\circ(\tau,0,0)) dy\ |_{\tau=0}\\
&=\frac{d}{d\tau}\overline{u}(x+\tau,0,t)|_{\tau=0}\\
&=X_1^b \overline u((x,0,t)).
\end{align*}
Similarly, using \eqref{eq:pi} one can check that 
\begin{align*}
X_j^b \overline{u}(Hg)&=\overline{X_j u}(Hg),\\
Y_{i,j}^b \overline{u}(Hg)&=\overline{Y_{i,j} u} (Hg),\quad i=1\cdots, n;\ j=1,2,
\end{align*}
hence 
\begin{align}
\mathcal{L}_0 \overline{u}(Hg)&=\overline{\Delta_{\mathcal{H}} u}(Hg),\label{eq:quo-gru}\\
P(X^b,Y^b) \overline{u}(Hg)&=\overline{P(X,Y) u}(Hg).\label{eq:quo-gruP}
\end{align}

Let $R$ be the representation of $G$ acting on $L^p(H\backslash G)$ given by the right translation, i.e.\ given $\overline{f}\in L^p(H\backslash G)$
\begin{equation}\label{eq:repR}
R_{g} (\overline{f})(Hg')=\overline{f}(H g'g).
\end{equation}
It is easy to check that $\{R_g\}$ are bounded from $L^p(H\backslash G)$ to $L^p(H\backslash G)$ by $1$. 
Given $\overline{u}\in C^\infty_0(H\backslash G)$, $1<p<\infty$, we consider the following linear map $T$ on $L^p(H\backslash G)$: 
\begin{equation}\label{eq:defineT}
T \overline{w}:=\int_{G} (R_{g^{-1}}\overline{w})\left (P(X,Y)\Psi(g)\right) dg, \quad \overline{w}\in L^p(H\backslash G).
\end{equation}
Notice that $G$ is locally compact and by \eqref{eq:HLpestimate2} the convolution operator $w\mapsto w\ast P(X,Y)\Psi$ is bounded in $L^p(G)$. Thus the transference method (Theorem 2.4 in \cite{Coi}) applies and we have for $\overline{w}\in C^\infty_0(H\backslash G)$,
\begin{equation}\label{eq:quo-gru2}
\|T\overline{w}\|_p\leq C_p\|\overline{w}\|_p, \quad 1<p<\infty,
\end{equation}
in particular if we take $\overline{w}=\mathcal{L}_0\overline{u}$, then
\begin{equation}\label{eq:quo-gru3}
\|T\mathcal{L}_0\overline{u}\|_p\leq C_p\|\mathcal{L}_0\overline{u}\|_p.
\end{equation}

In the end we show that $T\mathcal{L}_0\overline{u}=P(X^b,Y^b)\overline{u}$. Indeed, by \eqref{eq:repR} and \eqref{eq:defineT} 
\begin{align*}
T\mathcal{L}_0\overline{u}(Hg_0)&=\int_G (\mathcal{L}_0\overline{u})(Hg_0g^{-1}) (P(X,Y)\Psi(g))dg
\end{align*}
By \eqref{eq:quo-gru} and Fubini, then using \eqref{eq:kap1} and \eqref{eq:quo-gruP} we have
\begin{align*}
T\mathcal{L}_0\overline{u}(Hg_0)&=\int_G \left(\int_H \Delta_{\mathcal{H}} u( h\circ g_0\circ g^{-1})dh \right)P(X,Y)\Psi(g)dg \\
&=\int_H \left(\int_G \Delta_{\mathcal{H}} u( h\circ g_0\circ g^{-1}) P(X,Y)\Psi(g)dg\right)dh\\
&=\int_H \left( \Delta_{\mathcal{H}}u\ast P(X,Y)\Psi\right)(h\circ g_0) dh\\
&=\int_{H} P(X,Y) u (h\circ g_0)dh\\
&=P(X^b, Y^b)\overline{u}( Hg_0).
\end{align*}
Observing that $\partial^2_{x_ix_j}$, $|x|^2\partial^2_{t_it_j}$,
$x_k\partial^2_{x_it_j}$, and $\partial_{t_j}$ can be written as
$P(X^b, Y^b)$ for some quadratic polynomial $P$, we complete the proof
of the theorem. 
\end{proof}


\begin{bibdiv}
\begin{biblist}

\bib{Alm}{book}{
   author={Almgren, Frederick J., Jr.},
   title={Almgren's big regularity paper},
   series={World Scientific Monograph Series in Mathematics},
   volume={1},
   note={$Q$-valued functions minimizing Dirichlet's integral and the
   regularity of area-minimizing rectifiable currents up to codimension 2;
   With a preface by Jean E.\ Taylor and Vladimir Scheffer},
   publisher={World Scientific Publishing Co. Inc.},
   place={River Edge, NJ},
   date={2000},
   pages={xvi+955},
   isbn={981-02-4108-9},
   review={\MR{1777737 (2003d:49001)}},
}

\bib{AC}{article}{
   author={Athanasopoulos, I.},
   author={Caffarelli, L. A.},
   title={Optimal regularity of lower dimensional obstacle problems},
   language={English, with English and Russian summaries},
   journal={Zap. Nauchn. Sem. S.-Peterburg. Otdel. Mat. Inst. Steklov.
   (POMI)},
   volume={310},
   date={2004},
   number={Kraev. Zadachi Mat. Fiz.\ i Smezh. Vopr. Teor. Funkts. 35
   [34]},
   pages={49--66, 226},
   issn={0373-2703},
   translation={
      journal={J. Math. Sci. (N. Y.)},
      volume={132},
      date={2006},
      number={3},
      pages={274--284},
      issn={1072-3374},
   },
   review={\MR{2120184 (2006i:35053)}},
   doi={10.1007/s10958-005-0496-1},
}

\bib{ACS}{article}{
   author={Athanasopoulos, I.},
   author={Caffarelli, L. A.},
   author={Salsa, S.},
   title={The structure of the free boundary for lower dimensional obstacle
   problems},
   journal={Amer. J. Math.},
   volume={130},
   date={2008},
   number={2},
   pages={485--498},
   issn={0002-9327},
   review={\MR{2405165 (2009g:35345)}},
   doi={10.1353/ajm.2008.0016},
}


\bib{CSS}{article}{
   author={Caffarelli, L. A.},
   author={Salsa, S.},
   author={Silvestre, L.},
   title={Regularity estimates for the solution and the free boundary of the
   obstacle problem for the fractional Laplacian},
   journal={Invent. Math.},
   volume={171},
   date={2008},
   number={2},
   pages={425--461},
   issn={0020-9910},
   review={\MR{2367025 (2009g:35347)}},
   doi={10.1007/s00222-007-0086-6},
}


\bib{CRS}{article}{
   author={Caffarelli, Luis A.},
   author={Roquejoffre, Jean-Michel},
   author={Sire, Yannick},
   title={Variational problems for free boundaries for the fractional
   Laplacian},
   journal={J. Eur. Math. Soc. (JEMS)},
   volume={12},
   date={2010},
   number={5},
   pages={1151--1179},
   issn={1435-9855},
   review={\MR{2677613 (2011f:49024)}},
   doi={10.4171/JEMS/226},
}

\bib{CS}{article}{
   author={Caffarelli, Luis},
   author={Silvestre, Luis},
   title={An extension problem related to the fractional Laplacian},
   journal={Comm. Partial Differential Equations},
   volume={32},
   date={2007},
   number={7-9},
   pages={1245--1260},
   issn={0360-5302},
   review={\MR{2354493 (2009k:35096)}},
   doi={10.1080/03605300600987306},
}




\bib{Coi}{article} {
    AUTHOR = {Coifman, Ronald R.},
    AUTHOR = {Weiss, Guido},
     TITLE = {Transference methods in analysis},
      NOTE = {Conference Board of the Mathematical Sciences Regional
              Conference Series in Mathematics, No. 31},
 PUBLISHER = {American Mathematical Society},
   ADDRESS = {Providence, R.I.},
      YEAR = {1976},
     PAGES = {ii+59},
      ISBN = {0-8218-1681-0},
   MRCLASS = {43A22 (28A65 42-02 43A85)},
  MRNUMBER = {0481928 (58 \#2019)},
MRREVIEWER = {Alberto Torchinsky},
}

\bib{DS1}{article}{
   author={De Silva, Daniela},
   author={Savin, Ovidiu},
   title={$C^{2,\alpha}$ regularity of flat free boundaries for the thin
   one-phase problem},
   journal={J. Differential Equations},
   volume={253},
   date={2012},
   number={8},
   pages={2420--2459},
   issn={0022-0396},
   review={\MR{2950457}},
   doi={10.1016/j.jde.2012.06.021},
}

\bib{DS2}{article} {
    AUTHOR = {De Silva, Daniela},
    AUTHOR = {Savin, Ovidiu},
     TITLE = {Regularity of Lipschitz free boundaries for the thin one-phase problem},
      STATUS = {preprint},
       YEAR = {2012},
      EPRINT ={arXiv:1205.1755},
}

\bib{DS3}{article} {
    AUTHOR = {De Silva, Daniela},
    AUTHOR = {Savin, Ovidiu},
     TITLE = {$C^\infty$ regularity of certain thin free boundaries},
      STATUS = {preprint},
       YEAR = {2014},
      EPRINT ={arXiv:1402.1098},
}

\bib{DS4}{article}{
  author={De Silva, Daniela},
  author={Savin, Ovidiu},
  title={A note on higher regularity boundary Harnack inequality
},
  date={2014},
  status={preprint},
  eprint={arXiv:1403.2588}
}

\bib{DS5}{article}{
  author={De Silva, Daniela},
  author={Savin, Ovidiu},
  title={Boundary Harnack estimates in slit domains and applications to thin free boundary problems},
  date={2014},
  status={preprint},
  eprint={arXiv:1406.6039}
}



\bib{Folland}{article}{
    AUTHOR = {Folland, G. B.},
     TITLE = {Subelliptic estimates and function spaces on nilpotent {L}ie
              groups},
   JOURNAL = {Ark. Mat.},
  FJOURNAL = {Arkiv f\"or Matematik},
    VOLUME = {13},
      YEAR = {1975},
    NUMBER = {2},
     PAGES = {161--207},
      ISSN = {0004-2080},
   MRCLASS = {58G05 (35H05 44A25)},
  MRNUMBER = {0494315 (58 \#13215)},
MRREVIEWER = {A. S. Dynin},
}

\bib{GP}{article}{
   author={Garofalo, Nicola},
   author={Petrosyan, Arshak},
   title={Some new monotonicity formulas and the singular set in the lower
   dimensional obstacle problem},
   journal={Invent. Math.},
   volume={177},
   date={2009},
   number={2},
   pages={415--461},
   issn={0020-9910},
   review={\MR{2511747 (2010m:35574)}},
   doi={10.1007/s00222-009-0188-4},
}

\bib{GT}{book} {
    AUTHOR = {Gilbarg, David},
    author={Trudinger, Neil S.},
     TITLE = {Elliptic partial differential equations of second order},
    SERIES = {Classics in Mathematics},
      NOTE = {Reprint of the 1998 edition},
 PUBLISHER = {Springer-Verlag},
   ADDRESS = {Berlin},
      YEAR = {2001},
     PAGES = {xiv+517},
      ISBN = {3-540-41160-7},
   MRCLASS = {35-02 (35Jxx)},
  MRNUMBER = {1814364 (2001k:35004)},
}

\bib{Hei}{book}{
  author={Heinonen, Juha},
  title={Lectures on analysis on metric spaces},
  series={Universitext},
  publisher={Springer-Verlag},
  place={New York},
  date={2001},
  pages={x+140},
  isbn={0-387-95104-0},
  review={\MR{1800917 (2002c:30028)}},
  doi={10.1007/978-1-4613-0131-8},
}

\bib{HR}{book}{
    AUTHOR = {Hewitt, Edwin},
    AUTHOR = {Ross, Kenneth A.},
     TITLE = {Abstract harmonic analysis. {V}ol. {I}},
    SERIES = {Grundlehren der Mathematischen Wissenschaften [Fundamental
              Principles of Mathematical Sciences]},
    VOLUME = {115},
   EDITION = {Second},
      NOTE = {Structure of topological groups, integration theory, group
              representations},
 PUBLISHER = {Springer-Verlag},
   ADDRESS = {Berlin},
      YEAR = {1979},
     PAGES = {ix+519},
      ISBN = {3-540-09434-2},
   MRCLASS = {43-02 (22-01)},
  MRNUMBER = {551496 (81k:43001)},
}

\bib{Kato}{article} {
    AUTHOR = {Kato, Keiichi},
     TITLE = {New idea for proof of analyticity of solutions to analytic
              nonlinear elliptic equations},
   JOURNAL = {SUT J. Math.},
  FJOURNAL = {SUT Journal of Mathematics},
    VOLUME = {32},
      YEAR = {1996},
    NUMBER = {2},
     PAGES = {157--161},
      ISSN = {0916-5746},
   MRCLASS = {35J60 (35B65)},
  MRNUMBER = {1431263 (97k:35064)},
MRREVIEWER = {Giuseppe Di Fazio},
}

\bib{KN}{article}{
   author={Kinderlehrer, D.},
   author={Nirenberg, L.},
   title={Regularity in free boundary problems},
   journal={Ann. Scuola Norm. Sup. Pisa Cl. Sci. (4)},
   volume={4},
   date={1977},
   number={2},
   pages={373--391},
   review={\MR{0440187 (55 \#13066)}},
}

\bib{Le}{article}{
   author={Lewy, Hans},
   title={On the coincidence set in variational inequalities},
   note={Collection of articles dedicated to S. S. Chern and D. C. Spencer
   on their sixtieth birthdays},
   journal={J. Differential Geometry},
   volume={6},
   date={1972},
   pages={497--501},
   issn={0022-040X},
   review={\MR{0320343 (47 \#8882)}},
}


\bib{LN}{article}{
   author={Li, YanYan},
   author={Nirenberg, Louis},
   title={On the Hopf lemma},
	 date={2007},
   eprint = {arXiv:0709.3531},
}

\bib{AHN}{book}{
   author={Petrosyan, Arshak},
   author={Shahgholian, Henrik},
   author={Uraltseva, Nina},
   title={Regularity of free boundaries in obstacle-type problems},
   series={Graduate Studies in Mathematics},
   volume={136},
   publisher={American Mathematical Society},
   place={Providence, RI},
   date={2012},
   pages={x+221},
   isbn={978-0-8218-8794-3},
   review={\MR{2962060}},
}

\bib{Reiter}{article}{
    AUTHOR = {Reiter, Hans},
     TITLE = {\"{U}ber den {S}atz von {W}iener und lokalkompakte {G}ruppen},
   JOURNAL = {Comment. Math. Helv.},
  FJOURNAL = {Commentarii Mathematici Helvetici},
    VOLUME = {49},
      YEAR = {1974},
     PAGES = {333--364},
      ISSN = {0010-2571},
   MRCLASS = {43A20},
  MRNUMBER = {0355478 (50 \#7952)},
MRREVIEWER = {P. Muller-Roemer},
}

\bib{SCA}{article}{
    AUTHOR = {S{\'a}nchez-Calle, Antonio},
     TITLE = {{$L^p$} estimates for degenerate elliptic equations},
   JOURNAL = {Rev. Mat. Iberoamericana},
  FJOURNAL = {Revista Matem\'atica Iberoamericana},
    VOLUME = {4},
      YEAR = {1988},
    NUMBER = {1},
     PAGES = {177--185},
      ISSN = {0213-2230},
   MRCLASS = {35J70 (35S05 47G05 58G25)},
  MRNUMBER = {1009124 (90h:35105)},
MRREVIEWER = {Bruno Franchi},
       DOI = {10.4171/RMI/68},
       URL = {http://dx.doi.org/10.4171/RMI/68},
}


\bib{Stein}{book}{
    AUTHOR = {Stein, Elias M.},
     TITLE = {Harmonic analysis: real-variable methods, orthogonality, and
              oscillatory integrals},
    SERIES = {Princeton Mathematical Series},
    VOLUME = {43},
      NOTE = {With the assistance of Timothy S. Murphy,
              Monographs in Harmonic Analysis, III},
 PUBLISHER = {Princeton University Press},
   ADDRESS = {Princeton, NJ},
      YEAR = {1993},
     PAGES = {xiv+695},
      ISBN = {0-691-03216-5},
   MRCLASS = {42-02 (35Sxx 43-02 47G30)},
  MRNUMBER = {1232192 (95c:42002)},
MRREVIEWER = {Michael Cowling},
}

\bib{Xu}{article} {
    AUTHOR = {Xu, Chao Jiang},
     TITLE = {Subelliptic variational problems},
   JOURNAL = {Bull. Soc. Math. France},
  FJOURNAL = {Bulletin de la Soci\'et\'e Math\'ematique de France},
    VOLUME = {118},
      YEAR = {1990},
    NUMBER = {2},
     PAGES = {147--169},
      ISSN = {0037-9484},
     CODEN = {BSMFAA},
   MRCLASS = {49J10 (35H05 35J50 49N60)},
  MRNUMBER = {1087376 (92b:49008)},
MRREVIEWER = {C. S. Sharma},
       URL = {http://www.numdam.org/item?id=BSMF_1990__118_2_147_0},
}


\end{biblist}
\end{bibdiv}
\end{document}